\documentclass[a4paper,11pt]{amsart}
\usepackage[utf8]{inputenc} 
\usepackage[T1]{fontenc}    
\usepackage{lmodern}        
\usepackage{amsmath, amssymb, amsthm, mathtools, mathrsfs, esint, MnSymbol}
\usepackage{graphicx}
\usepackage{tikz}
\usetikzlibrary{arrows} 
\usepackage[a4paper,margin=3cm]{geometry}
\usepackage{color}
\usepackage{enumitem}
\usepackage{tcolorbox}
\usepackage{comment}
\usepackage{import}
\usepackage{ae}
\usepackage{a4wide} 
\usepackage[colorlinks=true, linkcolor=red, urlcolor=red, citecolor=red]{hyperref}
\mathtoolsset{showonlyrefs} 
\numberwithin{equation}{section}
\numberwithin{figure}{section}

\DeclareUnicodeCharacter{00A0}{ }

\usepackage[utf8]{inputenc} 
\usepackage{graphicx}
\usepackage{amssymb,amsmath,amsthm,mathtools}
\usetikzlibrary{arrows}
 
\usepackage[utf8]{inputenc}
\DeclareUnicodeCharacter{00A0}{ }
\usepackage{amssymb}
\usepackage{MnSymbol}
\usepackage{enumitem}

\usepackage[normalem]{ulem}

\usepackage{tcolorbox}

\usepackage{ae}
\usepackage[utf8]{inputenc}
\numberwithin{figure}{section}

\usepackage{mathrsfs,a4wide}
\usepackage{esint}
\usepackage{color}

\usepackage{comment}

\usepackage[colorlinks=true, linkcolor=red, urlcolor=red,citecolor=red]{hyperref}

\mathtoolsset{showonlyrefs} 
\usepackage{import}

\usepackage[leqno]{amsmath}

\numberwithin{figure}{section}

\newtheorem{theorem}{Theorem}[section]

\newtheorem{lemma}[theorem]{Lemma}
\newtheorem{proposition}[theorem]{Proposition}
\newtheorem{corollary}[theorem]{Corollary}
\theoremstyle{definition}
\newtheorem{definition}[theorem]{Definition}

\newtheorem{remark}[theorem]{Remark}
\numberwithin{equation}{section}

\newcommand{\R}{\mathbb{R}}
\newcommand{\N}{\mathbb{N}}

\newcommand{\Ha}{\mathcal{H}}

\newcommand{\beq}{\begin{equation}}
\newcommand{\eeq}{\end{equation}}

\newcommand{\eps}{\varepsilon}

\newcommand{\Div}{\operatorname{div}}
\newcommand{\pa}{\partial}

\newcommand{\medint}{-\kern -,5cm\int}
\newcommand{\medintinrigo}{-\kern -,315cm\int}

\def\Xint#1{\mathchoice
   {\XXint\displaystyle\textstyle{#1}}%
   {\XXint\textstyle\scriptstyle{#1}}%
   {\XXint\scriptstyle\scriptscriptstyle{#1}}%
   {\XXint\scriptscriptstyle\scriptscriptstyle{#1}}%
   \!\int}
\def\XXint#1#2#3{{\setbox0=\hbox{$#1{#2#3}{\int}$}
     \vcenter{\hbox{$#2#3$}}\kern-.5\wd0}}
\def\dashint{\Xint-}

\begin{document}
\title[Volume-preserving anisotropic mean curvature flow in 2D]{A variational approach to the volume-preserving \\ anisotropic mean curvature flow in 2D}

\author[A. Kubin]
{A. kubin}
\address[Andrea Kubin]{
	Department of Mathematics and Statistics,
P.O.\ Box 35 (MaD), FI-40014 University of Jyvaskyla}
\email{andrea.a.kubin@jyu.fi}

\author[D.A. La Manna]
{D. A.  La Manna}
\address[Domenico Angelo La Manna]{
	Università degli Studi di Napoli Federico II, Dipartimento di Matematica ed Applicazioni R. Caccioppoli.
}
\email{domenicoangelo.lamanna@unina.it}

\author[E. Pasqualetto]
{E. Pasqualetto}
\address[Enrico Pasqualetto]{
	Department of Mathematics and Statistics,
P.O.\ Box 35 (MaD), FI-40014 University of Jyvaskyla}
\email{enrico.e.pasqualetto@jyu.fi}

\keywords{Geometric evolutions; variational methods; minimizing movements scheme; mean curvature flow; anisotropic perimeter.}
\subjclass[2020]{53E10; 53E40;
49Q20; 37E35.}

\begin{abstract} 
In this article, we introduce a variational algorithm, in the spirit of the minimizing movements scheme, to model the volume-preserving anisotropic mean curvature flow in 2D.
We show that this algorithm can be used to prove the existence of classical solutions. Moreover, we prove that this algorithm converges to the global solution of the equation.
\end{abstract}

\maketitle
\tableofcontents

\section{Introduction}

In recent years, the study of algorithms that model geometric flows has attracted the interest of the mathematical community.
For example, we mention the pioneering works of Almgren, Taylor, and Wang in \cite{ATW1993}, and Luckhaus and Sturzenhecker in \cite{LS1995}, where they introduced a minimizing movement-type scheme
that approximate the mean curvature flow. Another algorithm used to approximate the mean curvature flow is the thresholding scheme introduced by Merriman, Bence, and Osher in \cite{MBO1992}; see also \cite{EO2015, E1993, LO2016, SE2019}.

This program has also been carried out in the fractional framework: for the implementation via a minimizing movements-type scheme, see \cite{CMP2015} (for the asymptotic analysis of the discrete algorithm under volume preserving, see \cite{DGKK2023}); for the thresholding scheme approach, see \cite{CS2010}. 
We also highlight recent developments on the approximation of the surface diffusion flow. A minimizing movements-type algorithm was proposed in \cite{CaTa1994} to model it. For a rigorous proof of the convergence of this algorithm, see \cite{CFJKsd}. In the case of anisotropic surface diffusion with elasticity, see \cite{Kub}. 

In this work, we introduce a new minimizing movements-type algorithm to model the volume-preserving anisotropic mean curvature flow, which also allows us to provide a new proof of the existence of classical solutions to the equation. We will present the algorithm in detail after a brief introduction to the equation. We recall that a smooth family of subsets of $\R^2$, denoted by $ \{ E_t \}_{ t \in [0,T)} $ for some $T>0$, is said to be a solution to the volume-preserving anisotropic mean curvature flow starting from $E_0$ if it satisfies
\begin{equation}\label{MAINEQsol}
		\left\{
		\begin{aligned}
			& V_t=     -\kappa^\varphi_{E_t}+ \dashint_{\pa E_t} \kappa^\varphi_{E_t}  \text{ on } \pa E_t,\\
			& E_0 \text{ initial datum, } 
		\end{aligned}
		\right.
	\end{equation}   
    where $V_t$ denotes the normal velocity, $k_{E_t}^\varphi(x)$ is the anisotropic mean curvature at the point $x \in \pa E_t$ and $\varphi$ is the given anisotropy that satisfies
    \begin{equation}
     \varphi \in C^\infty \text{ and }\nabla^2 \varphi (\nu) \xi \cdot \xi \geq C \vert \xi \vert^2 \quad \forall \nu \in \mathcal{S}^1,\,\xi \in \R^2 \text{ such that } \nu \bot \xi.
    \end{equation}
We briefly outline the physical and mathematical motivations behind the study of \eqref{MAINEQsol}. The anisotropic mean curvature flow with volume constraint naturally arises in physical processes involving the evolution of interfaces under anisotropic surface tension, subject to mass or volume conservation. This flow describes the motion of an interface whose normal velocity is proportional to its anisotropic mean curvature -- a geometric quantity that measures how the surface bends relative to a given (possibly crystalline) surface energy density. 
The volume constraint reflects incompressibility or mass conservation, as observed in the dynamics of grains in polycrystalline materials, vesicle membranes, or bubbles in anisotropic fluids. Physically, such flows capture coarsening phenomena, equilibrium shapes of crystals, and interface evolution in materials science. Foundational treatments of anisotropic surface energies and motion by anisotropic curvature can be found in the works of Taylor \cite{Taylor1991}, Gurtin \cite{Gurtin1982}, and Cahn–Hoffman \cite{CahnHoffman1972}, among others. 
The volume-preserving nature of the flow models physical constraints in closed systems and is closely related to theories of constrained capillarity and incompressible multi-phase flows \cite{AMcFW1998,EG1994}. If the initial set $E_0$ is sufficiently regular -- e.g., if it satisfies the interior and exterior ball conditions -- then the equation \eqref{MAINEQsol} admits a unique smooth solution for a short time interval \cite{ES1998}. 
A classical result by Huisken \cite{H1987} shows that, for convex initial sets, the solution exists for all times and converges exponentially fast to a sphere. Similarly, it follows from \cite{ES1998, N2021} that if the initial set is sufficiently close to a local minimizer of the isoperimetric problem, then the flow remains smooth and converges exponentially fast (see also \cite{DGDKK} for an analysis in the flat torus). 
However, for generic initial sets, the equation \eqref{MAINEQsol} may develop singularities in finite time \cite{M2001, MS2000}. In contrast to the standard mean curvature flow, \eqref{MAINEQsol} can develop singularities even in the plane, and the boundary may collapse in such a way that the curvature remains uniformly bounded up to the singular time, see \cite{EI2005, M2001, MS2000}.
There is a significative difference between \eqref{MAINEQsol} and the mean curvature flow: the former is nonlocal and does not satisfy the comparison principle. As a result, we cannot directly apply the notion of viscosity solutions to define level-set solutions using the methods introduced by Chen–Giga–Goto \cite{CGG1999} and Evans–Spruck \cite{ES1991}. However, in \cite{KK2020}, Kim and Kwon were able to construct a viscosity solution to \eqref{MAINEQsol} in the case of star-shaped sets. 
An important feature of \eqref{MAINEQsol} is that it can be formally interpreted as the $L^2$-gradient flow of the surface area functional. Since it also preserves volume, it may be regarded as the evolutionary analogue of the isoperimetric problem.
At the core of this idea lies the gradient flow structure of mean curvature flow: trajectories in state space follow the path of steepest descent of the area functional with respect to an $L^2$-metric. This perspective, in fact, inspired De Giorgi \cite{DeG1993} to introduce the theory of minimizing movements for general gradient flows in metric spaces, later developed in \cite{AGS2005}. Given a metric
$d$ and an energy functional $E$, each time step of this abstract scheme corresponds to solving the minimization problem
\begin{equation}
    x_k \in \arg \min \left\{ E(x)+ \frac{1}{2h} d^2(x,x_{k-1})  \right\} ,
\end{equation}
where $h>0$ is the ``length'' of the discrete time step. 
    However, the difficulty lies in the fact that the candidate metrics for computing the gradient flow of the area functional are completely degenerate, in the sense that the induced distance vanishes identically \cite{MM2006}.
A natural question, therefore, is whether one can introduce a minimizing movements-type algorithm to model the volume-preserving anisotropic mean curvature flow. The prototype algorithm can be described as follows:
given an initial set $E_0$ with volume $1$ and a time step $h>0$, the sets $E_k$ for $k \in \N$ are obtained iteratively by minimizing the functionals
\begin{equation}\label{17072025_1}
   E \mapsto P_\varphi(E)+ \frac{1}{2h}d_{L^2}^2(E;E_{k-1}),
\end{equation}
in the class of sets $E$ with volume $1$, and
where we have denoted the anisotropic perimeter by $P_\varphi$ and the candidate approximation for the $L^2$-metric by $d_{L^2}$.

The aim of our work is to introduce a minimizing movements scheme to model the volume-preserving anisotropic mean curvature flow.
To this end, we search for a function $d_{L^2}$ with the following property: given two sets $E,F \subset \R^2$ that are sufficiently close to each other, we have
\begin{equation}\label{L^2introduzione}
    d_{L^2}(F;E)= \left( \int_{\pa E} \vert \xi_{F,E} \vert^2 \, d \mathcal{H}^1 \right)^{\frac{1}{2}}= \sup_{\| f \|_{L^2(\pa E)} \leq 1} \int_{\pa E} f(x)\xi_{F,E}(x) \, d \mathcal{H}^1_x,
\end{equation}
where $\xi_{F,E}(x)$, for $x \in \pa E$,  is an approximation of the normal signed distance between the boundaries $\pa E $ and $\pa F$.  As a second objective, we aim to introduce a flexible algorithm -- that is, one capable of modeling various geometric flows, such as surface diffusion or the Mullins–Sekerka flow. 
In fact, in the definition \eqref{L^2introduzione}, if we replace the space of test functions $ f \in L^2(\pa E)$ with $f \in H^1(\pa E)$ or $ f \in H^{\frac{1}{2}}(\pa E)$, then we expect that our algorithm models the corresponding gradient flow of the area functional in the $H^{-1}$ or $H^{-\frac{1}{2}}$ metric, respectively. 
In this way, one might recover the surface diffusion flow and the Mullins–Sekerka flow. Before presenting our algorithm, we recall the  minimizing movements-type algorithm introduced by \cite{ATW1993, LS1995} for modeling motion by mean curvature.
Instead of modeling the $L^2$ distance directly, they model the one-half squared distance $\frac{1}{2}d^2_{L^2}$, as follows: given $E \subset \R^2$ and $F \subset \R^2$, 
\begin{equation}\label{17072025_2}
    \frac{1}{2}d_{L^2}^2(F;E)= \int_{F \Delta E} \mathrm{dist}(x, \pa E)\, d x,
\end{equation}
where $\mathrm{dist}(x, \pa E)$ denotes the distance of $x$ from the boundary $\pa E$.
The reason it mimics an $L^2$-metric is the following: if $\pa F$ coincides with the normal graph of a function defined on $\pa E$, that is,
\begin{equation}\label{normagraphintro}
    \pa F= \{ x+ \psi(x)\nu_{E}(x) \colon \, x \in \pa E \},
\end{equation}
where $\nu_E$ denotes the outer unit normal vector, then
\begin{equation}
    \frac{1}{2}d_{L^2}^2(F;E)= \int_{F \Delta E} \mathrm{dist}(x, \pa E)\, d x= \frac{1}{2}\int_{\pa E} \psi^2\, d \mathcal{H}^1 + \frac{1}{3} \int_{\pa E} | \psi |^3 \kappa_E \, d \mathcal{H}^1,
\end{equation}
where $\kappa_E(x)$ denotes the curvature at $ x \in \pa E$. In any case, we observe that it is not clear whether \eqref{17072025_2} can be written in the form of \eqref{L^2introduzione}, and thus it can be used to model other geometric flows as well. 
The algorithm introduced by Almgren, Taylor, and Wang, as well as by Luckhaus and Sturzenhecker, has been extensively studied in the mathematical literature. In the isotropic case, that is, when $\varphi$ coincides with the Euclidean norm, the volume constraint in the minimizing movements scheme is replaced by a volume penalization. 
Thus, the functional \eqref{17072025_1} becomes
\begin{equation}
    E \rightarrow P_\varphi(E)+ \frac{1}{2h}d_{L^2}^2(E;E_{k-1})+ \frac{1}{\sqrt{h}} \vert \vert E \vert- 1 \vert.
\end{equation}
For this functional, the existence of flat flows (weak solution) was studied in \cite{MSS2016}, and the consistency of the algorithm was proved in \cite{JN}. For results on the asymptotic behavior in the plane and in the flat torus, see instead \cite{ADGK2025, DGK2023, FJM2022, JMPS2023, JN2023}. In the anisotropic case, the existence of a flat flow (weak solution) and its asymptotic behavior in the plane were studied in \cite{KimKwon2025}.
However,  the problem regarding the consistency of flat flow solutions was not addressed in \cite{KimKwon2025}.
In the case of non-smooth anisotropy, we mention \cite{BCCN2009} where the authors study the volume preserving crystalline mean curvature flow starting from a convex set. In \cite{CKub} the authors compute  explicitly the anisotropic crystalline mean curvature flow with a volume constraint for a rectangle in the plane and on the lattice $\varepsilon\mathbb{Z}^2$ as $\varepsilon \to 0^+$. 

The function we introduce to model the $L^2$-distance is the following (see also \eqref{defdistL2}):
\begin{equation}\label{bernardbolzano}
    d_{L^{2}}(F;E):=\sup_{\| f\|_{L^2(\pa E)}\leq1}\int_{\R^2}f(\pi_{\pa E}(x))(\chi_F(x)-\chi_E(x))\,dx,
\end{equation}
where $\pi_{\pa E}$ denotes the projection onto $\pa E$.
Moreover, as shown in Lemma \ref{lem:l2dist}, under assumption \eqref{normagraphintro}, we have
\begin{equation}
    d_{L^2}^2(F;E)= \int_{\pa E} \vert \psi + \frac{\psi^2}{2}\kappa_E \vert^2 \, d \mathcal{H}^1.
\end{equation}
Recall that if we restrict the analysis to $f \in H^1(\pa E)$ in definition \eqref{bernardbolzano}, we get flat flow solutions to the surface diffusion equation (see
 \cite{CFJKsd}  and also \cite{Kub} for surface diffusion with elasticity). Based on these observations, we believe that if we let $f \in H^{\frac{1}{2}}(\pa E)$ in \eqref{bernardbolzano}, the algorithm would converge to solutions of the Mullins–Sekerka equation.
 
The algorithm we implemented is described in detail in Section \ref{sec:esistenza}, and in particular in Subsection \ref{canebernerdabolzano}, where we introduce a specific notion of flat flow. The main theorems of the paper are Theorem \ref{thmausilmain} and Theorem \ref{03thmfinaleGLOB}.

In Theorem \ref{thmausilmain}, the minimizing movements scheme introduced in Subsection \ref{canebernerdabolzano} is used to prove the existence of classical solutions to equation \eqref{MAINEQsol}, starting from sufficiently regular initial data.

In the last result of the paper (see Theorem \ref{03thmfinaleGLOB}), we establish the consistency of the minimizing movements algorithm -- namely, the convergence of the algorithm to the classical solution of the equation \eqref{MAINEQsol} -- throughout the entire time interval where the solution exists.

\section{Notation and preliminary results}
In this paper, we work in the two-dimensional Euclidean space $\R^2$. We denote by $\cdot$ the standard inner product in $\R^2$ and by $ \vert \cdot \vert $ the corresponding Euclidean norm. For $r>0$ we define the open ball of radius $r>0$ centered at $x \in \R^2$ by  $B_r(x)= \{ y \in \R^2 : \vert x- y \vert < r \}$.
When $x=0$, we simply write $B_r:=B_r(0)$. The unit circle is denoted by $\mathcal{S}^1:= \pa B_1$. Given any set $A\subset \R^2$, we denote by ${\rm cl}( A)$ and $ {\rm int} (A)$ its topological closure and interior, respectively, with respect to the Euclidean topology. The Lebesgue measure of a Borel set $A\subset \R^2$ is denoted by $\vert A \vert$, and $\mathcal{H}^1$ denotes the one-dimensional Hausdorff measure. The Hausdorff distance between sets is denoted by $\mathrm{dist}_{\mathcal{H}}$. Let $\varphi $ be a norm in $\R^2$. We define
\begin{equation}\label{mMvarphi}
m_\varphi := \inf_{\vert \nu \vert=1 } \varphi (\nu)\quad\text{ and }\quad M_\varphi := \sup_{\vert \nu \vert=1} \varphi (\nu).
\end{equation} 

Given $a,b \in \R^2$, we define the linear map $a \otimes b : \R^2 \rightarrow \R^2$ by $ a\otimes b (x):=(x \cdot b)a$. We denote by $\R^{2 \times 2}$ the space of $2 \times 2$ real matrices. Given $P,C \in \R^{2 \times 2}$, we define $P : C:=\sum_{i,j=1}^2 p_{ij}c_{ij}$. Throughout the paper, we use the notation $C(*,\cdots,*)$ to indicate a generic positive constant, which may vary from line to line, depending only on the parameters $*,\cdots,*$.
\subsection{Functional spaces}

Let $E\subset \R^2$ be a bounded set. We say that $E$ is of class $C^k$ if, for every point $x\in \pa E$, there exists an open neighborhood $U$ of $x$ such that $\pa E \cap U$ coincides with the graph of a $C^k$ function. 
If $E$ is an open bounded set of class $C^1$, then the outer unit normal vector to \(\pa E\) is well defined at every point $x\in \pa E$, and we denote it by $\nu_E(x)$. Moreover, at each $x\in \pa E$, we define the tangent vector $\tau_E(x)$  as the clockwise rotation by \(90^\circ\) of the outer unit normal vector $\nu_E(x)$. Let $A \subset \R^2$, and for a given $\delta>0$, we define
$$
I_\delta( A):=\{x\in\R^2:\,\text{dist}(x, A)<\delta\}.
$$	
We denote by $ d_E$ the signed distance function from $ \pa E$, that is,
\begin{equation}\label{distsd}
d_E(x):=
\left\{
\begin{array}{ll}
\mathrm{dist}(x, \pa E)\\
-\mathrm{dist}(x,\pa E) 
\end{array}\quad\begin{array}{ll}
\text{for }x \in \R^2 \setminus E,\\
\text{for }x\in E.
\end{array}\right.
\end{equation}

 The differential of a scalar function $f$ (respectively, a vector field $X$) along $\pa E$ is denoted by $\pa_{\tau} f$ (respectively, $ \pa_{\tau} X$). When clarification is needed, we write $ \pa_{\pa E} f$ (respectively, $ \pa_{\pa E} X$) instead. The tangential gradient on $\pa E$, denoted by $ \nabla_\tau$ (or $\nabla_{\pa E}$), is defined as
\[\nabla_{\tau} f:= \nabla f - (\nabla f \cdot\nu_E) \nu_E=(\pa_\tau f)\tau_E\, .\]
Any vector field $X\in C^1(\pa E, \R^2)$ can be extended to a $C^1$ vector field in an open neighborhood of $\pa E$; with a slight abuse of notation, we continue to denote this extension by $X$.
The tangential divergence of $X$ on $\pa E$ is defined by
\[
\Div_\tau X:= \Div X-  ( \nabla X \nu_E) \cdot \nu_E.
\]
Note that this definition is independent of the chosen extension of $X$. 
We denote the Laplace–Beltrami operator on $\pa E$ by $\pa_{\tau}^2 $. If $E$ is of class $C^k$ with $k \geq 2$, we recall that
the curvature $\kappa_E:\pa E\to\R$ and the second fundamental form $B_E:\pa E\to\R^{2\times2}$ are given by
$$
\kappa_E:=\Div_\tau\nu_E, \qquad B_E:=\kappa_E\,\tau_E\otimes\tau_E.
$$
Finally, for an open bounded set $E \subset \R^2$ of class $C^2$, we define
    \begin{equation}\label{sigmaE0906}
        \sigma_E:= \frac{1}{2\| \kappa_E \|_{L^\infty(\pa E)}  }.
    \end{equation}
Note that the above definition is well posed, since \(\|\kappa_E\|_{L^\infty(\pa E)} >0\) by the boundedness of \(E\).
\subsection{Sets of finite perimeter and anisotropic perimeter}

We briefly recall the definition of set of finite perimeter and introduce some notions that will be used throughout the manuscript. For a detailed discussion, we refer the reader to \cite{AmbFuscPall}.

Let $E \subset \R^2$ be a Borel set. The De Giorgi perimeter of $E$ is defined as  
\begin{equation}
     P(E)=\sup \left\{  \int_{E} \Div \Psi \,dx \colon \, \Psi \in C^1_c(\R^2, \R^2), \, \| \Psi \|_{\infty} \leq 1 \right\}. 
\end{equation}
We say that a Borel set $E \subset \R^2$ 
is a set of finite perimeter if $P(E) < + \infty$.  
If  $E  \subset \R^2$ is a set of finite perimeter, then its reduced boundary $\partial^* E \subset \R^2$ is well defined, along with the measure-theoretic outer unit normal vector field $\nu_{E}: \partial^* E \rightarrow \mathbb{R}^2$, which is a Borel measurable map. Recall that \(P(E)=\Ha^1(\pa^* E)\). Here and in the following, when $E$ is a set of finite perimeter we shall tacitly assume that $E$ denotes a representative such that $\pa E={\rm cl}(\pa^*E)$.
Observe that if an open set $E$ is of class $C^1$,  then the measure-theoretic outer unit normal vector field coincides with the classical outer unit normal vector field.
 \begin{definition}[Regular strictly convex norm]
    Let $\varphi \in C^\infty(\R^2 \setminus \{0 \})$ be a norm. We say that $\varphi$ is strictly convex if there exists a constant $C>0$ such that  \begin{equation}\label{unifell}
     \nabla^2 \varphi (\nu) \xi \cdot \xi \geq C \vert \xi \vert^2 \quad \forall \nu \in \mathcal{S}^1,\,\xi \in \R^2 \text{ such that } \nu \bot \xi.
\end{equation}
\end{definition}
\begin{remark}
    If $ \varphi$ is a regular strictly convex norm, then there exists a constant $J_\varphi>0$ such that   
  \begin{equation}\label{remform}
     \nabla^2 \varphi (\nu) \xi \cdot \xi \geq J_{\varphi} \vert \xi \vert^2 \quad \forall \nu \in \mathrm{cl}(I_\frac{1}{4}(\mathcal{S}^1)),\,\xi \in \R^2 \text{ such that } \nu \bot \xi.
\end{equation}  
\end{remark}
 Let $E$ be a set of finite perimeter, and given any Borel set $B \subset \R^2$, we define the relative $\varphi$-perimeter of $E$ in $B$ as
$$ P_{\varphi}(E;B):= \int_{\pa^* E \cap B} \varphi(\nu_{E} )\, d \mathcal{H}^1 .$$
When $B=\R^2$, we simply write
\begin{equation}\label{eq:defn.phi.per}
 P_\varphi (E):=\int_{\pa^* E} \varphi(\nu_E )\, d\Ha^1.
 \end{equation}
If the set $E$ is of class $C^2$, the anisotropic curvature of $E$ at a point $x\in \pa E$ is well defined and given by
\begin{equation}\label{09062025curv}
\kappa_E^\varphi (x):= \Div_{\tau} (\nabla \varphi(\nu_E(x))). 
\end{equation}
\begin{remark}\label{remarchinoKvar}
    The anisotropic curvature of $E$ can also be expressed as \begin{equation}\label{formvarphicurv}
        \kappa_{E}^\varphi= \Div (\nabla \varphi(\nu_E))= \kappa_E\,(\nabla^2 \varphi )(\nu_E) : (\tau_E \otimes \tau_E) = g(\nu_E)\kappa_E
    \end{equation}
    for some function $g \in C^{\infty}(\R^2 \setminus \{0\})$. 
\end{remark}
	We recall an anisotropic version of the Gauss–Bonnet theorem for curves (see \cite{KimKwon2025} for a proof).
 \begin{lemma}\label{lGaussB}
     Let $\varphi \in C^\infty(\R^2\setminus \{0\})$ be a regular strictly convex norm. There exists a constant $C(\varphi)>0$ such that for every bounded open set $E \subset \R^2 $ of class $C^2$ such that
    \begin{equation}
        \int_{\pa E} \kappa_E^\varphi \varphi (\nu_E  ) \, d \mathcal{H}^1= C(\varphi).
    \end{equation}
\end{lemma}
 
\subsection{Useful formulas for normal graphs}

In this subsection, we examine in detail some properties of sets whose boundaries can be described as ``small'' graphs over the boundary of a reference set. To this end, we recall the definition of a normal graph.

\begin{definition}\label{ngrafico}
   Let $E \subset \R^2 $ be an open bounded set of class $C^2$.
    Let $0 <\sigma \leq \sigma_E$, and let $F \subset \R^2$ be an open bounded set of class $C^1$ such that
    $$ F \Delta E \subset \mathrm{cl}(I_\sigma(\pa E)).$$ We say that $F$ is a normal graph over $\pa E$ if there exists a function $\psi \in C^1(\pa E,[-\sigma,\sigma])$ such that \begin{equation}\label{eq:f.normal.graph}
        \pa F= \{ x + \psi(x)\nu_{E}(x) \colon x \in \pa E\}.
    \end{equation}
\end{definition}
If $F$ is a normal graph over $\pa E$, then most of the relevant quantities can be expressed in terms of the height function, as shown in the following lemma.
\begin{lemma} \label{lem:norm.perimeter}
Let $E$ be an open bounded set of class $C^2$, and let $F$ be a normal graph over $\pa E$, in the sense of Definition \ref{ngrafico}. Then,
 \begin{equation}\label{nu_Fexp}
\nu_F(y)=\frac{-(\nabla_{\pa E}\psi)(\pi_{\pa E}(y))+(1+ \psi(\pi_{\pa E}(y))\kappa_{E}(\pi_{\pa E}(y)))\nu_E(\pi_{\pa E}(y))}{\sqrt{ (1+ \psi(\pi_{\pa E}(y))\kappa_{E}(\pi_{\pa E}(y)))^2+ \vert (\nabla_{\pa E} \psi)(\pi_{\pa E}(y)) \vert^2}},\quad y \in \pa F.
\end{equation}
Moreover, we have
\begin{equation}\label{per01102024}
P_\varphi(F)=\int_{\pa E} \varphi(-\nabla_{\pa E}\psi(x)+(1+ \psi(x)\kappa_{E}(x))\nu_E(x))\,d \mathcal{H}^1_x.
\end{equation}
\end{lemma}
\begin{proof}
The identity \eqref{nu_Fexp} can be found in \cite{FJM2018}.

To prove \eqref{per01102024},
we introduce the map  $\Psi: \pa E \to \pa F$, defined by $ \Psi(x):= x+\psi(x)\nu_E(x)$. Then a straightforward computation gives
\[
\nabla_{\pa E} \Psi = {\rm Id}_{2\times 2}+ \nabla_{\pa E} \psi \otimes \nu_E  + \psi \nabla_{\pa E} \nu_E
= ({\rm Id}_{2\times 2} + \psi B_{ E} )  + \nabla_{\pa E}\psi \otimes \nu_E 
.
\]
Hence, we get
\begin{equation}\label{eq:jacobian}
J_\tau \Psi=\sqrt{ (1+\psi \kappa_E)^2+ |\nabla_{\pa E}\psi|^2}.
\end{equation}
Thus, by the area formula, together with \eqref{eq:jacobian} and \eqref{nu_Fexp}, we obtain:
\begin{equation*}
    \begin{split}
        P_{\varphi}(F)=\int_{\pa F} \varphi(\nu_F)\,d \mathcal{H}^1&= \int_{\Psi(\pa E)} \varphi (\nu_F \circ \Psi \circ \pi_{\pa E} |_{\pa F}) \,d\mathcal{H}^1\\
        &=\int_{\pa E}  \varphi(\nu_F(\Psi))\sqrt{(1+ \psi \kappa_E)^2+ \vert \nabla_{\pa E} \psi \vert^2} \,d \mathcal{H}^1\\
        & = \int_{\pa E} \varphi(-\nabla_{\pa E}\psi(x)+(1+ \psi(x)\kappa_{E}(x))\nu_E(x))\,d \mathcal{H}^1_x,
       \end{split}
    \end{equation*}
where in the last line we used the fact that $\varphi$, being a norm, is $1$-homogeneous.
\end{proof}  
We recall the following technical lemma (for the proof, see \cite[Lemma 2.2]{Kub}).
\begin{lemma}\label{eq:transform.gradient}
    Let $E \subset \R^2$ be an open and bounded set of class $C^2$. Let $F \subset \R^2$ be an open and bounded set of class $C^1$ such that $\pa F $ as in \eqref{eq:f.normal.graph}.
    Let $G \in C^1(\pa F)$. We define $\hat{G} : \pa E \rightarrow \R$ as $ \hat{G}(x)= G(x+ \psi(x)\nu_E(x))$.
    Then 
    \begin{equation}\label{ruffini}
        \int_{\pa F} \vert \pa_{\pa F} G \vert^2 \, d \mathcal{H}^1= \int_{\pa E} \frac{\vert \pa_{\pa E} \hat{G} \vert^2}{\sqrt{ (1+ \psi \kappa_E)^2+ \vert \pa_{\pa E} \psi \vert^2}}\, d \mathcal{H}^1.
    \end{equation}
\end{lemma}

We next state the formula for the anisotropic curvature of a normal graph. For a proof see \cite[Remark 4.8]{Kub}, while in the isotropic setting we refer to \cite{FJM2018} (see also \cite[Lemma 2.5]{CFJKsd}).

\begin{lemma}
    Let $E\subset \R^2$ be a set of class $C^3$, and let
    $F$ be a normal graph over $\pa E$  in the sense of Definition \ref{ngrafico}, with height function $\psi$. Set $\Psi(x):= x+\psi(x)\nu_E(x)$ for $x \in \pa E$.
   Then, the anisotropic curvature of $F$ satisfies
 \begin{equation}\label{interpol1}
        \begin{split}
        \kappa_F^{\varphi}\circ \Psi=&\,g (\nu_F \circ \Psi) \kappa_F \circ \Psi =  - g(\nu_E) \pa_{\tau}^2 \psi + \kappa_E^{\varphi}+ R_0 \quad\text{ on } \pa E,   \\
R_0:=&\,a(\nu_E,\psi\kappa_E,\pa_{\tau} \psi) \pa_{\tau}^2 \psi+ b(\nu_E,\psi \kappa_E,\pa_{\tau}\psi) \pa_{\tau} (\psi \kappa_E)+ c(\nu_E,\psi,\pa_{\tau} \psi,\kappa_E),
\end{split}
    \end{equation}
    where $a,b\in C^{\infty}$ and $c \in C^{\infty}$ are smooth functions such that
\begin{equation}
b(\cdot,0,0)=a(\cdot, 0,0)=c(\cdot,0,0,\cdot)=0,
\end{equation}
while the function \(g\in C^\infty(\R^2\setminus\{0\})\) is as in Remark \ref{remarchinoKvar}.
\end{lemma}

In the sequel, we shall assume that $E_0\subset\R^2$ is a bounded open set of class $C^m$, $m\geq 3$. If $1\leq k\leq m$, $ \alpha \in [0,1]$ and $K,\sigma_0 >0$, we set
\begin{align}\label{22082024pom2}
\mathfrak{C}^{k,\alpha}_{K,\sigma_0}(E_0):=
\big\{E\subset\R^2: \,\, E \text{\, is bounded,}\,\,\pa E&=\{y+\varphi_E(y)\nu_{E_0}(y)\!:\,y\in\pa E_0\},\cr
& \|\varphi_E\|_{L^\infty(\pa E_0)}\leq \sigma_0,\,\|\varphi_E\|_{C^{k,\alpha}(\pa E_0)}\leq K\big\}.
\end{align}
The set $\mathfrak{H}^{k}_{K,\sigma_0}(E_0)$ is defined in the same way by replacing $\|\varphi_E\|_{C^{k,\alpha}(\pa E_0)}$ with  $\|\varphi_E\|_{H^k(\pa E_0)}$. Given $\{ E_n \}_{n \in \N},\, E$ in $\mathfrak{C}_{K,\sigma_0}^{k,\alpha}(E_0)$ (respectively in $\mathfrak{H}^{k,\alpha}_{K,\sigma_0}(E_0)$), we say that $E_n \rightarrow E $ in $\mathfrak{C}_{K,\sigma_0}^{k,\alpha}(E_0)$ (respectively in $\mathfrak{H}_{K,\sigma_0}^{k,\alpha}(E_0)$) provided $\varphi_{E_n}\to\varphi_E$ in $L^{\infty}(\pa E_0)$ and in $C^{k,\alpha}(\pa E_0)$ (respectively in $H^{k}(\pa E_0)$).
Interpolation inequalities are
an essential tool to prove the main result of the present paper. Thus, we explicitly recall the famous Gagliardo--Nirenberg theorem (see for instance \cite[Proposition 6.5]{Mantegazza2002} and \cite[Proposition 4.3]{DFM2023}).
\begin{proposition}[Gagliardo--Nirenberg interpolation inequalities]\label{interptensor}
Let $j,m$ be integers such that $0 \leq j <m $ and $0<r,q \leq +\infty$ and let $E\in\mathfrak{C}^{1}_{M,\sigma}(E_0):=\mathfrak{C}^{1,0}_{M,\sigma}(E_0)$ for some $M>0$ and $\sigma>0$. Then, for every covariant tensor $T=T_{i_1 \dots i_l}$ the following ``uniform'' interpolation inequalities hold
\begin{equation}\label{gn1}
 \Vert\nabla^j_{\tau} T\Vert_{L^{p}{(\pa E)}}\leq\,C\,\big(\Vert 
\nabla^m_{\tau} T\Vert_{L^{r}{(\pa E)}}+\Vert T\Vert_{L^{r}{(\pa E)}}\big)^{\theta}\Vert
 T\Vert_{L^q{(\pa E)}}^{1-\theta}\,,
\end{equation}
with the compatibility condition
\begin{equation}\label{comp1}
\frac{1}{p}=j+\theta \Big (\frac{1}{r}-m\Big )+\frac{1-\theta}{q}\,,
\end{equation}
for all $\theta \in[j/m,1)$ for which $p\in[1,+\infty)$ and where $C$ is a constant depending only on $j$, $m$, $p$, $q$, $r$ and $E_0$, $M$. Moreover, if $f:\partial F\to\R$ is a smooth function with $\int_{\partial E}f\,d\Ha^{1}=0$, the inequality \eqref{gn1} becomes
\begin{equation}\label{gn2}
\Vert{\pa^j_{\tau} f}\Vert_{L^p(\partial E)} \leq C\,\Vert{\pa^m_{\tau} f}\Vert_{L^r(\pa E)}^{\theta} \Vert{f}\Vert_{L^q(\pa E)}^{1-\theta}\,.
\end{equation}
Finally, if $f:\partial F\to\R$ is any smooth function, there holds
\begin{equation}\label{gn3}
\Vert{\pa^j_{\tau} f}\Vert_{L^p(\partial E)} \leq C \Vert{\pa^m_{\tau} f}\Vert_{L^r(\pa E)}^{\theta} \Vert\pa_{\tau} f\Vert_{L^q(\pa E)}^{1-\theta}, 
\end{equation}
with the compatibility condition
\begin{equation}\label{comp2}
\frac{1}{p}=j-1+\theta \Big (\frac{1}{r}-m+1\Big )+\frac{1-\theta}{q}\, ,
\end{equation}
for all $\theta \in[j/m,1)$ for which $p\in[1,+\infty)$  and where the constant $C$ is as above.

By density, these inequalities clearly extend to functions and tensors in the appropriate Sobolev spaces.
\end{proposition}	
If $E \in \mathfrak{C}_{K,\sigma}^{k,\alpha}(E_0)$  for some $k \in \N$, $ \alpha \in [0,1]$ and $K,\sigma>0$, then the classical interpolation inequality in H\"older norms holds, i.e.,
 for $ 0 < \beta < \alpha \leq 1 $ and $0 \leq l \leq m \leq k$ it holds that
\begin{equation}\label{interHOLDER}
    \| f \|_{C^{l,\beta}(\pa E)} \leq C \| f \|_{C^{m,\alpha}(\pa E)}^{\theta} \|   f \|_{C^0(\pa E)}^{1-\theta}, \quad \theta= \frac{l+ \beta}{m+ \alpha},
\end{equation}
where $C$ depends on $K,l,m,\alpha,\beta$. This result follows from the Euclidean case; see for example \cite[Example 1.9]{Lunardi2018}.

\subsection{The \texorpdfstring{$L^2$-``distance''}{L2"distance"}}
To implement a minimizing movement scheme, we first need to introduce a suitable notion of ``distance'' that will allow us to model the $L^2$-gradient flow of the $\varphi$-perimeter.

 Let $E$ be a bounded open set with $C^2$ boundary. For every $x\in I_{\sigma_E}(\pa E)$, the projection of $x$ onto $\partial E$ is well defined:
 \[
 \pi_{\pa E}(x)=x- d_E(x)(\nu_E \circ \pi_{\pa E})(x).
 \] 
Given a set $F$ of finite perimeter, which is sufficiently close to $E$, we define the ``$L^2-$ distance'' between $E$ and $F$ by
\begin{equation}\label{defdistL2}
	d_{L^{2}}(F;E):=\sup_{\| f\|_{L^2(\pa E)}\leq1}\int_{\R^2}f(\pi_{\pa E}(x))(\chi_F(x)-\chi_E(x))\,dx.
\end{equation}

\begin{lemma}\label{lem:l2dist}
    Let $E\subset \R^2$ be a bounded open set of class $C^2$, and let $\sigma \leq \sigma_E$. Let $F\subset \R^2$ be a set of finite perimeter such that $E \Delta F \subset {\rm cl}\big( I_{\sigma}(\pa E) \big)$, and define
    \begin{equation}\label{funzione}
    \xi_{F,E}(x):= \int_{-\sigma}^{\sigma} (\chi_{F}(x+t\nu_E(x))-\chi_E(x+t \nu_E(x))) (1+ t \kappa_E(x))\, d t, \quad x \in \pa E.
    \end{equation}
 Then,
\begin{equation}\label{eq:distance.and.area}
d_{L^2}(F;E)=\| \xi_{F,E}\|_{L^2(\partial E)}\qquad  \text{and } \qquad
|F|-|E|=\int_{\partial E}\xi_{F,E} \, d \mathcal{H}^1.
   \end{equation}
Moreover, if $F$ is a normal graph over $\pa E$ in the sense of Definition~\ref{ngrafico}, then
\begin{equation}\label{L^2distperlafunz}
    d_{L^2}^2(F;E)= \int_{\pa E} \vert \psi(x)+ \frac{\psi^2(x)}{2}\kappa_E(x) \vert^2 d\mathcal{H}^1_x.
        \end{equation}
\end{lemma}
\begin{proof}
By assumption, we have $\pa F  \subset {\rm cl}\big(I_\sigma(\pa E)\big)$.  For any $t\in(-\sigma,\sigma)$, we define the map $\Phi_t:\pa E\to\{d_E=t\}$ as
\begin{equation}\label{phit}
\Phi_t(y):=y+t\nu_E(y) \qquad \forall y\in\pa E.
\end{equation}
We denote by $J_\tau\Phi_t$ the tangential Jacobian of  $ \Phi_t$. A straightforward computation shows that
\begin{equation}\label{jacphit}
J_\tau \Phi_t(y)=1+t\kappa_E(y)\qquad \forall y\in\pa E  \,\text{ and }\, |t|<\sigma.
\end{equation}
Applying the coarea formula, equation \eqref{jacphit} implies that for any function
 $f\in L^2(\pa E)$, we have:
\begin{align}
\nonumber\int_{\R^2}f(\pi_{\partial E}(x))&(\chi_F(x)-\chi_E(x))\,dx
=\int_{-\sigma}^\sigma dt\int_{\{d_E=t\}}f(\pi_{\partial E}(y))(\chi_F(y)-\chi_E(y))\,d\Ha^{1}_y\\
&\label{eq:defn.dist.graph} =\int_{-\sigma}^\sigma dt\int_{\pa E}f(y)(\chi_F(y+t\nu_E(y))-\chi_E(y+t\nu_E(y)))J_\tau \Phi_t(y)\,d\Ha^{1}_y \\
&\nonumber=\int_{\pa E}f(y)\int_{-\sigma}^\sigma(\chi_F(y+t\nu_E(y))-\chi_E(y+t\nu_E(y)))J_\tau \Phi_t(y)\,dt \,d\Ha^{1}_y.
\end{align} 
Therefore, the function $f \in L^2(\pa E)$ that realizes the sup in the definition of $d_{L^2}(F;E)$, when $\xi_{F,E} \neq 0$, is
$$ f= \frac{ \xi_{F,E}}{\| \xi_{F,E} \|_{L^2(\pa E)}}.$$ 
Therefore, the first identity in \eqref{eq:distance.and.area} follows directly from \eqref{eq:defn.dist.graph}, the definition of the function $\xi_{F,E}$ in \eqref{funzione},
and the definition of $ d_{L^2}(F;E)$. 
By choosing $f=1$ in \eqref{eq:defn.dist.graph} we  obtain the second identity in \eqref{eq:distance.and.area}, since in this case \eqref{eq:defn.dist.graph} becomes
\[
\vert F \vert- \vert E \vert= \int_{-\sigma}^\sigma \,dt\int_{\{d_E=t\}}(\chi_F(y+t\nu_E(y))-\chi_E(y+t\nu_E(y))) J_\tau \Phi_t(y)\, d\Ha^{1}_y.
\]
If $ F$ is a normal graph over $\partial E$, then from \eqref{funzione} we obtain:
\begin{equation}\label{funzg_E(x)}
    \xi_{F,E}= \int_{0}^{\psi} 1+ \kappa_E t\, d t= \psi+ \frac{\psi^2}{2}\kappa_E.
\end{equation}
 Hence, combining \eqref{eq:distance.and.area} with \eqref{funzg_E(x)}, we get \eqref{L^2distperlafunz}.
\end{proof}

\subsection{Almost minimizers of the \texorpdfstring{$\varphi$}{phi}-perimeter}
We now recall the definition and some regularity properties of almost minimizers of the $\varphi$-perimeter.
\begin{definition}
    Let $E \subset \R^2$ be a set of finite perimeter. We say that $E$ is an $ (\omega_0,r_0,\beta)$-almost minimizer of the $\varphi$-perimeter if there exist $\omega_0>0,\, r_0>0, \beta >0$ such that for every  $x \in \R^2$, the following holds:
    \begin{equation}
        P_{\varphi}(E) \leq P_{\varphi}(G)+ \omega_0 r^{1+ \beta}\quad\text{ for all } G \text{ such that } E \Delta G \Subset  B_r(x) \text{ with } r \leq r_0.
    \end{equation}
\end{definition}
It is known that if $ E\subset \R^2$ is an $(\omega_0,r_0,\beta)$-almost minimizer of the $\varphi$-perimeter, then $\pa E$ is of class $C^{1,\eta}$ for every $ \eta \in [0 , \frac{\beta}{2})$; see  \cite{AlmgrenSimonSchoen1977}, \cite{Bombieri1982} and \cite{DePMa2015}. For the case where $\varphi$ is the Euclidean norm, see \cite[Theorem~1.9]{Tamanini1984}). 

We will use the following lemma. The proof is similar to those in \cite[Lemma~2.8]{CFJKsd} and \cite[Lemma~4.3]{Kub}, but we include it here for the reader’s convenience. 
\begin{lemma}\label{propdadim}
Let $K>0$ and $\sigma>0$, let $E \in \mathfrak{C}^{2,\alpha}_{K,\sigma}(E_0)$ with $\alpha > 0$, and let $F$ be an $(\omega_0, r_0, \beta)$-minimizer of the $\varphi$-perimeter for some $\beta \in (0,\tfrac{1}{2}]$. Given any $\gamma < \min\{\alpha, \tfrac{\beta}{2}\}$, there exists a constant $\delta_0 = \delta_0(K, \omega_0, \gamma) > 0$ such that, if
$$
F \Delta E \subset \mathrm{cl}(I_{\delta_0}(\pa E)),
$$
then there exists a function $\psi \in C^{1,\gamma}(\partial E)$ such that
\begin{equation}
   \pa F = \{ x + \psi(x)\nu_E(x) : x \in \pa E\}. 
\end{equation}
Moreover, for every $\eps>0$, the quantity
$\delta_0$ can be chosen also so that $\|\psi\|_{C^{1,\gamma}(\pa E)} \leq \eps$.
\end{lemma} 
\begin{proof}
By assumption, $E \Delta F \subset \mathrm{cl} (I_{\delta_0}(\pa E))$, we see that for every $x \in \pa E$ and every $\eta>0$, we have $ B_{\delta_0+\eta}(x) \cap \pa F \neq \emptyset$.  Fix $\varepsilon>0$. We claim that for all $\delta_0 \in (0,\frac{\varepsilon}{100 K})$, if $E$ and $F $ satisfy the assumption above, then
\begin{equation}\label{dist_norm}
    \vert \nu_E(x)-\nu_F(y) \vert \leq \varepsilon \quad \text{for all } y \in \pa F \cap \mathrm{cl}(B_{\delta_0}(x)). 
\end{equation}
We argue by contradiction. Suppose the claim is false. Then there exist $\varepsilon>0$ and sequences $\{E_n\}_{n \in \N},\, \{F_n\}_{n \in \N}  $ such that: $ E_n \in \mathfrak{C}_{K,\sigma}^{2,\alpha}(E_0)$, $F_n$ is an $(\omega_0,r_0,\beta)$-minimizer of the $ \varphi$-perimeter, $ E_n \Delta F_n \subset \mathrm{cl}(I_{\delta_0}(\pa E))$,  there exist $x_n \in \pa E_n$ and $ y_n \in B(x_n,\frac{1}{n}) \cap \pa F_n$ such that 
\begin{equation}\label{contradict}
    \vert \nu_{E_n}(x_n)- \nu_{F_n}(y_n) \vert \geq \varepsilon \text{ for all } n \in \N.
\end{equation}
Without loss of generality, up to extracting a subsequence, we may assume that $x_n,\,y_n \rightarrow x$ as $ n \rightarrow + \infty$, and that $  E_n \rightarrow  E  $ in $\mathfrak{C}^{2,\alpha'}_{K,\sigma}(E_0)$ for every $ \alpha' < \alpha$, while $ F_n \rightarrow F$ in the Hausdorff distance, where $F$ is an $(\omega_0,r_0,\beta)$-minimizer of the $ \varphi$-perimeter. As a result, $ \nu_{E_n}(x_n) \rightarrow \nu_E(x)$ as $n \rightarrow + \infty $. Moreover, using the $ (\omega_0,r_0,\beta)$-minimality of $F_n$, we also have $\nu_{F_n}(y_n) \rightarrow \nu_E(x)$, see \cite{Bombieri1982}, which contradicts \eqref{contradict}. The conclusion of the lemma then follows from \eqref{dist_norm} and a standard regularity argument; see, for instance, \cite[Lemma 4.3]{Kub}.
\end{proof}

\section{First properties of the scheme} \label{sec:esistenza}
\subsection{Setting of the problem}\label{canebernerdabolzano}
In this section, we introduce a variational algorithm to model the volume-preserving anisotropic mean curvature flow.  \\
Fix $h\in(0,1)$ to be a time step discretization. Let $E,F\Subset  \R^2$ be sets of finite perimeter, with $F$ sufficiently close to $E$ in the Hausdorff sense. We consider the energy functional
\begin{equation}\label{pom102092024}
\mathcal{F}_h(F,E):=P_{\varphi}(F)+\frac{1}{2h}d_{L^2}^2(F;E),
\end{equation}
where $P_\varphi$ denotes the $\varphi-$perimeter as defined in \eqref{eq:defn.phi.per}, and $d_{L^2}^2(F;E)$ is defined in \eqref{defdistL2}.
Let $\delta>0$, we define the admissible class
\begin{equation}\label{eq:class.minimizing}
  \mathcal {B}_\delta(E):=   \left\{ F\subset \R^2:\,\, P_\varphi(F)<\infty, \,\,  F \Delta E \subset {\rm cl}( I_{\delta}(\pa E )) \right\}.
\end{equation}
It is straightforward to verify that the functional
\[
\mathcal {B}_\delta(E)\ni F \longmapsto \mathcal {F}_{h}(F,E)
\]
admits a minimizer.
This basic observation allows us to define the
{\it approximate constrained flat flow}, which forms the core of the present work.
\begin{definition}[Approximate constrained flat flow] \label{12092023def1}
		Let $\delta>0$ and $h>0$ be fixed positive numbers. Let $E_0 \Subset   \mathbb{R}^2$ be a set of class $C^4$ such that $\vert E_0 \vert=1$.
        
        Define the sequence of sets $\{ E_{kh}^{h,\delta}\}_{k \in \mathbb{N}}$ iteratively by setting $E_0^{h,\delta}=E_0$, and for $k \geq 1$,
		$$   E_{k h}^{h,\delta} \in \arg \min \left\{\mathcal{F}_{h}(E,\, E_{(k-1)h}^{h,\delta}) \colon \, E \in \mathcal{B}_\delta( E_{(k-1)h}^{h,\delta} ),\, \vert E \vert=1\right\},$$
        where $\mathcal{F}_h$ is defined in \eqref{pom102092024}.
		We define the piecewise constant interpolation in time by 
        \begin{equation}\label{agligderigligb}
            E_{t}^{h,\delta}:= E_{k h}^{h,\delta} \quad \text{for any } t \in [k h,\, (k+1)h).
        \end{equation}
We refer to the family $\{E_t^{h,\delta}\}_{t \geq 0}$ as an approximate constrained flat flow with initial datum $E_0$ and time step $h$. 
\end{definition}
   
The starting point of our analysis is the following lemma, which provides an estimate of the $\varphi$-perimeter of a set  $F$ over $\pa E$ in terms of the height function and the $L^2-$distance. 
\begin{lemma}\label{Lemma1006} Let $K>0$, $\sigma>0$ and $E\in\mathfrak{C}^{2,\alpha}_{K, \sigma}(E_0)$ for some $ \alpha \in (0,\frac{1}{2}]$, and assume that $ \vert E \vert=1$. Then there exist constants $\Lambda$ and $ \delta_1$, depending only on $K$, such that the following holds: if $F \in \mathcal{B}_{\delta_1}(E)$ is such that $$ \pa F= \{ x +\psi(x) \nu_{E}(x) : x \in \pa E\}$$ with $ \psi \in C^{1}(\pa E)$ and $ \| \psi \|_{C^{1}(\pa E)} \leq  \delta_1$,  then
\begin{equation}\label{lemmatesi1}
   \frac{J_\varphi}{2} \|\nabla_{\pa E} \psi \|_{L^2(\pa E)}^2+P_{\varphi}(E) \leq P_{\varphi}(F)+ \Lambda d_{L^2}(F;E),
\end{equation}
where $J_{\varphi}$ is defined in \eqref{remform}.
\end{lemma}
\begin{proof} 
    We observe that 
    \begin{equation}
    \begin{split}
    \nabla_{\pa E} ( \nabla_{\pa E} \varphi (\nu_E))=\nabla^2_{\pa E} \psi (\nu_E) \nabla_{\pa E} \nu_E=\kappa_E  \nabla^2_{\pa E} \varphi (\nu_E)  \tau_E \otimes \tau_E
    \end{split}
    \end{equation}
    and consequently, \begin{equation}\label{26022025f1}
    \begin{split}
        \Div_{\pa E} ( \nabla_{\pa E} \varphi (\nu_E) )= \mathrm{Tr} \nabla_{\pa E} ( \nabla_{\pa E} \varphi (\nu_E)) 
        =\kappa_E \nabla_{\pa E}^2 \varphi(\nu_E )\tau_E\cdot \tau_E.
        \end{split}
    \end{equation}
    By using the convexity of $ \varphi$ and formula \eqref{remform}, for $ \delta_1$ sufficiently small, we obtain:
    \begin{equation}\label{oggi27feb}
    \begin{split}
\varphi((1+\psi\kappa_E)\nu_E-\nabla_{\pa E} \psi ) 
&\geq \varphi(\nu_E+\psi\kappa_E\nu_E)- \nabla \varphi( (1+\psi\kappa_E)\nu_E) \cdot \nabla_{\pa E} \psi  + \frac{J_{\varphi}}{2} \vert  \nabla_{\pa E} \psi \vert^2\\
& =  \varphi(\nu_E+\psi\kappa_E\nu_E)- \nabla_{\pa E} \varphi(\nu_E) \cdot \nabla_{\pa E} \psi + \frac{J_{\varphi}}{2} \vert  \nabla_{\pa E} \psi \vert^2                         \\
& \geq \varphi(\nu_E)+ \psi\kappa_E \nabla \varphi(\nu_E)\cdot \nu_E - \nabla_{\pa E} \varphi( \nu_E) \cdot \nabla_{\pa E} \psi  + \frac{J_{\varphi}}{2} \vert  \nabla_{\pa E} \psi \vert^2\\
& = \varphi(\nu_E)+ \psi \kappa_E \varphi(\nu_E)- \nabla_{\pa E} \varphi( \nu_E) \cdot \nabla_{\pa E} \psi+ \frac{J_{\varphi}}{2} \vert  \nabla_{\pa E} \psi \vert^2
\end{split}
    \end{equation}
   where in the last equality we used the identity $ \nabla \varphi (x) \cdot x = \varphi (x)$.
    Combining formulas \eqref{per01102024}, \eqref{oggi27feb}, and choosing $\delta_1$ small enough, we get:
    \begin{equation}\label{oggi27feb2}
    \begin{split}
    P_{\varphi}(F)&=\int_{\pa E}\varphi(-\nabla_{\pa E} \psi+(1+\psi \kappa_E)\nu_E)\, d \mathcal{H}^1 \\
    &\geq \int_{\pa E} \varphi(\nu_E)+ \psi \kappa_E \varphi(\nu_E)- \nabla_{\pa E} \varphi(\nu_E) \cdot \nabla_{\pa E} \psi\, d \mathcal{H}^1 + \frac{J_{\varphi}}{2} \|\nabla_{\pa E} \psi \|^2_{L^2(\pa E)}\\
    & \geq P_{\varphi}(E)- C(K,\varphi) \| \psi \|_{L^2(\pa E)}+ \frac{J_{\varphi}}{2} \| \nabla_{\pa E} \psi \|^2_{L^2(\pa E)}- \int_{\pa E} \nabla_{\pa E} \varphi(\nu_E) \cdot \nabla_{\pa E} \psi \, d \mathcal{H}^1.
    \end{split}
    \end{equation}
  Applying the divergence theorem, formula \eqref{26022025f1}, and the Cauchy–Schwarz inequality, we estimate the last term of the above equation:
    \begin{equation}
        \begin{split}
            \int_{\pa E} \nabla_{\pa E} \varphi(\nu_E) \cdot \nabla_{\pa E} \psi\, d \mathcal{H}^1  
           & =\int_{\pa E} \Div_{\pa E}(\nabla_{\pa E} \varphi (\nu_E) ) \psi \, d \mathcal{H}^1   \\
            &\leq \| \Div_{\pa E}(\nabla_{\pa E} \varphi (\nu_E) ) \|_{L^2(\pa E)} \| \psi\|_{L^2(\pa E)} \leq C(K) \| \psi \|_{L^2(\pa E)} .
        \end{split}
    \end{equation}
   Substituting this bound into \eqref{oggi27feb2}, we obtain:
    \begin{equation}\label{glablinz}
        P_{\varphi}(F) \geq P_\varphi (E)- C(K,\varphi) \| \psi \|_{L^2(\pa E)}+ \frac{J_{\varphi}}{2} \| \nabla_{\pa E} \psi \|^2_{L^2(\pa E)}.
    \end{equation}
    Moreover, for $\delta_1$ small enough, we have the inequality: \begin{equation}\label{fomr07112024}
      \frac{\psi^2}{2}\leq \big(\psi+ \frac{\psi^2}{2}\kappa_E\big)^2 \leq 2 \psi^2.
    \end{equation} 
    Recalling that $$d_{L^2}^2(F;E)= \int_{\pa E} \vert \psi + \frac{\psi^2}{2}\kappa_E \vert^2 \, d \mathcal{H}^1,$$ and using \eqref{fomr07112024} together with \eqref{glablinz}, we finally deduce:  \begin{equation}\label{perdiseq}
        P_\varphi(F)  \geq P_\varphi(E) - C(K) d_{L^2}(F;E)+ \frac{J_{\varphi}}{2} \|  \nabla_{\tau} \psi \|_{L^2(\pa E)}^2 .
    \end{equation}
\end{proof}
Let $E \subset \R^2$ and $\sigma>0$. For every $t \in (-\sigma,\sigma)$,
we define \begin{equation}\label{110625form1}
E_t:= \{ x \in \R^2 \colon \, d_{E}(x) \leq t \}.
\end{equation}

In the next lemma, we establish a minimality property of the set $E$. Roughly speaking, we show that if $d_{L^2}$ in the definition of $\mathcal F_h$ is not squared, the penalization becomes too strong, rendering the minimization problem trivial.
\begin{lemma}   \label{lem:quasi.min0}
Let $K>0 $, $\sigma>0$ and $E \in \mathfrak{C}_{K,\sigma}^{2,\alpha}(E_0)$ for some $\alpha \in (0,\frac{1}{2}]$ and such that $ \vert E \vert=1$. Let $\mathcal{B}_\delta(E)$ be the class defined in \eqref{eq:class.minimizing}, and define the functional $\mathcal{J}: \mathcal{B}_\delta(E) \rightarrow \R$ by
    \begin{equation}
        \mathcal{J}(A):= P_{\varphi}(A)+ \Lambda' d_{L^2}(A;E)+ \vert \vert A \vert-1 \vert.
    \end{equation}
Then there exist constants $\delta_2= \delta_2(\Lambda,K)$ and $ \Lambda'=\Lambda'(K)$ such that
\[
\{E\}= \underset{A \in \mathcal B_{\delta_2}(E) }{\rm argmin \,} \mathcal{J}(A).
\]
\end{lemma}
\begin{proof}
For every $\delta_2>0$, by applying the direct method in the calculus of variations, one can immediately show the existence of a set  $A_m\in \mathcal B_{\delta}(E)$  that minimizes the functional $ \mathcal{J}$.
     Our goal is to prove that $A_m=E$. We proceed in three steps.
\\
\textit{Step 1:} There exist constants $\Lambda'= \Lambda'(K)$ and $\delta'= \delta'(K)$ such that 
\begin{equation}
\{E\}= {\rm argmin \,} \left\{ \mathcal J(E_t) \colon \, t \in [-\delta',\delta'] \right\},
\end{equation}
     where $E_t $ is defined in \eqref{110625form1}.

Using formulas \eqref{per01102024} and \eqref{eq:distance.and.area}, we obtain
\begin{equation}
    \mathcal J(E_t)= \int_{\pa E} \varphi(\nu_E)(1+ t \kappa_E)\, d \mathcal{H}^1+ \Lambda' \vert t \vert \left( \int_{\pa E} \left|1+ \frac{t}{2}\kappa_E \right|^2 \, d\mathcal{H}^1 \right)^{\frac{1}{2}}+ \vert t\vert \left| \int_{\pa E}1+ \frac{t}{2}\kappa_E \, d \mathcal{H}^1  \right|.
\end{equation}
Therefore, for $\delta'$ sufficiently small and $\Lambda'$ sufficiently large, we have:
\begin{equation}
    \mathcal J(E_t) \geq P_{\varphi}(E)+ \vert t \vert \left(- \left| \int_{\pa E}\kappa_{E} \varphi(\nu_E)\, d \mathcal{H}^1 \right|+ \frac{\Lambda'}{2} \right) \geq P_{\varphi}(E)=\mathcal J (E) \quad \forall t \in [-\delta',\delta'].
\end{equation}
      \\
     \textit{Step 2:}  $A_m$ is an $(\omega,r_0,\frac{1}{2})$-almost minimizer of the $\varphi$-perimeter, with $ r_0:=\frac{\sigma_E}{2}$.
     
     Let $G\subset \R^2$ be such that $ A_m \Delta  G  \Subset B_r(x)$ for some $r \leq r_0$. We define $\hat{G}:= G \cap E_{\sigma_E}$. Then we have $ A_m \Delta \hat G  \subset  {\rm cl}(I_{\sigma_E}(\pa E))$. By the minimality of $A_m$, we obtain:    \begin{equation}\label{lambdapome}
    \begin{split}
     P_{\varphi}&(A_m) -P_{\varphi}(\hat G) \leq \Lambda (d_{L^2}(\hat G;E)- d_{L^2}(A_m ;E))+ \vert \hat G \Delta A_m \vert\\
        &\leq \Lambda \int_{I_{\sigma}(\pa E)} f_G \circ \pi_{\pa E} (\chi_{\hat G}-\chi_{A_m})dy + \vert \hat G \Delta A_m \vert   \\
        &= \Lambda' \int_{B_{C(K)r}(\pi_{\pa E}(\cdot)) \cap \pa E}  f_{\hat G}(\cdot) \int_{l_K-c_K r}^{l_K+c_K r} (\chi_{\hat G}(\cdot+t \nu_{E}(\cdot))- \chi_{A_m}(\cdot+t \nu_{E}(\cdot)))J \Phi_t(\cdot) \\
        &  \qquad+\vert \hat G \Delta A_m \vert  \\
        & \leq\Lambda' \| f_{\hat G} \|_{L^2(\pa E)}C(K)  r^{1+ \frac{1}{2}} + \pi r^{2}\leq ( \Lambda' C(K)+1) r^{1+ \frac{1}{2}}=  \tilde \omega r^{1+ \frac{1}{2}},
        \end{split}
    \end{equation}
    where $f_{\hat G} $ is the function realizing the supremum in $ d_{L^2}(\hat G;E)$, and in the third line we used the fact that $ \pa E$ is of class $C^2$, so there exist constants $c_K, l_K$ such that: $$ B_r(x) \subset \{  y \colon\,\, \pi_{\pa E}(y) \in B_{c_K r}(\pi_{\pa E}(x)) \cap \pa E, \, d_E(y) \in (l_K-c_Kr,l_K+c_Kr)\},$$
    and in the third inequality, we applied the Cauchy–Schwarz inequality. Moreover, using a standard calibration argument, we obtain:
    \begin{equation}\label{11061037f}
        P_{\varphi}(\hat G) \leq P_{\varphi}(G)+ \tilde{C}(K)r^2.
    \end{equation}
    Combining this with \eqref{lambdapome}, we get:
    \begin{equation}\label{lambdaminultima}
        P_{\varphi}(A_m) \leq P_\varphi(G)+ \omega r^{1+\frac{1}{2}} \text{ for all } G : G \Delta A_m \Subset B_r(x),\, x \in \R^2, \, r \leq r_0.
    \end{equation}
    \\  \textit{Step 3:} $A_m=E$.
    
Let $\delta_1$ be the constant depending on $K$ given by Lemma \ref{Lemma1006}. Applying Proposition \ref{propdadim}, we obtain a constant $\delta_0=\delta_0(K,\omega,\delta_1)$ such that if $F$ is an $(\omega,r_0,\frac{1}{2})$-almost minimizer of the $\varphi$-perimeter and $F \Delta E \subset I_{\delta_0}(\pa E)$, then $\pa F$ coincides with the graph of a $C^{1,\gamma}$ function with $C^1$-norm less than $\delta_1$. We now define $$\delta_2:= \min \{  \delta_0,\delta_1, \delta' \}.$$ Therefore, by Proposition \ref{propdadim} again, we conclude that $ \pa A_m$ is a normal graph over $ \pa E$, i.e.,
\begin{equation}
    \pa A_m= \{ x + \psi(x) \nu_E(x) \colon \, x \in \pa E \},
\end{equation}
 with $ \| \psi \|_{C^1(\pa E)} \leq \delta_1 $. Hence, we are in a position to apply inequality \eqref{lemmatesi1}, which gives
 \begin{equation}
 \begin{split}
     \frac{J_{\varphi}}{2}\| \nabla_{\pa E}& \psi \|_{L^2(\pa E)}^2+ P_{\varphi}(E) \leq P_\varphi(A_m) + \Lambda d_{L^2}(A_m;E) \\
     &\leq P_\varphi(A_m)+ \Lambda d_{L^2}(A_m;E) + \vert \vert A_m \vert -1 \vert  \leq P_\varphi(E)
\end{split}
 \end{equation}
 where in the last inequality we used the minimality of $A_m$. From this inequality, we deduce that $ \nabla_{\pa E} \psi=0$, so $ \psi$ must be constant. Finally, by Step 1, it follows that $A_m=E$.
\end{proof}
\begin{corollary}
    \label{dueuno}  Let $K>0$, $\sigma>0$ and $E\in\mathfrak{C}^{2,\alpha}_{K,\sigma}(E_0)$ for some $ \alpha \in (0,\frac{1}{2}]$, such that $ \vert E \vert=1$. Then, for every $ F \in \mathcal{B}_{\delta_2}(E) $ such that $\vert F \vert=1 $, the following inequality holds:
\begin{equation}\label{lemmatesi2}
    P_{\varphi}(E) \leq P_{\varphi}(F)+ \Lambda' d_{L^2}(F;E),
\end{equation}
where $\delta_2$ and $\Lambda'$ are the constants given in Lemma \ref{lem:quasi.min0}.
\end{corollary}

\subsection{Minimizers of \texorpdfstring{$\mathcal F_h$}{Fh} are almost minimizers of the perimeter}

\begin{corollary}\label{importante}
 Let $K>0$, $\sigma>0$ and $E\in\mathfrak{C}^{2,\alpha}_{K,\sigma}(E_0)$ for some $ \alpha \in (0,\frac{1}{2}]$, and such that $ \vert E \vert=1$. Let $\delta_2$ be the constant given in Lemma \ref{lem:quasi.min0}. Let $F$ be a minimizer of the problem
	\begin{equation}\label{pbh1111111}
		\min\big\{\mathcal F_{h}(A,E):  A \in \mathcal{B}_{\delta_2}(E),\, \vert A \vert=1\big\}.
	\end{equation}
 Then 
\begin{equation}\label{formuimp}
	 d_{L^2}(F;E) \leq 2 \Lambda' h,
\end{equation}
where $\Lambda'$ is the constant from Lemma \ref{lem:quasi.min0}. 
\end{corollary}
\begin{proof}
Using the minimality of $F$ in \eqref{pbh1111111} and then applying the inequality in \eqref{lemmatesi2}, we obtain
\begin{equation}\label{eq:cor.easy}
	P_{\varphi}(F)+ \frac{1}{2h} d_{L^2}^2(F;E) \leq P_{\varphi}(E) \leq P_{\varphi}(F) + \Lambda' d_{L^2}(F;E).
\end{equation}
The conclusion then follows immediately from \eqref{eq:cor.easy}.
\end{proof}
\begin{lemma}\label{kafka2Lemma}
     Let $K>0$, $\sigma>0$ and $E\in\mathfrak{C}^{2,\alpha}_{K,\sigma}(E_0)$ for some $ \alpha \in (0,\frac{1}{2}]$, and such that $ \vert E \vert=1$. Let $\delta_2 > 0$ be the constant provided by Lemma~\ref{lem:quasi.min0}.
     There exists a constant $C_1(K)>0$ such that, for every  $ F \in \mathcal{B}_{\delta_2}(E)$ and every $t \in (-\delta_2,\delta_2)$, the following inequalities hold:
    \begin{align}
        &P_\varphi(F_t ) \leq P_\varphi(F) +C_1(K) \vert  F_t \Delta F  \vert, \label{eq:inner.parallel}  \\
        &P_\varphi(F^t) \leq P_\varphi(F)+ C_1(K)\vert F^t \Delta F \vert, \label{eq:outer.parallel}
    \end{align}
    where $F_t := F \cup E_t$ and $F^t := F \cap E_t$.
\end{lemma}
\begin{proof}
We prove only \eqref{eq:inner.parallel}, as the proof of \eqref{eq:outer.parallel} follows by a similar argument. We begin by observing that \begin{equation}\label{lab1}
    P_{\varphi}(F_t)= \int_{\pa E_t \cap F^c} \varphi(\nu_{E_t})+ P_{\varphi}(F; E_\sigma \setminus E_t)
\end{equation}
and that 
\begin{equation}\label{lab2}
    P_\varphi(F)= \int_{\pa^* F \cap E_t} \varphi(\nu_{F})+P_{\varphi}(F; E_\sigma \setminus E_t). 
\end{equation}
    Let $\varepsilon>0 $ be such that $d_E \in C^2(I_{\delta_2+ \varepsilon}(\pa E))$. Let $ \eta \in C^{\infty}_c(I_{\delta_2+ \varepsilon}(\pa E), [0,1])$ be such that $\eta(x)=1$ for all $ x\in I_{\delta_2}(\pa E)$. Define $$T_t(x):= \nabla \varphi(\nu_{E_t} \circ \pi_{\pa E_t}(x) \eta (x)).$$ 
    Applying the divergence theorem yields
    \begin{equation}\label{kafka1}
        \int_{E_t \setminus F} \Div(T_t)= \int_{\pa E_t \cap F^c} \nabla \varphi(\nu_{E_t}) \cdot \nu_{E_t}- \int_{\pa F \cap E_t} \nabla \varphi(\nu_{E_t} \circ \pi_{\pa E_t}(x) \eta(x)) \cdot \nu_F(x).
    \end{equation}
   By the $1$-homogeneity of $\varphi$, we have 
   $$ \nabla \varphi(\nu_{E_t}) \cdot \nu_{E_t}= \varphi(\nu_{E_t}).$$ Moreover, the convexity of $\varphi$, together with the triangle inequality, implies that for all $\xi, \eta \in \R^2$,
    \begin{equation}
        \nabla \varphi (\eta) \cdot \xi \leq \varphi (\eta + \xi)- \varphi(\eta) \leq \varphi (\xi) ,
    \end{equation}
    so that, by applying the same inequality to $-\xi$, we obtain
    \[
     -1\leq  \frac{\nabla \varphi(\eta) \cdot \xi}{\varphi(\xi)} \leq 1. 
   \]
  In particular, for all $x\in \pa^* F$, it follows that  
  \[-\varphi(\nu_F(x))\leq \nabla \varphi(\nu_{E_t} \circ \pi_{\pa E_t} (x) ) \cdot \nu_F(x) \leq \varphi(\nu_F(x)).\]
    Combining \eqref{lab1}, \eqref{lab2}, and \eqref{kafka1}, we conclude that
    \begin{equation}
        P_{\varphi}(F_t) \leq P_{\varphi}(F)+ C(K) \vert F_t \Delta F \vert, 
    \end{equation}
    where $$ C(K):= \sup_{t \in (-\delta_2,\delta_2)} \| \Div(T_t) \|_{\infty}.$$
\end{proof}
\begin{lemma}\label{lemmalambdaminz}
    Let $K>0$, $\sigma>0$ and $E\in\mathfrak{C}^{2,\alpha}_{K,\sigma}(E_0)$ for some $ \alpha \in (0,\frac{1}{2}]$ and such that $ \vert E \vert=1$.  Then there exist constants $ \Sigma= \Sigma(K)$ and $ \delta_3= \delta_3(K)$ such that the following minimization problems are equivalent:
    \begin{align}\label{eq:minimo.0}
    &\min\big\{\mathcal F_{h}(A,E): \,  A \in \mathcal{B}_{\delta_3}(E) , \, \vert A \vert=1\big\},\\
        &
        \label{minlemma}
        \min \left\{ \mathcal{F}_h(A,E)+ \Sigma\vert \vert A \vert - 1 \vert \colon \, A \in \mathcal{B}_{\delta_3}(E)  \right\}.
    \end{align}
   Moreover, for every set $F$ 
    that solves \eqref{eq:minimo.0} (and hence also \eqref{minlemma}),
    there exist constants $\omega_0,\, r_0>0$ depending only on $K$ such that: $F$ is an $(\omega_0,r_0,\frac{1}{2})$-almost minimizer of the $\varphi$-perimeter; $\pa F \subset I_{\delta_3}(\pa E)$; $\pa F$ coincides with the graph of a function $ \psi : \pa E \rightarrow \R$. 
\end{lemma}
\begin{proof}
 Let $\Sigma:= C_1(K)+1$, where $C_1(K)$ is the constant from Lemma \ref{kafka2Lemma}. Applying Proposition \ref{propdadim} with $\omega_0 =\Sigma+3\omega $, where $\omega$ is the constant appearing in \eqref{lambdaminultima}, we obtain a corresponding constant $\delta_0= \delta_0(K,\omega_0,\alpha)$.  Let  $\delta_2$ be the constant from Lemma \ref{lem:quasi.min0}, and define 
 $$ \delta_3:= \min \{  \delta_0, \delta_2 \} .$$
 
 To establish the equivalence between problems \eqref{eq:minimo.0} and \eqref{minlemma}, it suffices to show that any minimizer of \eqref{minlemma} has unit measure.
 Let $F$ be a solution to the problem \eqref{minlemma}.
 Arguing by contradiction, suppose that $|F|\not=1$.
Let us first assume that $ \vert F \vert >1$. \\
  \textit{Claim:} $F$ is an $(\omega_0,r_0, \frac{1}{2})$-almost minimizer of the $\varphi$-perimeter, where $r_0 = \frac{\delta_3}{2}$.
  
  Let $ G \subset \R^2 $ be such that $ F \Delta G \Subset B_r(x) $ for some $ r \leq r_0$. Define  $\hat{G}:= G \cap E_{\delta_3} $. Let $\Lambda'$ be the constant of Lemma \ref{lem:quasi.min0}.\\
 \textit{Case $d_{L^2}(\hat G;E) \leq 4 \Lambda' h$.}\\
By the minimality of $F$ in \eqref{minlemma}, we have
\begin{equation}\label{11112024form1}
\begin{split}
    P_\varphi(F)-P_\varphi(\hat G) & \leq \frac{1}{2 h} \big[d_{L^2}(\hat G;E)-d_{L^2}(F;E)\big] \big[d_{L^2}(\hat G;E)+d_{L^2}(F;E)
   \big] + \Sigma \vert F \Delta \hat G \vert\\
   &\leq 3 \Lambda' \big[ d_{L^2}(\hat G;E)-d_{L^2}(F;E)\big]+ \Sigma \vert F \Delta \hat G \vert.
   \end{split}
\end{equation}
\textit{Case $d_{L^2}(\hat G;E) > 4 \Lambda' h$.}\\ 
Note that
\begin{equation}\label{eq:dist.G.F}
    2 d_{L^2}(F;E) \leq 4 \Lambda' h < d_{L^2}(\hat G;E) \implies \frac{d_{L^2}(\hat G;E)}{2} < d_{L^2}(\hat G;E)-  d_{L^2}(F;E).
\end{equation}
Using the minimality of $F $, along with \eqref{eq:dist.G.F} and \eqref{lemmatesi2}, we deduce:
\begin{equation}\label{11112024form2}
    P_{\varphi}(F)-P_{\varphi}(\hat G) \leq P_{\varphi}(E)-P_{\varphi}(\hat G) \leq  \Lambda' d_{L^2}(\hat G;E) \leq  2 \Lambda' \big[ d_{L^2}(\hat G;E)-d_{L^2}(F;E)\big].
\end{equation}
Arguing as in \eqref{lambdapome}, and combining inequalities \eqref{11112024form1} and \eqref{11112024form2}, we conclude that
\begin{equation}\label{21112024primpal}
    P_\varphi(F) \leq P_\varphi(\hat G)+ (3 \tilde \omega+ \Sigma)r^{1+\frac{1}{2}},
\end{equation}
where $\tilde\omega$ is the constant from \eqref{lambdapome}. Using this estimate together with \eqref{11061037f}, the claim follows.

Using Lemma \ref{propdadim}, we obtain the existence of a function $ \psi \in C^{1,\gamma}(\pa E)$ such that
$$ \pa F = \{ x + \psi(x)\nu_E(x)\colon x \in \pa E\}\subset I_{\delta_3}(\pa E).$$ 
  For each $t \in (-\delta_3,\delta_3)$, let $F_t$ be defined as in Lemma \ref{kafka2Lemma}, and choose $t_0\in (-\delta_3,\delta_3)$ such that $ \vert F_{t_0} \vert=1$. Then there exists a function $\psi_{t_0}\in C^0(\pa E)$ such that
$$ \pa F_{t_0}= \{ x + \psi_{t_0}(x)\nu_E(x)\colon x \in \pa E\}.$$
By construction, we have  $$ \vert \psi_{t_0}(x)+ \frac{\psi_{t_0}^2(x)}{2}\kappa_E(x) \vert^2 \leq \vert \psi(x)+ \frac{\psi^2(x)}{2}\kappa_E(x) \vert^2 \text{ for all } x \in \pa E.$$
Recalling formula \eqref{L^2distperlafunz}, it follows that
\begin{equation}
    d_{L^2}^2(F_{t_0};E) \leq d_{L^2}^2(F;E).
\end{equation}
Using the minimality of $F$, the above inequality, and Lemma \ref{kafka2Lemma}, we deduce
\begin{equation}\label{eq:conclusion.equivalence}
\begin{split}
    P_{\varphi}(F)+ \frac{1}{2 h} d_{L^2}^2(F;E)+ \Sigma \vert \vert F \vert- 1\vert &\leq P_{\varphi}(F_t)+ \frac{1}{2 h} d_{L^2}^2(F_t;E) \\
    & \leq P_{\varphi}(F)+ C_1(K) \vert \vert F \vert -1 \vert +\frac{1}{2 h} d_{L^2}^2(F_t;E) \\
     &\!\!\!\!\Downarrow\\
     \Sigma \vert \vert F \vert- 1\vert    \leq C_1(K) \vert \vert F \vert -1 \vert +&\frac{1}{2 h} d_{L^2}^2(F_t;E) - \frac{1}{2 h} d_{L^2}^2(F;E) \leq C_1(K) \vert \vert F \vert -1 \vert.
\end{split}
\end{equation}
Since $\Sigma >C_1(K)$, this yields a contradiction unless $ \vert F \vert\leq 1$.
To conclude that $ \vert F \vert= 1$, it suffices to repeat the same argument, with the only modification being that we test the minimality of $F$ against the set $F^t$ (instead of $F_t$) in \eqref{eq:conclusion.equivalence}.
 \end{proof}
\subsection{First variation of the functional}
We have established $C^{1,1/2}$-regularity properties for minimizers of $\mathcal F_h$, which enable us to compute the first variation of $\mathcal F_h$ at a minimizer. In this subsection, we derive the formula for the first variation of the distance $d_{L^2}(F;E)$, when $F$ is of class $C^1$. 
\begin{proposition}\label{propEL}
	Let $ E \subset \R^2$  be a bounded open set of class $C^2$, and let $ \sigma \leq \sigma_E$. Let $F  \in \mathcal{B}_{\sigma}(E)$ be a set of class $C^{1}$  such that $F$ is a normal graph over $\pa E$, i.e.  $  \pa F= \{ x + \psi(x)\nu_{E}(x) \colon x \in \pa E\}$. Let $ \upsilon \in C^1(\pa E)$, and define the vector field $$X(x):= \upsilon(\pi_{\pa F}(x)) \nu_F(\pi_{\pa F}(x)) \eta(x) \text{ where }  \eta \in C_c^{\infty}(I_{\sigma}(\partial E)),\, \eta(x)= 1\text{ for all } x \in I_{\frac{3}{4}\sigma}(\pa E).$$
    Let $\Upsilon: (-\varepsilon,\varepsilon) \times \R^2 \rightarrow \R^2$ be the solution of the Cauchy problem
	\begin{equation}
		\begin{cases}
			\frac{\pa }{\pa t} \Upsilon (t,x)= X(\Upsilon(t,x)) \quad \forall x \in \R^2,\\
			\Upsilon(0,x)=x \quad \forall x \in \R^2. 
		\end{cases}
	\end{equation}
	Let $ f \in L^2(\partial E)$ be the function realizing the supremum in the definition of  $d_{L^{2}}(F;E)$, namely,
	\begin{equation}
		d_{L^{2}}(F;E)= \int_{\R^2} f(\pi_{\pa E}(x))(\chi_F(x)-\chi_E(x)) \, dx.
	\end{equation}
	Then, \begin{equation}\label{tesi01102024}
		\frac{d }{d t} d_{L^{2}}^2(\Upsilon(t,F);E) |_{t=0}= 2 d_{L^2}(F;E)\int_{\partial F} f(\pi_{\pa E}(y)) \upsilon(y) \, d \mathcal{H}^{1}_y.
	\end{equation}
\end{proposition}
\begin{proof} Let $ F_t:= \Upsilon(t,F)$.
We first observe that, for $t$ sufficiently small, there exists a function $ \psi_t \in C^1(\pa E)$ such that $$ \pa F_t= \{ x + \psi_t(x) \nu_E(x) \colon x \in \pa E\},$$
and moreover $\| \psi_t- \psi \|_{\infty} \rightarrow 0$ as $t \rightarrow 0$.  
By using the definition of $\xi_{F,E}$, applied with $F_t$ in place of $F$ and using \eqref{L^2distperlafunz}, we find 
$$d_{L^2}^2(F_t;E)=\int_{\pa E} \vert \xi_{F_t,E}(x)\vert^2 d\mathcal{H}^1_x.$$ 
Since $\| \psi_t- \psi \|_{\infty} \rightarrow 0$ as $t \rightarrow 0$, it follows that $\xi_{F_t,E} \rightarrow \xi_{F,E} $ in $L^{\infty}$ as $t \rightarrow 0$. Moreover, by \cite[Proposition 17.8]{Maggibook}, we have
\begin{equation}\label{18072024pom2}
\lim_{t\to0}\frac{1}{t}\bigg(\int_{F_t}\varphi\,dx-\int_{F}\varphi\,dx\bigg)=\int_{\pa F}\varphi \upsilon\,d\Ha^{1}, 	\text{ for all } \varphi \in C^0(\R^2).
\end{equation}  
Recalling the definition of $\Phi_t$ from \eqref{phit}, and using \eqref{funzione}, we have
\[
\xi_{F_t,E}(x)=  \int_{-\sigma}^{\sigma} (\chi_{F_t}(x+t\nu_E(x))-\chi_E(x+t \nu_E(x))) J_\tau \Phi_s(x) \, d s,
\]
where $J_\tau \Phi_s(x)= 1+s \kappa_E (x)$. Therefore,
\begin{equation}
    \begin{split}
        &\frac{\int_{\pa E} \vert \xi_{F_t,E} \vert^2 \,d \mathcal{H}^1-\int_{\pa E} \vert \xi_{F,E} \vert^2 \, d \mathcal{H}^1 }{t}\\
        &\qquad= \frac{\int_{\pa E} \xi_{F_t,E}(\xi_{F_t,E}+\xi_{F,E}) \,d \mathcal{H}^1- \int_{\pa E}\xi_{F,E} (\xi_{F_t,E}+\xi_{F,E}) \, d \mathcal{H}^1 }{t} \\
        &\qquad=  \frac{1}{t}\int_{\pa E} (\xi_{F_t,E}+\xi_{F,E})(x)\left(
         \int_{-\sigma}^{\sigma} (\chi_{F_t}(x+s\nu_E(x))-\chi_F(x+s \nu_E(x))) J_\tau \Phi_s(x) d s\right)\,d\Ha^1_x
        \\
        &\qquad=\frac 1t\int_{I_\sigma(\pa E)} (\xi_{F_t,E}+\xi_{F,E}) \circ \pi_{\pa E} \, (\chi_{F_t}-\chi_F )\, dx.
    \end{split}
\end{equation}
Combining this identity with \eqref{18072024pom2}, we obtain \eqref{tesi01102024}.
\end{proof}
We state the formula for the first variation of the $\varphi$-perimeter.  We omit the proof, since it is classical and can be found in many standard textbooks.
\begin{proposition}
    Let $E \subset \R^2$ be a bounded open set of class $C^2$, and let $ \sigma \leq \sigma_E$.
    Let $\eta \in C_c^{\infty}(I_{\sigma}(\pa E))$ be such that $\eta(x)= 1$ for all $x \in I_{\frac{3}{4}\sigma}(\pa E)$. 
    Given $ \upsilon \in C^1(\pa E)$, define the vector field
    $$X(x):= \upsilon(\pi_{\pa E}(x)) \nu_E(\pi_{\pa E}(x)) \eta(x).$$
    Let $\Upsilon: (-\varepsilon,\varepsilon) \times \R^2 \rightarrow \R^2$ be the solution of the Cauchy problem
	\begin{equation}
		\begin{cases}
			\frac{\pa }{\pa t} \Upsilon (t,x)= X(\Upsilon(t,x)),  \\
			\Upsilon(0,x)=x,
		\end{cases}
        \qquad x\in \R^2.
	\end{equation}
    Then,
    \begin{equation}
        \frac{d}{d t} P_{\varphi}(\Upsilon(t,E)) |_{t=0}= \int_{\pa E} \upsilon \, \Div_{\tau} (\nabla \varphi(\nu_E)) \, d \mathcal{H}^1 = \int_{\pa E} \upsilon \kappa_E^\varphi\, d \mathcal{H}^1.
    \end{equation}
\end{proposition}

\section{The first time step: regularity estimate}\label{sec:base}
We begin this section with the following technical lemma, which will be useful several times later on.

\begin{lemma}\label{lemm1}
    Let $K>0$, $\sigma_0>0$ and $E\in\mathfrak{C}^{2,\alpha}_{K,\sigma_0}(E_0)$ be such that $\vert E \vert=1$. Then there exist constants $\sigma,
    C$ depending only on $K$ such that the following holds: if $F \subset \R^2$ with $ \pa F = \{ x+ \psi(x)\nu_E(x)\colon x \in \pa E\}$ for some function $ \| \psi \|_{C^1(\pa E)} \leq \sigma$ with $ \vert F \vert=1$, then \begin{equation}\label{tesilemm}
        \frac{1}{C} \| \pa_{\pa E} \psi \|_{L^2(\pa E)} \leq  \| \pa_{\pa E} \xi_{F,E} \|_{L^2(\pa E)} \leq C \| \pa_{\pa E} \psi \|_{L^2(\pa E)}
    \end{equation}
    where $\xi_{F,E}$ is defined as in Lemma \ref{lem:l2dist}.
\end{lemma}
\begin{proof}
By Lemma \ref{lem:l2dist}, we have  
$ \xi_{F,E}= \psi + \frac{\psi^2}{2}\kappa_E$. 
For $\sigma$ sufficiently small, we obtain  
\begin{equation}\label{perc}
    \frac{\psi^2}{2} \leq \big(\psi+ \frac{\psi^2}{2}\kappa_E\big)^2= \xi_{E,F}^2 .
\end{equation}   
Computing the tangential gradient of $\xi_{F,E}$, we find that
\begin{equation}\label{050502025zoo}
    \pa_{\pa E} \xi_{F,E}= \pa_{\pa E} \psi+ \psi \kappa_E \pa_{\pa E} \psi+ \frac{\psi^2}{2} \pa_{\pa E}\kappa_E.
\end{equation}
Since $\vert E \vert = \vert F \vert$, we have $ \int_{\pa E} \xi_{F,E} \, d \mathcal{H}^1=0$ (see formula \eqref{eq:distance.and.area}). 
Therefore, using \eqref{perc}, \eqref{050502025zoo}, and the H\"{o}lder inequality, we obtain \begin{equation}\label{asdhhhhh}
    \begin{split}
        \| \pa_{\pa E} \xi_{F,E} \|_{L^2(\pa E)} &\leq (1+ \sigma K) \| \pa_{\pa E} \psi \|_{L^2(\pa E)}+ \frac{\sigma}{\sqrt{2}} \| \xi_{F,E} \|_{L^2(\pa E)}\\
        & \leq  (1+ \sigma K) \| \pa_{\pa E} \psi \|_{L^2(\pa E)}+ \frac{\sigma C_1}{\sqrt{2}} \| \pa_{\pa E} \xi_{F,E} \|_{L^2(\pa E)},
        \end{split}
    \end{equation}
    where in the last inequality we used the Poincar\'{e} inequality: $$\| \xi_{F,E} \|_{L^2(\pa E)} \leq C_1 \| \pa_{\pa E} \xi_{F,E} \|_{L^2(\pa E)}.$$ 
     Taking $\sigma$ small enough in \eqref{asdhhhhh}, we obtain  \begin{equation}\label{05052025primacaf1}
      \| \pa_{\pa E} \xi_{F,E} \|_{L^2(\pa E)} \leq C \| \pa_{\pa E} \psi \|_{L^2(\pa E)} . 
    \end{equation}
   From \eqref{050502025zoo} and using the Sobolev embedding together with the H\"{o}lder inequality, we get
    \begin{equation}\label{050502025dcaf}
        \begin{split}
            \| \pa_{\pa E} \psi \|_{L^2(\pa E)} &\leq \| \pa_{\pa E} \xi_{F,E} \|_{L^2(\pa E)} + \| \pa_{\pa E} \psi \|_{L^2(\pa E)} K \| \psi \|_{\infty}+ \frac{1}{2}\| \pa_{\pa E} \kappa_E \|_{L^2(\pa E)}  \| \psi^2 \|_{L^2(\pa E)}\\
            & \leq \| \pa_{\pa E} \xi_{F,E} \|_{L^2(\pa E)} + \sigma K\| \pa_{\pa E} \psi \|_{L^2(\pa E)}+ \sigma C(K)\| \xi_{F,E}\|_{L^2(\pa E)}\\
            & \leq \| \pa_{\pa E} \xi_{F,E} \|_{L^2(\pa E)} + \sigma K\| \pa_{\pa E} \psi \|_{L^2(\pa E)}+ \sigma C(K)C_1\| \pa_{\pa E} \xi_{F,E}\|_{L^2(\pa E)},
        \end{split}
    \end{equation}
    where the second inequality uses \eqref{perc}, and the third uses the Poincaré inequality. Taking $\sigma$ small enough in \eqref{050502025dcaf}, we get
    \begin{equation}\label{050525dcaftesi}
      \frac{1}{C} \| \pa_{\pa E} \psi \|_{L^2(\pa E)} \leq  \| \pa_{\pa E} \xi_{F,E} \|_{L^2(\pa E)}.   
    \end{equation}
    Combining  \eqref{05052025primacaf1} and \eqref{050525dcaftesi} yields the desired formula, i.e., \eqref{tesilemm}.
\end{proof}

The following theorem provides one of the key ingredients for proving the convergence of the discrete scheme. Its proof involves several delicate estimates and relies on the Euler–Lagrange equation.
\begin{theorem}\label{mainthm}
	Let $K>0$, $\sigma_0>0$ be fixed, and let $\delta_3$ be the constant given by Lemma \ref{lemmalambdaminz}. Then there exist constants $\delta_4=\delta_4(K)$, $\sigma_1= \sigma_1(K)$, $\hat h=\hat h(K) $ and $\hat K=C(K)$, such that the following holds. Fix $0<h<\hat h$ and let $E \subset \R^2$ be of class $C^4$ abd such that $E \in \mathfrak{H}^3_{K,\sigma_0}(E_0)$, with $|E|=1$, and assume that
    \begin{equation}\label{eq:assumption.curvature}
    \| \kappa_E \|_{L^\infty(\pa E)} +\|\pa_{\pa E} \kappa_E\|_{L^2(\pa E)}
     + \sqrt h
    \| \pa_{\pa E}^2\kappa_E \|_{L^2(\pa E)} \leq K.\end{equation} 
    Let $F$ be a minimizer of \begin{equation}\label{pbh}
		\min\big\{\mathcal F_{h}(A,E): A \in \mathcal{B}_{ \delta_4}(E),\, \vert A \vert=1\big\}.
	\end{equation}
    Then, $\partial F \Subset I_{ \delta_4}(\partial E)$ and $\pa F$ coincides with the graph of a smooth function $\psi : \partial E \rightarrow \R$ satisfying
	\begin{equation}\label{tesi16nov}
	\mathrm{dist}_{\mathcal{H}}(\pa F,\pa E)=\| \psi \|_{L^{\infty}(\pa E)} \leq  \hat K h, \quad \| \pa_{\pa E} \psi \|_{L^2(\pa E)} \leq  \hat K h.
	\end{equation}
    Moreover,
    \begin{equation}\label{tesi18nov}
    \begin{split}
      &\frac{\| \pa_{\pa E}^2\psi \|_{L^2(\pa E)}}
{\sqrt h} +  \| \pa_{\pa E}^3  \psi \|_{L^2(\pa E)} \leq  \hat K,
\\ 
&\| \kappa_F  \|_{L^\infty(\pa F)}+
     \| \pa_{F}\kappa_F \|_{L^2(\pa F)} 
     + \sqrt h
    \| \pa_{\pa F}^2\kappa_F \|_{L^2(\pa F)} \leq \hat K.
\end{split}
    \end{equation}
Finally, $ F \in \mathfrak{H}^3_{\hat{K},\sigma_1}(E_0)$.    
\end{theorem}
\begin{proof}
    We divide the proof into several steps. \\
    \textit{Step 1: The Euler--Lagrange equation and first consequences.}
    
Fix $\eta>0$, to be chosen later.  By Lemmas \ref{propdadim} and \ref{lemmalambdaminz}, there exist $\delta_4 \leq \delta_3$ and a function $\psi \in C^1(\pa E)$ with $ \| \psi \|_{C^1(\pa E)} \leq \eta$ and such that
    $$\pa F= \{ x + \psi(x) \nu_{E}(x) \colon x \in \pa E\}.$$ Using the minimality of $F$, we have
     $ P_\varphi(F) \leq P_\varphi(E)$. Therefore, by \eqref{lemmatesi1}, it follows that
    \begin{equation}\label{120625form1}
       \frac{J_\varphi}{2} \|\pa_{\tau} \psi \|^2_{L^2(\pa E)} \leq \Lambda d_{L^2}(F;E),
    \end{equation}
    where $\Lambda$ is the constant from Lemma \ref{Lemma1006}.
     From formulas  \eqref{L^2distperlafunz}, \eqref{fomr07112024} and \eqref{formuimp}, we obtain 
    \begin{equation}\label{112024fomrimm}
       \frac{1}{\sqrt{2}} \| \psi \|_{L^2(\pa E)} \leq d_{L^2}(F;E) \leq 2\Lambda' h,
    \end{equation}
    where $\Lambda'$ is the constant appearing in \eqref{formuimp}. Combining \eqref{formuimp}, \eqref{120625form1}, and the Sobolev embedding, we find that
    \begin{equation}\label{13112024sera1}
    \begin{split}
         \| \psi \|_{\infty} &\leq C(K) \big( \| \psi \|_{L^2(\pa E)}+ \| \pa_\tau \psi \|_{L^2(\pa E)} \big) \\
         &\leq C(K) \big( d_{L^2}(F;E)+ d_{L^2}(F;E)^\frac{1}{2}  \big) \leq C(K)(h+ \sqrt{h}) \\
         &\leq C(K)\sqrt{h}.
         \end{split}
    \end{equation} 
    By \eqref{13112024sera1}, for $ \hat h$ sufficiently small we have $ \pa F \Subset  I_{\delta_4}(\pa E)$, so that $F$ satisfies the Euler--Lagrange equation:
    \begin{equation}\label{eq1112}
        \kappa_F^{\varphi}(y)+ \frac{d_{L^2}(F;E)}{h}f\circ \pi_{\pa E}(y)=\lambda_F \quad \text{for all }y \in \pa F,
    \end{equation}
where $f\in L^2(\pa E)$ realizes the supremum in the definition of $d_{L^2}(F;E)$, and $\lambda_F$ is a Lagrange multiplier. 
 Since $F$ is a normal graph over $\pa E$, using \eqref{funzg_E(x)}, we can rewrite \eqref{eq1112} using \eqref{funzg_E(x)} as
    \begin{equation}
        \kappa_F^{\varphi}(x+ \psi(x)\nu_E(x))+ \frac{\psi(x)+ \frac{\psi(x)^2 \kappa_E(x)}{2}}{h}= \lambda_F \quad \text{ for all }x \in \pa E.
    \end{equation}
    Since the anisotropy is smooth and uniformly convex, classical elliptic regularity theory implies that  $\psi \in C^4$, and hence $F$ is as regular as $E$.
   Integrating \eqref{eq1112}, we obtain
   \begin{equation}\label{13022025pom1}
    \begin{split}
        \lambda_F P_{\varphi}(F) &\leq  \int_{\pa F} \kappa_F^{\varphi} \varphi(\nu_F) \,d \mathcal{H}^1+ \frac{d_{L^2}(F;E)}{h} M_\varphi \int_{\pa E} \vert f \vert \vert J \Psi \vert \,d \mathcal{H}^1 \\
        &\leq c_\varphi + C(K,\varphi) \| f \|_{L^2(\pa E)} \| J \Psi \|_{L^2(\pa E)} \leq C(K,\varphi),
        \end{split}
    \end{equation}
    where 
$$ \Psi: \pa E \rightarrow \pa F, \qquad \Psi(x):= x+ \psi(x)\nu_E(x),  $$
 and we used $ \| f \|_{L^2(\pa E)} \leq 1$, $\|J \Psi \|_{L^2(\pa E)} \leq C(K)$  and the anisotropic Gauss--Bonnet theorem (Lemma \ref{lGaussB}).
Using \eqref{mMvarphi} and the isoperimetric inequality,
$$1=\vert F \vert^{\frac{1}{2}} \leq \frac{P(F)}{2 \sqrt{\pi}} \leq \frac{P_\varphi(F)}{m_\varphi2 \sqrt{\pi}},  $$
 where $m_\varphi$ is defined as in \eqref{mMvarphi}. Combining this with \eqref{13022025pom1}, we deduce that \begin{equation}\label{smoltiply}
        \lambda_F \leq \frac{C(K,\varphi)}{P_\varphi(F)}\leq C(K,\varphi).
    \end{equation}
Using formulas \eqref{funzg_E(x)}, \eqref{L^2distperlafunz}, and \eqref{interpol1}, the Euler--Lagrange equation on $\pa E$ becomes
\begin{equation}\label{eqeulerobvera}
        \frac{\psi(x)+ \frac{\psi^2(x)\kappa_E(x)}{2}}{h}=-\kappa_F^{\varphi} \circ \Psi (x) +\lambda_F = g(\nu_E(x)) \pa_{\tau}^2\psi(x)-\kappa_E^\varphi(x)-R_0(x)+\lambda_F,
    \end{equation}
    where 
 \begin{equation}\label{eqeulroresto}
     R_0=a(\nu_E,\psi\kappa_E,\pa_{\tau} \psi) \pa_{\tau}^2 \psi+ b(\nu_E,\psi \kappa_E,\pa_\tau\psi)\pa_{\tau} (\psi \kappa_E)+ c(\nu_E,\psi,\pa_{\tau} \psi,\kappa_E).
    \end{equation}
  We recall that $C_g = \sup_{ \eta \in \mathcal{S}^1} g(\eta)\geq\inf_{ \eta \in \mathcal{S}^1} g(\eta)=c_g>0$.
  By \eqref{smoltiply} and \eqref{eqeulerobvera}, we obtain
  \begin{equation}\label{asdasd1}
  \begin{split}
      c_g\| \pa_{\tau}^2 \psi \|_{L^2(\pa E)} &\leq C_g\| \kappa_E \|_{L^2(\pa E)}+ \| R_0 \|_{L^2(\pa E)}+ \frac{\|  \psi + \frac{\psi^2 \kappa_E}{2} \|_{L^2(\pa E)}}{h}+ \vert \lambda \vert \\
      &\leq C(K)+ \| R_0 \|_{L^2(\pa E)},
      \end{split}
  \end{equation}
  where we used the estimate
  \begin{equation}\label{eq:l2estimate}
      d_{L^2}(F;E) = \left\|  \psi + \frac{\psi^2 }{2} \kappa_E \right\|_{L^2(\pa E)} \leq C(K) h.
  \end{equation}
  We now estimate $R_0$. For  $\eta=\eta(K)$ small enough, we have $ \|a(\nu_E,\psi \kappa_E,\pa_{\tau} \psi) \|_{\infty} \leq \frac{c_g}{9}$. 
  By formula \eqref{13112024sera1}, we have $ \| \psi \|_{\infty} \leq C(K) \sqrt{h}$. Using that $ \| \psi \|_{C^1} \leq \eta$  and $a,\,b,\,c \in C^\infty$, we obtain 
  \begin{equation}\label{160625pom1}
  \begin{split}
      &\| a (\nu_E,\psi \kappa_E,\pa_{\tau} \psi) \pa_{\tau}^2 \psi \|_{L^2(\pa E)} \leq \frac{c_g}{9}\| \pa_{\tau}^2 \psi \|_{L^2(\pa E)},\\
      &\| b(\nu_E,\psi \kappa_E,\pa_{\tau} \psi) \pa_{\tau} (\psi \kappa_E)+ c(\nu_E,\psi,\pa_{\tau} \psi,\kappa_E) \|_{L^2(\pa E)} \leq C(K).
       \end{split}
  \end{equation}
  Plugging this into \eqref{asdasd1}, we conclude that
  \begin{equation}
     \| \pa_{\tau}^2 \psi \|_{L^2(\pa E)} \leq C(K). 
  \end{equation}
  By the Sobolev embedding theorem,
  \begin{equation}
      \| \psi \|_{H^2(\pa E )} \leq C(K) \implies \| \psi \|_{W^{1,\infty}(\pa E)} \leq C(K).
  \end{equation}
  We now estimate the second derivative of $\psi$ in $L^\infty$. 
  
  By \eqref{eqeulerobvera}, \eqref{smoltiply} and \eqref{13112024sera1}, we obtain
  \begin{equation}\label{13022025pom2}
  \begin{split}
      c_g\| \pa_{\tau}^2 \psi \|_{L^\infty(\pa E)} &\leq \frac{\| \psi + \frac{\psi^2 \kappa_E}{2}\|_{L^\infty(\pa E)}}{h}+ C_g\| \kappa_E \|_{L^\infty(\pa E)}+ \| R_0 \|_{L^\infty(\pa E)}+ \vert \lambda_F \vert \\
      &\leq  \frac{C(K)}{\sqrt{h}}+ C(K)+ \frac{c_g}{9} \| \pa_{\tau}^2 \psi \|_{L^\infty(\pa E)},
      \end{split}
  \end{equation}
  where we used \eqref{160625pom1} and
  \begin{equation} 
  \begin{split}
      \| R_0 &\|_{L^\infty(\pa E)} \\
      &\leq \frac{c_g}{9} \| \pa_{\tau}^2 \psi \|_{L^\infty(\pa E)}+ \| b \|_{L^\infty(\pa E)} \|  \pa_{\tau}(\psi \kappa_E) \|_{L^{\infty}(\pa E)}+ \| c \|_{L^\infty(\pa E)}\\
      & \leq  \frac{c_g}{9} \| \pa_{\tau}^2 \psi \|_{L^\infty(\pa E)}+ C(K)\big[ \| \pa_{\tau} \psi \|_{L^\infty(\pa E)} \| \kappa_F \|_{L^\infty(\pa E)}+ \| \psi \|_{L^\infty(\pa E)} \| \pa_{\tau} \kappa_E \|_{L^\infty(\pa E)}\big] +C(K)\\
      & \leq \frac{c_g}{9} \| \pa_{\tau}^2 \psi \|_{\infty}+ C(K) \big[ C(K)+ C(K) \sqrt{h} \frac{C(K)}{\sqrt{h}}\big]+ C(K) \leq \frac{c_g}{9} \| \pa_{\tau}^2 \psi \|_{\infty}+ C(K).
      \end{split}
  \end{equation}
  We conclude from \eqref{13022025pom2} that
  \begin{equation}\label{infitolapsi}
      \| \pa_{\tau}^2 \psi \|_{\infty} \leq \frac{C(K)}{\sqrt{h}}.
  \end{equation} 
  To summarize, in Step 1 we established the following estimate: there exists a constant $C = C(K)$ such that
  \begin{equation}\label{tesistep2}
      \frac1h \| \psi \|_{L^2(\pa E)} + \frac{1}{\sqrt h} \| \pa_{\tau} \psi \|_{L^2(\pa E)} + \| \pa_{\tau}^2  \psi\|_{L^2(\pa E)}+ \sqrt h \| \pa_{\tau}^2 \psi \|_{L^\infty(\pa E)}
    \leq C(K). 
  \end{equation}
  \textit{Step 2: Improving the bounds.} 
 We claim that
  \begin{equation}\label{tesistep3}
\begin{split}
    \| \pa_{\tau}^2 \psi \|_{L^2(\pa E)} \leq C(K) \sqrt{h},\quad \| \pa_{\tau}^3 \psi \|_{L^2(\pa E)} \leq C(K).
\end{split}
\end{equation}
To establish this, we multiply equation \eqref{eqeulerobvera} by $ \pa_{\tau}^4 \psi$ and integrate over $\pa E$, obtaining
  \begin{equation}\label{moltobene}
      \int_{\pa E}\frac{\psi \pa_{\tau}^4 \psi}{h}- \int_{\pa E} g(\nu_E)\pa_{\tau}^2 \psi \pa_{\tau}^4  \psi= \int_{\pa E} \frac{\psi^2 \kappa_E \pa_{\tau}^4 \psi}{2 h}-\int_{\pa E} \kappa_E^{\varphi} \pa_{\tau}^4  \psi + \int_{\pa E} (-R_0+\lambda_F)\pa_{\tau}^4  \psi. 
  \end{equation}
  For the left-hand side, integration by parts and the Cauchy–Schwarz inequality yield
  \[
  \begin{split}
   \int_{\pa E}\frac{\psi \pa_{\tau}^4 \psi}{h} - \int_{\pa E} g(\nu_E) \pa_{\tau}^2 \psi \pa_{\tau}^4 \psi 
  &= \int_{\pa E}\frac{ (\pa_{\tau}^2 \psi)^2}{h}+\int_{\pa E} g(\nu_E)( \pa_{\tau}^3 \psi )^2
  +
  \int_{\pa E}\pa_\tau g(\nu_E) \pa_{\tau}^2 \psi \pa_{\tau}^3\psi
  \\&
  \geq \frac{c_g}{2} \int_{\pa E} \vert \pa_{\tau}^3 \psi \vert^2+\left(\frac1h-C(K,\varphi)\right) \int_{\pa E} \vert \pa_{\tau}^2 \psi \vert^2.
  \end{split}
  \]
  Thus, for $h$ small enough, inequality \eqref{moltobene} becomes
  \begin{equation}\label{veform1}
      \frac{1}{2}\int_{\pa E} \frac{\vert\pa_{\tau}^2 \psi\vert^2}{h}+ \frac{c_g}{2}\int_{\pa E} \vert \pa_{\tau}^3 \psi \vert^2 \leq  \int_{\pa E} \frac{\psi^2 \kappa_E \pa_{\tau}^4 \psi}{2 h}-\int_{\pa E} \kappa_E^{\varphi} \pa_{\tau}^4 \psi - \int_{\pa E} R_0 \pa_{\tau}^4 \psi. 
  \end{equation}
  Let $\varepsilon>0$ be some constant that will be chosen later. To estimate the first term on the right-hand side of \eqref{veform1}, we integrate by parts and apply Young's inequality:
  \begin{equation}\label{veform3}
  \begin{split}
      \int_{\pa E} \frac{\psi^2 \kappa_E^{\varphi} \pa_{\tau}^4 \psi}{2 h}& = -\int_{\pa E} \psi^2\frac{ \pa_{\tau} \kappa_E^\varphi \pa_{\tau}^3 \psi}{2 h}- \int_{\pa E}\psi \kappa_E^\varphi  \frac{\pa_{\tau}\psi \pa_{\tau}^3 \psi}{h} \\
      &\leq \varepsilon \int_{\pa E} \vert \pa^3_{\tau}\psi \vert^2 + \frac{C(K)}{\varepsilon}\int_{\pa E} \vert \pa_{\tau}  \kappa_E^\varphi \vert^2+ \varepsilon \int_{\pa E} \vert \pa_{\tau}^3 \psi \vert^2 + \frac{C(K)}{\varepsilon} \int_{\pa E} \frac{\psi^2 (\pa_{\tau} \psi)^2}{h^2}\\
      &\leq 2\varepsilon \int_{\pa E} \vert \pa^3_{\tau} \psi \vert^2 + \frac{C(K)}{\varepsilon},
      \end{split}
  \end{equation}
  where we have used the estimates $ \| \psi \|_{L^2(\pa E)} \leq C(K) h$ and $ \| \pa_{\tau} \psi \|_{L^2(\pa E)} \leq C(K)\sqrt{h}$.
 For the second term, we estimate:
  \begin{equation}\label{veform2}
  \begin{split}
      \int_{\pa E} \kappa_E^{\varphi} \pa_{\tau}^4 \psi = \int_{\pa E} - \pa_{\tau} \kappa_E^{\varphi} \pa_{\tau}^3 \psi &\leq \varepsilon \int_{\pa E} \vert \pa_{\tau}^3 \psi \vert^2 + \frac{1}{\varepsilon} \int_{\pa E} \vert \pa_{\tau} \kappa_E^{\varphi} \vert^2 \\
      & \leq \varepsilon \int_{\pa E} \vert \pa_{\tau}^3 \psi \vert^2 + \frac{C(K)}{\varepsilon}.
      \end{split}
  \end{equation}
 The last term to estimate is $\int_{\pa E} R_0 \pa^4_{\tau} \psi$. It holds that
 \begin{equation}\label{veform4}
    \left| \int_{\pa E} R_0 \pa_{\tau}^4 \psi \right|\leq C(K)+(\frac{c_g}{9}+3\varepsilon) \int_{\pa E} \vert \pa_{\tau}^3 \psi \vert^2 + \frac{C(K)}{\sqrt{h}} \int_{\pa E} \vert \pa_{\tau}^2 \psi \vert^2.
 \end{equation}
The proof of this inequality is somewhat lengthy and would interrupt the logical flow of the argument, so we defer it to the appendix (see Lemma \ref{lemmaapendicead}).
\\
Combining \eqref{veform1}, \eqref{veform3}, \eqref{veform2}, and \eqref{veform4}, we obtain, for $ \varepsilon$ and $\hat{h}$ sufficiently small,
\begin{equation}\label{formpreexm}
\begin{split}
    \frac{1}{4h} \int_{\pa E} \vert \pa_{\tau}^2 \psi \vert^2 +\frac{c_g}{2}\int_{\pa E} \vert \pa_{\tau}^3 \psi \vert^2 &\leq \frac{ C(K)}{\sqrt h} \int_{\pa E} \vert \pa_{\tau}^2 \psi \vert^2 + (\frac{c_g}{9}+6 \varepsilon)\int_{\pa E} \vert \pa_{\tau}^3 \psi \vert^2 + C(K).
\end{split}
\end{equation}
Hence, by this and \eqref{tesistep2}, 
we have our claim, i.e., \eqref{tesistep3}. We observe that this claim gives the first inequality of \eqref{tesi18nov}. \\
\textit{Step 3: The bound on the third derivative yields an improved estimate on the first derivative.}

The goal of this step is to establish inequality \eqref{tesi16nov}.
Using the Sobolev embedding, we obtain:
\begin{equation}
    {\rm dist}_{\mathcal{H}}(\pa F, \pa E) = \| \psi \|_{L^\infty(\pa E )} \leq C(K) \big(   \| \psi \|_{L^2(\pa E)}+ \| \pa_\tau \psi \|_{L^2(\pa E)} \big) .
\end{equation}
Combining this with \eqref{tesistep2}, we get:
\begin{equation}
  {\rm dist}_{\mathcal{H}}(\pa F, \pa E) = \| \psi \|_{L^\infty(\pa E )} \leq C(K) \big( h+\| \pa_\tau \psi \|_{L^2(\pa E)} \big) . 
\end{equation}
Therefore, it remains to estimate $ \| \pa_\tau \psi \|_{L^2(\pa E)} $. Differentiating the equation \eqref{eqeulerobvera} once, we obtain:
\begin{equation}\label{asdasd123asd}
    \frac{\pa_{\tau} \psi}{h}+ \frac{\psi^2 \pa_{\tau} \kappa_E}{2h}+ \frac{\kappa_E\psi \pa_{\tau} \psi}{h}= \pa_\tau\big( g(\nu_E) \big)\pa^2_\tau\psi+g(\nu_E)\pa_{\tau}^3 \psi -\pa_{\tau} \kappa_E^\varphi- \pa_{\tau} R_0 \quad\text{ on } \pa E.
\end{equation}
Multiplying this equation by $ \pa_{\tau} \psi$, integrating over $\pa E$, using \eqref{tesistep3},  the fact that $g$ is smooth and the Sobolev embedding, we obtain
\begin{equation}\label{nablanabla}
    \begin{split}
        \frac{\| \pa_{\tau} \psi \|_{L^2(\pa E)}^2}{h} \leq& \| \pa_{\tau}^3 \psi \|_{L^2(\pa E)} \| \pa_{\tau} \psi \|_{L^2(\pa E)}+ \| \pa_{\tau} \kappa_{E}^\varphi \|_{L^2(\pa E)} \| \pa_{\tau} \psi \|_{L^2(\pa E)}\\
        &+ \| \pa_\tau \big( g(\nu_E) \big) \|_{L^\infty(\pa E)} \| \pa_\tau^2 \psi \|_{L^2(\pa E)} \| \pa_\tau \psi \|_{L^2(\pa E)}                  \\
        &+ \| \pa_{\tau}\kappa_E \|_{L^2(\pa E)} \| \pa_{\tau} \psi \|_{L^2(\pa E)} \frac{\| \psi \|_{L^\infty(\pa E)}^2}{h}\\
        & + \frac{\| \psi \|_{L^2(\pa E)}}{h} \| \pa_{\tau} \psi\|_{L^\infty(\pa E)} \| \kappa_{E}\|_{L^\infty(\pa E)} \| \pa_{\tau}\psi \|_{L^2(\pa E)}\\
        &+ \| \pa_{\tau} R_0 \|_{L^2(\pa E)} \| \pa_{\tau} \psi \|_{L^2(\pa E)}\\
         \leq & \| \pa_{\tau} \psi \|_{L^2(\pa E)} \big(C(K)+\| \pa_{\tau} R_0 \|_{L^2(\pa E)}  \big).
    \end{split}
\end{equation}
Since $\| \pa_{\tau} R_0 \|_{L^2(\pa E)}$ is bounded (see \eqref{formutileoggi.1}), inequality \eqref{nablanabla} implies
\begin{equation*}
    \| \pa_{\tau} \psi \|_{L^2(\pa E)} \leq C(K) h,
\end{equation*}
which is precisely the second inequality in \eqref{tesi16nov}.\\
\textit{Step 4: Conclusions: The curvature estimate and $F \in \mathfrak{H}^3_{\hat{K},\sigma_1}(E_0)$.}         

In this final step, we complete the proof of Theorem \ref{mainthm} by establishing the inequalities in the second line of \eqref{tesi18nov}. We begin by noting some immediate consequences of the inequalities in \eqref{tesi16nov} and the first two inequalities in \eqref{tesi18nov}. From the Euler–Lagrange equation \eqref{eq1112}, we immediately deduce that
\[
\|\kappa^\varphi_{ F}\|_{L^\infty(\pa F)}\leq \lambda_F + \frac{1}{2h}\left\|2\psi+ \psi^2 \kappa_E \circ \pi_{\pa E}\right\|_{L^\infty(\pa F)} \leq C(K)+ C(K) \frac{\| \psi \|_{L^\infty(\pa E)}}{h} \leq C(K).
\]
To estimate the first derivative of the curvature, we differentiate the Euler–Lagrange equation \eqref{eq1112}. Recall  that 
$$f = \frac{2\psi+  \psi^2\kappa_E}{2d_{L^2}(F;E)} $$ 
and $x = \pi_{\pa E}(x+ \psi(x)\nu_E(x))$ for all $x \in \pa E$. Applying \eqref{ruffini} with $G= f \circ \pi_{\pa E}$, we obtain
\begin{equation}\label{eq:first.der.curv}
    \begin{split}
        \int_{\pa F} \vert \pa_{\pa F} \kappa_F^{\varphi} \vert^2  &= \frac{d_{L^2}^2(F;E)}{h^2} \int_{\pa F} \vert \pa_{\pa F} (f \circ \pi_{\pa E} )\vert^2 =  \frac{d_{L^2}^2(F;E)}{h^2} \int_{\pa E} \frac{\vert \pa_{\pa E} f \vert^2}{ \sqrt{(1+ \psi \kappa_E)^2+ \vert \pa_{\pa E} \psi \vert^2}}    \\
        &\leq  \frac{d_{L^2}^2(F;E)}{h^2} C(K)\int_{\pa E} \vert \pa_{\pa E} f \vert^2   \leq C(K) \frac{\| \pa_{\pa E}\psi \|_{L^2(\pa E)}^2+ \| \psi \|_{L^\infty(\pa E)}^2}{h^2} \leq C(K),
    \end{split}
\end{equation}
where in the final step we again used \eqref{tesi16nov} and the assumption \eqref{eq:assumption.curvature}.
To complete the estimate for the derivative of the curvature, we differentiate the identity in Remark \ref{remarchinoKvar}, obtaining
\[
\pa_{\pa F} \kappa_{F}^\varphi=
g(\nu_F) \pa_{\pa F} \kappa_F +\nabla g(\nu_F) \nabla\nu_F \tau_F \kappa_F= g(\nu_F) \pa_{\pa F} \kappa_F + (\pa_{\pa F}g )(\nu_F)\kappa_F^2.
\]
It follows that
\[
\frac1c \int_{\pa F} \vert \pa_F \kappa_F \vert^2 \leq \int_{\pa F} \vert g(\nu_F) \pa_F \kappa_F \vert^2 \leq \int_{\pa F} \vert \pa_F \kappa_F^\varphi \vert^2
+C(\varphi,K) \int_{\pa F} \kappa_F^4
\leq C(K),
\]
where the second inequality uses the smoothness of $g$, and the final bound follows from \eqref{eq:first.der.curv} and \eqref{eq:assumption.curvature}. To estimate the second derivative of the curvature, we differentiate the Euler–Lagrange equation \eqref{eq1112} twice, obtaining
\[
\pa^2_{\pa F} \kappa_F^{\varphi} +\frac{d_{L^2}(E;F)}{h} \pa_{\pa F}^2 
( f\circ \pi_{\pa E})=0.
\]
We define $\xi:= \psi + \frac{\psi^2}{2}\kappa_E$, so that $f= \frac{\xi}{d_{L^2}(F;E)}$. Applying the chain rule and using Lemma \ref{lem:l2dist}, we find
\begin{equation}\label{eq:anisotropic.curv.0}
\begin{split}
\| &\pa^2_{\pa F} \kappa_F^{\varphi} \|_{L^2(\pa F)}^2
\leq \frac{1}{h^2} 
\int_{\pa F} \left( | (\pa_{\pa E}^2 \xi) \circ \pi_{\pa E}  | |\pa_{\pa F} \pi_{\pa E}|^2 
+| (\pa_{\pa E} \xi ) \circ \pi_{\pa E}  | |\pa_{\pa F} \pi_{\pa E}| |\pa^2_{\pa F}\pi_{\pa E}|\right)^2 
\\
&=\frac{1}{h^2} 
\int_{\pa E}\left(|(\pa_{\pa E}^2 \xi )  | |(\pa_{\pa F} \pi_{\pa E} )\circ \pi^{-1}_{\pa E}|^2
+ |(\pa_{\pa E} \xi )  | |(\pa_{\pa F} \pi_{\pa E}) \circ \pi_{\pa E}^{-1}|  |\pa^2_{\pa F}\pi_{\pa E} \circ \pi_{\pa E}^{-1}|\right)^2 J_{F,E}
\\
&= \frac{1}{h^2} (I_1+I_2),
\end{split}
\end{equation}
where $I_1$ and $I_2$ are defined in the obvious way, and where $J_{F,E}:= \sqrt{ (1+ \psi \kappa_E)^2+ \vert \pa_{\pa E} \psi \vert^2}$.
If we prove that $|I_1|,|I_2|\leq C(k)h$, 
then \eqref{tesi18nov} follows. To proceed, we differentiate the identity from Remark \ref{remarchinoKvar} twice, yielding
\[
\pa^2_{\pa F} \kappa_{F} = \frac{1}{g(\nu_F)} (\pa^2_{\pa F} \kappa^\varphi_{F} 
-\pa_{\pa F}^2 g (\nu_F) \kappa_F - 2 \kappa_F \pa_{\pa F} g (\nu_F)\pa_{\pa F} \kappa_F).
\]
Since $g$ is bounded from below, the previous estimates on the curvature suffice to show that
\begin{equation} \label{eq:curv.anisotropic}
 \|\pa^2_{\pa F} \kappa_{F}\|_{L^2(\pa F)}^2 \leq C(K)(
\|\pa^2_{\pa F} \kappa^\varphi_{F}\|_{L^2(\pa F)}^2 +1).   
\end{equation}
We now estimate $I_1$; the estimate for $I_2$ follows by a similar argument.

Using the explicit expression for $\xi$ and \eqref{tesilemm}, we obtain
\begin{equation}\label{eq:estimate.I1}
    \begin{split}
|I_1|\leq& C(K)\int_{\pa E}|(\pa_{\pa E}^2 \xi )  |^2|(\pa_{\pa F} \pi_{\pa E} )\circ \pi^{-1}_{\pa E} |^4 
\leq C(K)
\int_{\pa E}\left|\pa^2_{\pa E} \psi +\pa^2_{\pa E}(\kappa_{\pa E}\psi) \right|^2
\\
\leq &  C(K)(  h
+\int_{\pa E}|\pa_{\pa E}\kappa_{ E}|^2| \pa_{\pa E}\psi|^2 
+h^2|\pa^2_{\pa E}\kappa_{ E}|^2)
\\
\leq &  C(K)(h 
+C(K)h
+\|\pa_{\pa E}\kappa_{ E}\|_{L^2}^2 h),
\end{split}
\end{equation}
where we have used the following pointwise estimate (see \cite{JN}):
\begin{equation}\label{nablapaFpiE}
\begin{split}
    &\nabla_{\pa F} \pi_{\pa E}= \nabla \pi_{\pa E}-\nabla \pi_{\pa E} \nu_F \otimes \nu_F \text{ and }\\
    & \nabla \pi_{\pa E}=I -(\nu_{E} \circ \pi_{\partial E})\otimes(\nu_{ E} \circ \pi_{\pa E})-d_E (B_{E} \circ \pi_{\partial E} )(I+ d_E B_E \circ \pi_{\partial E})^{-1}\\
    &  \implies |(\pa_{\pa F} \pi_{\pa E} )\circ \pi^{-1}_{\pa E} | \leq C(K).
    \end{split}
\end{equation}
In the last inequality of \eqref{eq:estimate.I1}, we used the Sobolev embedding and the assumption \eqref{eq:assumption.curvature} and \eqref{tesi16nov}, which imply:
\[\begin{split}
\int_{\pa E}|\pa_{\pa E}\kappa_{ E}|^2| \pa_{\pa E}\psi|^2
&\leq \|\pa_{\pa E}\kappa_{ E} \pa_{\pa E}\psi \|_{L^\infty(\pa E)} 
\|\pa_{\pa E}\psi\|_{L^2(\pa E)}\|\pa_{\pa E}\kappa_E\|_{L^2(\pa E)}
\\
&\leq \frac{C(K)}{\sqrt{h}}\sqrt{h} \|\pa_{\pa E}\psi\|_{L^2(\pa E)}\|\pa_{\pa E}\kappa_E\|_{L^2(\pa E)} \\
& \leq C(K)h.
\end{split}
\]
Therefore, combining \eqref{eq:assumption.curvature} and \eqref{eq:estimate.I1}, we conclude that 
\[
|I_1|\leq C(K) h.
\]
Applying the same reasoning to $I_2$, we get
$$ \vert I_2 \vert \leq C(K) h .$$
The bounds on $I_1$ and $I_2$, together with \eqref{eq:anisotropic.curv.0} and\eqref{eq:curv.anisotropic},
conclude the proof of \eqref{tesi18nov}.
Finally, we observe that, using the Sobolev embedding, \eqref{interHOLDER}, and \eqref{tesi16nov}, \eqref{tesi18nov}, we obtain
$$  \| \psi \|_{C^{1,\frac{1}{4}}(\pa E)} \leq C(K)h^{\gamma} \text{ for some } \gamma>0. $$
Therefore, there exists $\sigma_1$ such that $F \in \mathfrak{H}^3_{\hat{K},\sigma_1}(E_0)$, which concludes the proof of Theorem \ref{mainthm}.
\end{proof}
\section{The Iteration procedure}
In this section, we prove a crucial iteration formula.  Let $A \in \mathfrak{C}^1_{M, \sigma_{E_0}}(E_0)$, for some $M>0$. If $f$ is a smooth function on $\pa A$, there exists a constant $C(M)$ such that 
\begin{equation}\label{eleminter}
    \| f \|_{L^\infty(\pa A)}^2 \leq C(M) \big( \frac{1}{\varepsilon} \| f \|_{L^2(\pa A)}^2+  \varepsilon \| \pa_{\pa A} f \|_{L^2(\pa A)}^2  \big)
\end{equation}
for every $\varepsilon \in (0,1)$.
Before explaining the iteration algorithm, we present a technical lemma that supports the iterative procedure.
\begin{lemma}\label{lemmaiter1}
    Let $F,\, E$ and $\psi: \pa E \rightarrow \R$ be as in Theorem \ref{mainthm}. There exists $C_0 = C_0(K)$ such that
\begin{equation}\label{tesilemmaiter1}
       (1-h C_0) \int_{\pa E} | \pa_{\tau} \xi \vert^2 +\frac{2}{3}h \int_{\pa E} g(\nu_E)\vert \pa^2_\tau \psi \vert^2  \leq -h \int_{\pa E} \pa_{\tau} \kappa_E^\varphi \pa_{\tau} \xi,
    \end{equation}
    where $\xi:=\psi+ \frac{ \kappa_E \psi^2}{2}$.
\end{lemma}
\begin{proof}
   We multiply \eqref{asdasd123asd} by $ \pa_{\tau} \xi$ to obtain
   \begin{equation}\label{gia3}
   \begin{split}
   \int_{\pa E}\frac{|\pa_{\tau} \xi|^2}{h}= -\int_{\pa E} \pa_{\tau} \kappa_E^\varphi \pa_{\tau} \xi + \int_{\pa E}\pa_\tau\big( g(\nu_E) \big)\pa^2_\tau\psi \pa_\tau \xi+\int_{\pa E}g(\nu_E) \pa_{\tau} \xi \pa_{\tau}^3 \psi -\int_{\pa E} \pa_{\tau} \xi \pa_{\tau} R_0 .
   \end{split}
   \end{equation}
   Integrating the above equation by parts once, we get
   \begin{equation}\label{1906estim1}
     \int_{\pa E}\frac{|\pa_{\tau} \xi|^2}{h} =  -\int_{\pa E} \pa_{\tau} \kappa_E^\varphi \pa_{\tau} \xi + \int_{\pa E}\pa_\tau\big( g(\nu_E) \big)\pa^2_\tau\psi \pa_\tau \xi - \int_{\pa E} \pa_\tau \big( g(\nu_E) \pa_\tau \xi  \big) \pa_\tau^2 \psi+ \int_{\pa E} \pa_\tau^2 \xi R_0. 
   \end{equation}
   We aim to estimate the following integrals:
   \begin{align}
       &\int_{\pa E}\pa_\tau\big( g(\nu_E) \big)\pa^2_\tau\psi \pa_\tau \xi, \label{1906form1}\\
       &\int_{\pa E} -\pa_\tau \big( g(\nu_E) \pa_\tau \xi  \big) \pa_\tau^2 \psi, \label{1906form2}\\
       &\int_{\pa E} \pa_\tau^2 \xi R_0 \label{1906form3} .
   \end{align}
Let $\varepsilon>0$ be fixed (to be chosen later). We now compute $\pa^2_\tau \xi$, which will be used several times in the sequel:
\begin{equation}
    \pa_\tau^2 \xi= \pa_\tau^2 \psi \big( 1 + \psi \kappa_E \big) + \pa_\tau^2 \kappa_E \frac{ \psi^2}{2}+ \pa_\tau \psi \pa_\tau \kappa_E \psi + \frac{\vert \pa_\tau \psi \vert^2}{2} \kappa_E.
\end{equation}
		\textit{Estimate of \eqref{1906form1}. }
 Since $g\in C^2$, \eqref{eq:assumption.curvature}, as well as the Cauchy–Schwarz and Young’s inequalities, we obtain
\begin{equation}\label{1906estim2}
           \int_{\pa E}\pa_\tau\big( g(\nu_E) \big)\pa^2_\tau\psi \pa_\tau \xi \leq C(K,\varepsilon) \| \pa_\tau \xi \|^2_{L^2(\pa E)}+  \varepsilon \| \pa_\tau^2 \psi \|_{L^2(\pa E)}^2. 
        \end{equation}
\textit{Estimate of \eqref{1906form2}. }        \\
By the Leibniz rule, we obtain
\begin{equation}\label{1906estim3}
    \int_{\pa E} -\pa_\tau \big( g(\nu_E) \pa_\tau \xi  \big) \pa_\tau^2 \psi= \int_{\pa E} -g(\nu_E )\pa_\tau^2 \xi \pa^2_\tau \psi - \pa_\tau \big( g(\nu_E)   \big)\pa_\tau \xi \pa_\tau^2 \psi.
\end{equation}
Using the regularity of $g$, formula \eqref{eq:assumption.curvature}, the Cauchy–Schwarz inequality, and Young's inequality, we deduce
\begin{equation}\label{1906estim4}
    \int_{\pa E} - \pa_\tau \big( g(\nu_E)   \big)\pa_\tau \xi \pa_\tau^2 \psi \leq C(K,\varepsilon) \| \pa_\tau \xi \|^2_{L^2(\pa E)}+ \varepsilon \| \pa_\tau^2 \psi \|_{L^2(\pa E)}^2.
\end{equation}
Using the definition of $\xi$, the regularity of $g$, formulas \eqref{eq:assumption.curvature}, \eqref{tesi16nov}, \eqref{tesi18nov}, the Cauchy–Schwarz and Young's inequalities, as well as the Sobolev embedding, we deduce
\begin{equation}\label{19061038}
\begin{split}
   &\int_{\pa E} -g(\nu_E )\pa_\tau^2 \xi \pa^2_\tau \psi= \int_{\pa E}- g(\nu_E) \big( \pa_\tau^2\psi( 1+ \psi \kappa_E)+ \pa_\tau^2 \kappa_E \frac{\psi^2}{2}+ \pa_\tau \psi \pa_\tau \kappa_E \psi+ \frac{\vert \pa_\tau^2 \psi \vert^2}{2} \kappa_E   \big) \pa_\tau^2 \psi\\
   & \leq -\frac{8}{9}\int_{\pa E} g(\nu_E) \vert \pa^2_\tau \psi \vert^2 + C(K) \| \pa_\tau^2 \kappa_E \|_{L^2(\pa E)} \| \psi \|_{L^\infty(\pa E)}^2 \| \pa_\tau^2 \psi \|_{L^2(\pa E)}\\
   & \quad C(K) \| \pa_\tau \kappa_E \|_{L^\infty(\pa E)} \| \psi \|_{L^\infty(\pa E)} \| \pa_\tau \psi \|_{L^2(\pa E)} \| \pa_\tau^2 \psi \|_{L^2(\pa E)} + C(K) \| \pa_\tau \psi \|_{L^2(\pa E)} \| \pa^2_\tau \psi \|_{L^2(\pa E)}\\
   & \leq -\frac{8}{9}\int_{\pa E} g(\nu_E) \vert \pa^2_\tau \psi \vert^2 \\
   & \qquad+ C(K)\bigg(\frac{h}{\sqrt{h}}  \| \psi \|_{H^1(\pa E)} \| \pa_\tau^2 \psi \|_{L^2(\pa E)} + \big(\frac{h}{\sqrt{h}}+1\big) \| \pa_\tau \xi \|_{L^2(\pa E)} \| \pa_\tau^2 \psi \|_{L^2(\pa E)} \bigg)\\
   & \leq -\frac{8}{9}\int_{\pa E} g(\nu_E) \vert \pa^2_\tau \psi \vert^2 + 2 \varepsilon \| \pa_\tau^2 \psi \|_{L^2(\pa E)}^2 + C(K,\varepsilon) \| \pa_\tau \xi \|^2_{L^2(\pa E)},
   \end{split}
\end{equation}
where in the first inequality, we have used the estimate $ \| \psi \kappa_E \|_{L^\infty(\pa E)} \leq h C(K) \leq \frac{1}{9}$, which holds for $h$ sufficiently small, and the second and third inequalities follow from Lemma~\ref{lemm1}.
\\
\textit{Estimate of \eqref{1906form3}. }  

We recall the definition of $R_0$, as given in formula \eqref{interpol1}:
\begin{equation}
R_0:=a(\nu_E,\psi\kappa_E,\pa_{\tau} \psi) \pa_{\tau}^2 \psi+ b(\nu_E,\psi \kappa_E,\pa_{\tau}\psi) \pa_{\tau} (\psi \kappa_E)+ c(\nu_E,\psi,\pa_{\tau} \psi,\kappa_E),
    \end{equation}
    where $a,b\in C^{\infty},\, c \in C^{\infty}$ are smooth functions satisfying
\begin{equation}
b(\cdot,0,0)=a(\cdot, 0,0)=c(\cdot,0,0,\cdot)=0. 
\end{equation}
We aim to estimate $c(\nu_E, \psi, \pa_\tau \psi,\kappa_E)$.
We recall that $ \|\psi\|_{H^2(\pa E)} \leq C(K) $. Hence, by the Sobolev embedding theorem, we have $ \| \psi \|_{C^{1, \frac{1}{2}}(\pa E)} \leq C(K)$. Using \eqref{interHOLDER}, we get
\begin{equation}
    \| \psi \|_{C^1(\pa E)} \leq C \| \psi \|_{C^{1,\frac{1}{2}}(\pa E)}^\theta \| \psi \|_{C^0(\pa E)}^{1- \theta} \leq C(K) \| \psi \|^{1- \theta}_{H^1(\pa E)} \leq C(K) h^{1-\theta}.
\end{equation}
Recalling that $c$ is smooth and satisfies $ c(\cdot, 0,0,\cdot)=0$, we apply the Taylor expansion of $c$. For sufficiently small $h$, this yields 
\begin{equation}\label{19061050}
    c(\nu_E, \psi, \pa_\tau \psi,\kappa_E) \leq C(K )\| \psi \|_{C^1(\pa E)} \leq \varepsilon C(K) \| \pa_\tau^2 \psi \|_{L^2(\pa E)}+ C(K,\varepsilon) \| \pa_\tau \xi \|_{L^2(\pa E)},
\end{equation}
where the final inequality follows from \eqref{eleminter} and Lemma \ref{lemm1}. Applying the same reasoning used in formula \eqref{19061038}, and making use of \eqref{19061050}, we obtain
\begin{equation}\label{casa1}
    \int_{\pa E }\pa^2_\tau \xi c \leq \varepsilon C(K) \| \pa_\tau^2 \psi \|_{L^2(\pa E)}^2 + C(K,\varepsilon) \| \pa_\tau \xi \|^2_{L^2(\pa E)}.
\end{equation}
Using the same argument as in formula \eqref{19061038}, together with Lemma \ref{lemm1}, the Cauchy–Schwarz and Young's inequalities, and the Sobolev embedding, we obtain
\begin{equation}\label{casa2}
\begin{split}
    \int_{\pa E} \pa_\tau^2 \xi \pa_\tau &( \psi \kappa_E) b(\nu_E, \psi \kappa_E, \pa_\tau \psi ) \leq C(K) \| \pa^2_\tau \xi \|_{L^2(\pa E)} \| \pa_\tau (\psi \kappa_E) \|_{L^2(\pa E)}\\
    & \leq C(K) \big( \| \pa_\tau^2 \psi \|_{L^2(\pa E)}+ \| \pa_\tau \xi \|_{L^2(\pa E)} \big) \big( \| \pa_\tau \kappa_E \|_{L^2(\pa E)} \| \psi \|_{L^\infty(\pa E)}+ \| \pa_\tau \psi \|_{L^2(\pa E)} \big)\\
    & \leq \varepsilon \| \pa^2_\tau \psi \|_{L^2(\pa E)}^2+C(K,\varepsilon)\| \pa_\tau \xi \|_{L^2(\pa E)}^2.
    \end{split}
\end{equation}
Arguing as in formula \eqref{19061038}, together with the formulas in \eqref{160625pom1}, we obtain
\begin{equation}\label{casa3}
\begin{split}
    \int_{\pa E} & \pa^2_\tau \xi a(\nu_E,\psi\kappa_E, \pa_\tau \psi) \pa^2_\tau \psi \\
    &\leq \big(  \| \pa_\tau^2 \psi \|_{L^2(\pa E)}+ \| \pa_\tau \xi \|_{L^2(\pa E)}\big) \| a(\nu_E,\psi\kappa_E, \pa_\tau \psi) \|_{L^\infty(\pa E)} \| \pa_\tau^2 \psi \|_{L^2(\pa E)} \\
    & \leq \big(\frac{c_g}{9}+ \varepsilon \big) \| \pa_\tau^2 \psi \|_{L^2(\pa E)}^2+ C(K,\varepsilon) \| \pa_\tau \xi \|^2_{L^2(\pa E)},
    \end{split}
\end{equation}
where $c_g= \min_{\nu \in \mathcal{S}^1} g(\nu)$.
Combining \eqref{casa1}, \eqref{casa2}, and \eqref{casa3}, we obtain
\begin{equation}\label{1906estim5}
  \int_{\pa E} \pa_\tau^2 \xi R_0  \leq \big( \frac{c_g}{9}+ \varepsilon C(K) \big) \| \pa_\tau^2 \psi \|^2_{L^2(\pa E)} + C(K,\varepsilon) \| \pa_\tau \xi \|^2_{L^2(\pa E)}.   
\end{equation}
Therefore, combining \eqref{1906estim1}, \eqref{1906form1}, \eqref{1906form2}, \eqref{1906form3}, \eqref{1906estim2}, \eqref{1906estim3}, and \eqref{1906estim5}, we obtain
\begin{equation}
\begin{split}
    \int_{\pa E} \frac{\vert \pa_\tau \xi \vert^2}{h} &\leq - \int_{\pa E } \pa_\tau \kappa_E^\varphi \pa_\tau \xi + C(K,\varepsilon) \| \pa_\tau \xi \|^2_{L^2(\pa E)}\\
    & \qquad + \big( \frac{c_g}{9}+ C(K)\varepsilon \big) \| \pa_\tau^2 \psi \|_{L^2(\pa E)}^2 - \frac{8}{9} \int_{\pa E} g(\nu_E) \vert \pa^2_\tau \psi \vert^2 ,
    \end{split}
\end{equation}
   which, for $\varepsilon$ sufficiently small, implies \eqref{tesilemmaiter1}.
\end{proof}
\subsection{Set up of the algorithm}
Since we need to employ the procedure outlined at the beginning of Section \ref{sec:esistenza}, we now describe how the algorithm operates in two successive steps.

For every set $A\subset \R^2$ of class $C^4$, we define the function
\begin{equation}
    Q_h(A)
    := \| \kappa_A \|_{L^\infty(\pa A)}+\|\pa_{\pa A}\kappa_A \|_{L^2(\pa A)}
     + \sqrt h
    \| \pa_{\pa A}^2\kappa_A \|_{L^2(\pa A)}.
\end{equation}
Let $K>0$, and let $h_1, \, K_1,\delta_1$  be the constants given by Theorem \ref{mainthm} (i.e., the constants $\hat h$, $\hat K$ and $\delta_4$  in Theorem \ref{mainthm}). Let $E \in \mathfrak{H}^{3}_{K,\sigma_{E_0}}(E_0)$ be such that $Q_{h_1}(E)\leq K$ and $\vert E \vert=1$. Let $h \leq h_1$, and let $F \subset\R^2$ be a minimizer of \begin{equation}
		\min\big\{\mathcal F_{h}(A,E):  A \in \mathcal{B}_{\delta_1}(E),\, \vert A \vert=1\big\}.
	\end{equation}
By Theorem \ref{mainthm}, the set $F$ is also of class $C^4$ and satisfies the following properties:
\begin{equation}\label{glablig1}
\begin{split}
   & F \in \mathfrak{H}^3_{K_1 , \sigma_E}(E),\,\,\, d_{\mathcal{H}}(\pa E, \pa F) \leq K_1 h ,\,\,\,Q_{h}(F) \leq K_1 \,\,\,\text{ for all } h \in (0,h_1], \\
   & \pa F = \{ x+ \psi_{F , E}(x)\nu_{E}(x)\colon\, x \in \pa E \},\\
   & \|  \psi_{F , E} \|_{H^1(\pa E)} \leq K_1 h,\, \, \| \psi_{F,E} \|_{H^2(\pa E)} \leq K_1 \sqrt{h} .
\end{split}    
\end{equation}
Let $h_2, \, K_2,\delta_2$ be given by Theorem \ref{mainthm} (with $K=K_1$). Let $h \leq h_2$, and let $G \subset\R^2$ be a minimizer of \begin{equation}
		\min\big\{\mathcal F_{h}(A,F):  A \in \mathcal{B}_{\delta_2}(F),\,\,\,\vert A \vert=1\big\}.
	\end{equation}
By Theorem \ref{mainthm}, the set $G$ is of class $C^4$
and satisfies the following properties:
\begin{equation}\label{glablig2}
\begin{split}
    & G \in \mathfrak{H}^3_{K_2 , \sigma_F}(F),\,\,\,d_{\mathcal{H}}(\pa F, \pa G) \leq K_2 h ,\,\,\,  Q_h(G) \leq K_2 \,\,\,\text{ for all } h \in (0,h_2] \\
   & \pa G = \{ x+ \psi_{G, F}(x)\nu_{F}(x)\colon\, x \in \pa F \},\\
   & \|  \psi_{G, F} \|_{H^1(\pa F)} \leq K_2 h,\, \|  \psi_{G, F} \|_{H^2(\pa F)} \leq K_2 \sqrt{h}.
\end{split}
\end{equation}

\begin{lemma}\label{lemmaiter2}
    Let $K>0$, and define $\delta_3 := \min \{ \delta_1, \delta_2\}$ where $\delta_1$ and $ \delta_2$ are as above. For every $\delta \leq \delta_3$, there exists $\tilde{h}$ such that if $ h \in (0, \tilde h]$ and $ E, F, G$ are as in \eqref{glablig1},\eqref{glablig2}, then 
\begin{equation}\label{tesiiter2}
    \int_{\pa F}  \vert \pa_{\pa F} \xi_{G,F} \vert^2+ \frac{h}{2}g(\nu_{F}) \vert \pa_{\pa F}^2 \psi_{G,F} \vert^2  \leq (1+ C_1 h) \int_{\pa E} \vert \pa_{\pa E} \xi_{F, E} \vert^2
\end{equation}
for some constant $C_1= C_1(K)$, where $\xi_{G,F}= \psi_{G,F} + \kappa_{F}\frac{ \psi_{G,F}^2}{2} $, $\xi_{F,E}= \psi_{F,E} + \kappa_{E}\frac{ \psi_{F,E}^2}{2} .$
\end{lemma}
\begin{proof}
In what follows, we denote by $C$  a generic constant depending on $K$. We define
\[\
\widehat{\kappa_F^\varphi} \,\colon \pa E \to \R, \quad \widehat{\kappa_F^\varphi}(x):= \kappa_F (x+ \psi_{F,E}(x)\nu_E(x)),
\]
and similarly
\[
\widehat{\xi_{G,F}}\, \colon \pa E \to\R,\quad\widehat{\xi_{G,F}}(x):= \xi_{G,F} (x+ \psi_{F,E}(x)\nu_E(x)).
\]
By applying Lemma~\ref{lemmaiter1} to the sets $F$ and $G$, we obtain 
\begin{equation}\label{duue1}
        (1-hC_0)\int_{\pa F} | \pa_{\pa F} \xi_{G,F} \vert^2+ \frac{2}{3}h \int_{\pa F} g(\nu_F) \vert \pa_{\pa F}^2 \psi_{G,F} \vert^2 \leq -h \int_{\pa F} \pa_{\pa F} \kappa_F^\varphi \pa_{\pa F} \xi_{G,F},
      \end{equation}
    where $C_0=C_0(K)$. Finally, we define
    $$
    J_{F,E}:= \sqrt{(1+ \psi_{F,E} \kappa_E)^2+ \vert \pa_{\pa E} \psi_{F,E}\vert^2}.$$
Using the parallelogram identity, Lemma \ref{eq:transform.gradient}, Young’s inequality, and the Taylor expansion of the function $t \rightarrow \frac{1}{\sqrt{1+t}}$, we obtain
\begin{equation}\label{200625form1}
\begin{split}
    -h \int_{\pa F} \pa_{\pa F} \kappa_F^\varphi \cdot \pa_{\pa F} \xi_{G,F} &= -h \int_{\pa E} \frac{ \pa_{\pa E} \widehat{\kappa_F^\varphi}  \pa_{\pa E} \widehat{\xi_{G,F}}}{ J_{F,E}} d \mathcal{H}^1\\
    &= -h \int_{\pa E} \pa_{\pa E} \widehat{\kappa_F^\varphi}  \pa_{\pa E} \widehat{\xi_{G,F}} -h \int_{\pa E}\bigg(\frac{1}{J_{F,E}}-1 \bigg)\pa_{\pa E} \widehat{\kappa_F^\varphi} \cdot \pa_{\pa E} \widehat{\xi_{G,F}}\\
    &  \leq -h\int_{\pa E} \pa_{\pa E} \widehat{\kappa_F^\varphi}  \pa_{\pa E} \widehat{\xi_{G,F}} + \varepsilon h C\int_{\pa F} \vert \pa_{\pa F} \xi_{G,F} \vert^2 \\
        &  \,\,\,\qquad \qquad + \frac{C}{\varepsilon} h \int_{\pa E} \vert \pa_{\pa E} \widehat{\kappa_F^\varphi} \vert^2 \big( \psi_{F,E}^2+ \psi_{F,E}^4 + \vert \pa_{\pa E} \psi_{F,E} \vert^4   \big). 
    \end{split}
\end{equation}
To estimate the last integral, we apply \eqref{perc}, \eqref{eleminter}, \eqref{glablig1},\eqref{glablig2}, Lemma \ref{lemm1}, as well as the Sobolev embedding and Poincaré inequality, and assuming $h$ is sufficiently small compared to $\varepsilon$, we have that
\begin{equation}\label{08052025form2}
\begin{split}
\frac{Ch}{\eps}&\int_{\pa E }|\pa_{\pa E} \widehat{\kappa_F^\varphi}|^2 (\psi_{F,E}^2+\psi_{F,E}^4+|\pa_{\pa E} \psi_{F,E}|^4) \, d \mathcal{H}^1
\\
&\leq \frac{Ch}{\eps}\big(\|\psi_{F,E}\|_{L^\infty(\pa E)}^2+\|\psi_{F,E}\|_{L^\infty(\pa E)}^4+\| \pa_{\pa E} \widehat{\kappa^\varphi_E} \|_{L^\infty(\pa E)}^2\|\pa_{\pa E}\psi_{F,E}\|_{L^\infty(\pa E)}^2 \| \pa_{\pa E} \psi_{F,E} \|^2_{L^2(\pa E)}\big)\cr
& \leq \frac{Ch}{\eps}\Big[\frac{1}{\eps^2}\|\psi_{F,E}\|_{L^2(\pa E)}^2+\eps^2\|\pa_{\pa E}\psi_{F,E}\|_{L^2(\pa E)}^2+ \frac{C}{h} h \| \pa_{\pa E} \psi_{F,E} \|^2_{L^2(\pa E)} \Big]\\
& \leq \frac{Ch}{\eps}\Big[\frac{1}{\eps^2}\|\xi_{F,E}\|_{L^2(\pa E)}^2+C\|\pa_{\pa E}\xi_{F,E}\|_{L^2(\pa E)}^2 \Big]\\
&\leq h C \| \pa_E \xi_{F,E} \|^2_{L^2(\pa E)} .
\end{split}
\end{equation}
Substituting \eqref{08052025form2} into \eqref{200625form1}, we obtain
\begin{equation}\label{nador10}
\begin{split}
 -h \int_{\pa F} \pa_{\pa F} \kappa_F^\varphi  \pa_{\pa F} \xi_{G,F}
 &\leq  -h\int_{\pa E} \pa_{\pa E} \widehat{\kappa_F^\varphi}  \pa_{\pa E} \widehat{\xi_{G,F}} + \varepsilon h C\int_{\pa F} \vert \pa_{\pa F} \xi_{G,F} \vert^2  \\
 & \qquad  + h C \| \pa_{\pa E} \xi_{F,E} \|^2_{L^2(\pa E)} .
 \end{split}
\end{equation}
Now, by differentiating the Euler–Lagrange equation \eqref{eqeulerobvera}, we get
\begin{equation}\label{nador2}
    \frac{\pa_{\pa E} \xi_{F,E}}{h}= - \pa_{\pa E} \widehat{\kappa_F^\varphi} \quad \text{in }\pa E.
\end{equation}  
We aim to estimate
\begin{equation}
    -h\int_{\pa E} \pa_{\pa E} \widehat{\kappa_F^\varphi}  \pa_{\pa E} \widehat{\xi_{G,F}}.
\end{equation}
Using \eqref{nador2}, along with the Cauchy–Schwarz and Young's inequalities, we obtain
\begin{equation}\label{nador11}
\begin{split}
     -h\int_{\pa E} \pa_{\pa E} \widehat{\kappa_F^\varphi}  \pa_{\pa E} \widehat{\xi_{G,F}} &= \int_{\pa E} \pa_{\pa E} \xi_{F,E} \pa_{\pa E} \widehat{\xi_{G,F}} \leq \frac{1}{2} \| \pa_{\pa E} \xi_{F,E} \|^2_{L^2(\pa E)}+ \frac{1}{2} \| \pa_{\pa E} \widehat{\xi_{G,F}} \|_{L^2(\pa E)}^2\\
     & \leq \frac{1}{2} \| \pa_{\pa E} \xi_{F,E} \|^2_{L^2(\pa E)}+ \frac{1}{2} \Bigg[ \int_{\pa E} \frac{\vert \pa_{\pa E} \widehat{\xi_{G,F}} \vert^2}{J_{F,E}}+ \int_{\pa E} \big( 1- \frac{1}{J_{F,E}}\big)\vert \pa_{\pa E} \widehat{\xi_{G,F}} \vert^2  \Bigg]\\
     & \leq \frac{1}{2} \| \pa_{\pa E} \xi_{F,E} \|^2_{L^2(\pa E)}+ \frac{1}{2} \| \pa_{\pa F} \xi_{G,F} \|^2_{L^2(\pa F)}\\
     & \quad+ \int_{\pa E} \kappa_E\psi_{F,E} \vert \pa_{\pa E} \widehat{\xi_{G,F}} \vert^2+ \frac{1}{2}\int_{\pa E} (\vert\kappa_E \psi_{F,E}\vert^2+ \vert \pa_{\pa E} \psi_{F,E} \vert^2 )\vert \pa_{\pa E} \widehat{\xi_{G,F}} \vert^2 \\
     & \leq \frac{1}{2} \| \pa_{\pa E} \xi_{F,E} \|^2_{L^2(\pa E)}+ \frac{1}{2} \| \pa_{\pa F} \xi_{G,F} \|^2_{L^2(\pa F)}+ h C(K) \int_{\pa E} \vert \pa_{\pa E} \widehat{\xi_{G,F}} \vert^2
     \\
     & \leq \frac{1}{2} \| \pa_{\pa E} \xi_{F,E} \|^2_{L^2(\pa E)}+ \big(\frac{1}{2}+ hC(K)\big) \| \pa_{\pa F} \xi_{G,F} \|^2_{L^2(\pa F)},
     \end{split}
\end{equation}
where in the fourth inequality we used the Sobolev embedding. By combining \eqref{nador10}, \eqref{nador11}, and \eqref{duue1}, we obtain \eqref{tesiiter2}.
\end{proof}

\section{Proof of the main theorem}
In this section, we use the iterative estimates established previously to prove that the approximate constrained flat flow (see Definition \ref{12092023def1}) converges to the classical solution of \eqref{MAINEQsol} as $h \rightarrow 0$.
\begin{theorem}\label{thmausilmain}
		 There exist constants $\hat T,\hat C,\hat \delta,\hat \sigma$ with the following property: for every $\delta < \hat \delta$, there exists $\tilde{h}$ such that $E^{h,\delta}_t \in \mathfrak{H}^3_{\hat C, \sigma_1}(E_0)$, i.e.,
		$$
			\pa E^{h,\delta}_t= \{x+ f^{h,\delta}(t,x)\nu_{E_0}(x) \, \colon x \in \pa E_0    \}, \, \| f^{h,\delta} \|_{H^3(\pa E_0)} \leq \hat C,\, \| f^{h,\delta} \|_{L^\infty(\pa E_0)} \leq \sigma_1,
	$$
		for all $t \in [0,T_0]$ and $ 0 < h \leq \tilde{h}$, where $ \{ E^{h,\delta}_t \}_{t \geq0}$ is an approximate constrained flat flow starting from $E_0$.
		
		Moreover, the functions $f^{h,\delta}$ converge in 
		$L^\infty([0,T_0], H^3(\pa E_0))$ to a function $f^\delta$, and the associated family $ \{ E_t^\delta \}_{t \in [0,T_0]}$, which is characterized by
		\begin{equation}
		 	\pa E_t^\delta= \{  x + f^\delta(t,x) \nu_{E_0}(x) : x \in \pa E_0 \}, \,\, E_0 \Delta E_t^\delta \subset {\rm cl}\big(I_{\sigma_1}(E_0)\big)
		\end{equation}
		is a classical solution of the problem \eqref{MAINEQsol} on the interval $[0,T_0]$.
\end{theorem}
\begin{proof}
  In the proof of the theorem, we will omit to explicitly mention “up to subsequences” for the sake of brevity, except where it is strictly necessary for clarity.  Let $\hat \delta < \delta_3$ where $\delta_3$ is the constant from Lemma \ref{lemmaiter2}, and fix $\delta \leq \hat{\delta}$. Let $ \{ E_{hk}^{h,\delta}\}_{k \in \N} $ be an approximate constrained flat flow starting from $E_0$; see Definition \ref{12092023def1}. To simplify the notation, we write $E_k= E_{hk}^{h,\delta}$ for $k \geq 0$. 
  
  We are now in a position to apply Theorem \ref{mainthm}, which yields
 \begin{equation}\label{12052025form-1}
			\begin{split}
				&\pa E_1= \{  x +\psi_1(x)\nu_{E_0}(x)\colon x \in \pa E_0 \},\\
				&\| \psi_1 \|_{H^1(\pa E_0)} \leq K_0 h, \quad \| \psi_1 \|_{H^3(\pa E_0)} \leq K_0, \\
				&  \|  \kappa_{E_1}^{\varphi} \|_{H^1(\pa E_1)} \leq K_0, \quad  \| \pa_{\pa E_1}^2 \kappa_{E_1}^\varphi \|_{L^2(\pa E_1)} \leq \frac{K_0}{ \sqrt{h} }.
			\end{split} 
		\end{equation}
        Moreover, using the interpolation inequality (see Proposition \ref{interptensor}), we obtain
 \begin{equation}\label{12052025form-1/2}
            \|\pa^2_{E_0} \psi_1 \|_{L^2(\pa E_0)} \leq K_0 \sqrt{h}. 
        \end{equation}
        Let $k_0 \in \N$ be the largest index such that
		\begin{equation}
			\pa E_k \subset {I}_{\delta }(\pa E_0) \quad \forall k \leq k_0.
		\end{equation}
		We set $T_0:= k_0 h$.\\
        	\textit{Claim 1:} For every $k \leq k_0$, the following holds:
		\begin{equation}\label{12052025form1}
			\|  \pa_{\pa E_k}\kappa_{E_k}^{\varphi} \|_{L^2(\pa E_k)} \leq K_0, \quad  \| \pa_{\pa E_k}^2 \kappa_{E_k}^\varphi \|_{L^2(\pa E_k)} \leq \frac{K_0}{ \sqrt{h} }.
		\end{equation}
				We prove \eqref{12052025form1} by induction. Equation \eqref{12052025form-1} establishes the base case $k=1$. Now suppose that the claim holds for all integers up to $k-1$. Applying Theorem \ref{mainthm}, we deduce that
        	\begin{equation}\label{12052025form2}
			\begin{split}
				&\pa E_{k}= \{  x +\psi_{k}(x)\nu_{E_{k-1}} (x)\colon \, x \in \pa E_{k-1} \},\\
				&\| \psi_{k} \|_{H^1(\pa E_{k-1})} \leq L_1 h, \quad \| \psi_{k} \|_{H^3(\pa E_{k-1})} \leq L_1.
			\end{split} 
		\end{equation}
        For each $j \geq 0$, define $\xi_j:= \xi_{E_{j}, E_{j-1}}$ as in \eqref{funzione}. Applying Lemma \ref{lemmaiter2}, we obtain
        \begin{equation}
            \int_{\pa E_{j-1}}  \vert \pa_{\pa E_{j-1} } \xi_{j} \vert^2+ \frac{h}{2}g(\nu_{E_{j-1}}) \vert \pa_{\pa E_{j-1}}^2 \psi_{j} \vert^2  \leq (1+ C_1 h) \int_{\pa E_{j-2}} \vert \pa_{\pa E_{j-2}} \xi_{j-1} \vert^2
        \end{equation}
        for every $1 \leq j \leq k$. Iterating this inequality and using \eqref{12052025form-1} and \eqref{12052025form-1/2}, we find that
        \begin{equation}
            \begin{split}
             	\int_{\pa E_{k-1}} \big( \vert \pa_{E_{k-1}}\xi_{k} \vert^2+ \frac{h}{4}\sum_{j=1}^k g(\nu_{E_{j-1}}) \vert \pa^2_{\pa E_{j-1}} \psi_{j} \vert^2 \big) & \leq (1+C_1 h )^{k-1} \int_{\pa E_{0}} \vert \pa_{\pa E_0}\xi_{1}\vert^2\\
				& \leq e^{2C_1 h k} K_0^2 h^2 \\
				& \leq e^{2C_1 h k_0} K_0^2 h^2 \leq 2 K_0^2 h^2, 
            \end{split}
        \end{equation}
        where we used the fact that $h k_0 = T_0$ and that $T_0$ is sufficiently small. Possibly increasing the constant $L_0$, we conclude that
        \begin{equation}\label{zizz2}
            \|  \pa_{\pa E_{k-1}} \xi_k \|^2_{L^2(\pa E_k)}+ \frac{h}{4}\sum_{i=1}^k \int_{\pa E_{i-1}} g(\nu_{E_{i-1}}) \vert \pa^2_{\pa E_{i-1}} \psi_i \vert^2 \leq L_0^2 h^2.
        \end{equation}
        Therefore, by repeating the argument from Step 4 of Theorem \ref{mainthm}, we obtain the desired estimate \eqref{12052025form1}, completing the proof of the claim.\\
        \textit{Claim 2:} $T_0>0$. By definition of $k_0$,
         there exists a point $x_0 \in \pa E_{k_0}$ such that $$\mathrm{dist}(x_0, \pa E_0)\geq  \frac{\delta}{2}.$$ The set $E_{k_0}$ satisfies the assumptions of Lemma \ref{lemmalambdaminz}, and is therefore an almost minimizer of the $\varphi$-perimeter with a constant $\omega_0$ that is independent of $h$. Consequently, for $E_{k_0}$ the density estimates are satisfied, both for the perimeter and for the volume; see \cite{Bombieri1982}. Using these density estimates together with the assumption $\mathrm{dist}(x_0,E_0)\geq  \frac{\delta}{2}$, we deduce that
		\begin{equation}
			\vert E_{k_0} \Delta  E_0 \vert \geq c \delta^2,
		\end{equation}
        for some constant $c$ depending on $\omega_0$.
		Now, combining this inequality with \eqref{zizz2} and the triangle inequality, we obtain
        	\begin{equation}
			\begin{split}
				c \delta^2 &\leq \vert E_{k_0} \Delta  E_0 \vert \leq \sum_{j=1}^{k_0} \vert E_j \Delta E_{j-1} \vert = \sum_{j=1}^{k_0} \| \xi_j \|_{L^1(\pa E_{j-1})}\leq P(E_{j-1})^{\frac{1}{2}} \sum_{j=1}^{k_0} \| \xi_j \|_{L^2(\pa E_{j-1})}\\
				&\leq C_\varphi P_\varphi(E_{j-1})^{\frac{1}{2}} \sum_{j=1}^{k_0} \| \xi_j \|_{L^2(\pa E_{j-1})} \leq  C_\varphi P_\varphi(E_0)^{\frac{1}{2}} C(L_0)\sum_{j=1}^{k_0} \| \pa_{\pa E_{j-1}} \xi_j \|_{L^2(\pa E_{j-1})} \\
				& \leq C_\varphi P_\varphi(E_0)^{\frac{1}{2}} C(L_0) L_0 k_0 h = C_\varphi P_\varphi(E_0)^{\frac{1}{2}} C(L_0)  L_0 T_0,
			\end{split}
		\end{equation}
		where we have used the Poincar\'{e} inequality, along with the fact that  $P_\varphi(E_j) \leq  P_\varphi (E_0)$, which follows from the minimizing movements scheme.\\
\textit{Claim 3:} There exist constants $\hat C,\sigma_1>0$ such that 
		\begin{equation}\label{12052025claiml'ego}
			E_j \in \mathfrak{H}^4_{\hat C,\sigma_1}(E_0) \text{ for all } 0 \leq j \leq k_0 . 
		\end{equation}
Using \eqref{zizz2} and the Sobolev embedding theorem, we obtain 
        \begin{equation}\label{240625form1}
            {\rm dist}_{\mathcal{H}}(\pa E_j, \pa E_{k}) \leq C(L_0) h\vert j-k \vert  \leq C(L_0)T_0 \text{ for all } j,\,k \leq k_0.
        \end{equation}
        Recall that each $E_j$ is an almost minimizer of the $\varphi$-perimeter for a constant $\omega$ independent of $h$. We apply Lemma \ref{propdadim} and get, for $T_0$ sufficiently small, the existence of functions $f_j: \pa E_0 \rightarrow \R$ such that 
		\begin{equation}
			\pa E_j = \{ x+ f_j(x)\nu_{E_0}(x)\colon x \in \pa E_0  \}.
		\end{equation}
		Moreover, Lemma \ref{propdadim} also guarantees that $ \| f_j \|_{C^{1,\gamma}(\pa E_0)}\leq \varepsilon$ for $\varepsilon>0$.
		Using the estimate \eqref{12052025form1}, specifically the bound $ \| \kappa_{E_j}^{\varphi} \|_{H^1(\pa E^j)}  \leq K_0$, we deduce that $\| f_j \|_{ H^3(\pa E_0)} \leq \hat C$. \\
        \textit{Claim 4:} There exists a constant $L_{lip}>0$ such that for all $0 \leq i,k \leq k_0$,
		\begin{equation}\label{12052025claiml'ego2}
			\| f_i- f_k \|_{L^\infty(\pa E_0)} \leq L_{lip} h\vert i-k \vert.
		\end{equation}
        This claim follows from \eqref{240625form1} and the observation that
        \begin{equation}
            \| f_{i}- f_{i-1} \|_{L^\infty(\pa E_0)} \leq C   \| \psi_{i}\|_{L^\infty(\pa E_{i-1})} = C{\rm dist}_{\mathcal{H}}(\pa E_i, \pa E_{i-1}).
        \end{equation}
        	Combining \eqref{12052025claiml'ego} and \eqref{12052025claiml'ego2}, and applying the Arzelà–Ascoli theorem, we conclude that there exists a subsequence $\{h_m\}_{m \in \mathbb{N}}$ such that 
            \begin{equation}
              f_{h_m}(t) \rightarrow f^\delta(t) \text{ in }  L^\infty(\pa E_0), \text{ for a.e. } t \in [0,T_0] \text{ as }m \rightarrow + \infty,
            \end{equation}
		 where
      \begin{equation}\label{estime-1}
			f^\delta \in \mathrm{Lip} ([0,T_0], L^\infty(\pa E_0)), \quad f^\delta \in L^\infty ([0,T_0], H^3(\pa E_0)).
		\end{equation}
		In what follows, we omit the dependence on the subsequence $m$. By Sobolev embedding, we further obtain that 
        $$f^\delta \in L^\infty([0,T_0], C^{2,\frac{1}{2}}(\pa E_0)).$$
       We then define the family $ \{ E^\delta_t \}_{t \in [0,T_0]} $ by
		\begin{equation}\label{27052025form2}
			E^\delta_t \Delta E_0 \subset I_{\sigma_1}(\pa E_0) \text{ and } \pa E^\delta_t := \{  x+  f^\delta(t,x) \nu_{E_0}(x) \colon \, x \in \pa E_0\}.
		\end{equation}
        	\textit{Claim 5:} $E^\delta_t$ is a distributional solution of equation \eqref{MAINEQsol}.
		
		We define the discrete normal velocity on $\pa E_{j_m}$ by 
		\begin{equation}
			V_{m,j} : \pa E_{j} \rightarrow \R, \quad V_{m,j} := \frac{\psi_{{j}+1}}{h_m}.
		\end{equation}
		Let $\Psi_{j} : \pa E_0 \rightarrow \pa E_{j}$ be defined by $$\Psi_{j}(x):= x+ f_{j} (x)\nu_{E_0}(x) .$$ We recall that 
		$$J_{\tau} \Psi_{j}(x)= \sqrt{(1+ f_{j}(x) \kappa_{E_0}(x))^2+\vert \pa_{\pa E_0} f_{j}(x) \vert^2 }, \quad x \in \pa E_0 .$$
		We also define $N_{j} : \pa E_0 \rightarrow \R^2$ as
		$$  N_{j}(x):= \frac{-\pa_{\pa E_0}  f_{j}(x)}{1+ \kappa_{E_0}(x) f_{j}(x)} \tau_{E_0}(x)+ \nu_{E_0}(x).$$
		Then, we observe that
		\begin{equation}
			\vert N_{j} \vert= \frac{ J_{\tau} \Psi_{j}}{1+ \kappa_{E_0} f_{j}}.
		\end{equation}
		\textit{Subclaim:} The following holds: 
		\begin{equation}\label{estimeoverl2}
			\lim_{m\to+\infty} \|  V_{m,j_m(t)} \circ \Psi_{j_m(t)} - \frac{f_{{j_m (t)}+1}-f_{j_m(t)}}{\vert N_{j_m(t)} \vert h_m} \|_{L^2(\pa E_0)} =0 \,\, \text{ for a.e. } t \in [0,T_0],
		\end{equation}
        where $j_m(t):= \lfloor \frac{t}{h_m} \rfloor$ .
		Using the estimate in \eqref{zizz2}, we obtain 
		$$\| \psi_{{j}+1} \circ \Psi_{j} \|_{C^1(\pa E_0)} \leq C h_m^{\frac{1}{2}} \text{ and } \| \psi_{{j}+1} \circ \Psi_{j} \|_{H^1(\pa E_0)} \leq C h_m. $$
		Combining these with the bound $ \| f_{j} \|_{C^{1,\gamma}}(\pa E_0) \leq \varepsilon$, we deduce that
		$$ \vert f_{{j}+1}(x)-f_{j}(x) \vert \leq C \vert \psi_{{j}+1} \circ \Psi_{j} (x)\vert \, \, \forall x \in \pa E_0 \text{ and }  \|f_{{j}+1}-f_{j} \|_{C^1(\pa E_0)} \leq C h_m^{\frac{1}{2}}.$$
		Let $G : \pa  E_0 \rightarrow \R$ be a function such that $ \| G \|_{C^1(\pa E_0)} \leq C h^{\gamma}$ for some $\gamma$. We define 
		$$\Psi_t : \pa E_0 \rightarrow \R^2 , \,\, \Psi_t(x):=x+t \nu_{E_0}(x),  $$ and we recall that $J_{\tau} \Psi_t= 1+ t \kappa_{E_0}$. Applying the coarea formula gives:
		\begin{equation}\label{13052025primacf}
			\begin{split}
				&\int_{\R^2} G \circ \pi_{\pa E_0}(x) \big(   \chi_{E_{{j}+1}}(x)- \chi_{E_{{j}}}(x) \big)\, dx \\
				&= \int_{\pa E_0} G(x) \int_{-\sigma_1}^{\sigma_1} \big(  \chi_{E_{{j}+1}}(\Psi_t(x))- \chi_{E_{{j}}}(\Psi_t(x)) \big) (1+ t \kappa_{E_0}(x)\, dt \, d \mathcal{H}^1_x\\
				&= \int_{\pa E_0} G(x) \int_{f_{j}(x)}^{f_{{j}+1}(x)} (1+ t \kappa_{E_0}(x)) \, dt \, d \mathcal{H}^1_x\\
				&= \int_{\pa E_0} G(x) (f_{{j}+1}(x)-f_{{j}}(x)) (1+ f_{j}(x)\kappa_{E_0}(x))\, d \mathcal{H}^1_x+ o(h_m^2)\\
				&= \int_{\pa E_0} G(x) J_{\tau} \Psi_{j}(x) \frac{f_{{j}+1}(x)- f_{j}(x)}{\vert N_{j} (x) \vert} \, d \mathcal{H}^1_x+ o(h_m^2).
			\end{split}
		\end{equation}
		We define
		$$ \Phi_{{j},t} : \pa E^{j} \rightarrow \R^2, \,\, \Phi_{{j},t}(x):= x+ t \nu_{E_{j}}(x) , $$
		and we recall that $ J_\tau \Phi_{{j},t}= 1+ t \kappa_{E_{j}}$.
		We compute the same integral differently:
		\begin{equation}\label{estimeover1}
			\begin{split}
				&\int_{\R^2} G \circ \pi_{\pa E_0}(x) \big(  \chi_{E_{{j}+1}}(x)- \chi_{E_{j}}(x) \big) \, d x\\
				&= \int_{\pa E_{j}} \int_{-\delta}^{\delta} G \circ \pi_{\pa E_0}(\Phi_{{j},t}(x)) \big(  \chi_{E_{{j}+1}}(\Phi_{{j},t}(x))- \chi_{E_{j}j}(\Phi_{{j},t}(x)) \big) (1+ t \kappa_{E_{j}}(x))\, dt \, d \mathcal{H}^1_x \\
				&= \int_{\pa E_{j}} \int_{0}^{\psi_{{j}+1}(x)} G \circ \pi_{\pa E_0} ( \Psi_{{j},t}(x)) (1+ t \kappa_{E_{j}}(x)) \, d t \,d \mathcal{H}^1_x\\
				& = \int_{\pa E_{j}} \int_{0}^{\psi_{{j}+1}(x)} \big( G \circ \pi_{\pa E_0} ( \Psi_{{j},t}(x)) -G \circ \pi_{\pa E_0}(x)+ G \circ \pi_{\pa E_0}(x)\big)(1+ t \kappa_{E_{j}}(x)) \, d t \,d \mathcal{H}^1_x \\
				& = \int_{\pa E_{j}} \psi_{{j}+1}(x) G \circ \pi_{\pa E_0}(x)  \,d \mathcal{H}^1_x + o(h_m^2)\\
				& = \int_{\pa E_0} \psi_{{j}+1} \circ \Psi_j (x) G (x) J_{\tau} \Psi_{j}(x)\, d\mathcal{H}_x^1+ o(h_m^2).
			\end{split}
		\end{equation}
		Comparing  \eqref{13052025primacf} and \eqref{estimeover1} we find that for all $G: \pa E_0 \rightarrow \R$ with $ \|G \|_{C^1(\pa E_0)} \leq Ch_m^{\gamma}$, it holds
		\begin{equation}\label{estimeoverl3}
			\int_{\pa E_0}  G(x) J_{\tau} \Psi_{{j}}(x) \bigg[ \psi_{{j}+1}\circ \Psi_{j}(x)-\frac{f_{{j}+1}(x)- f_{j}(x)}{\vert N_{j} (x) \vert}  \bigg]\, d \mathcal{H}^1_x=              o(h_m^2).
		\end{equation}
		We define
		\begin{equation}\label{estimeoverl4}
			G(x):= \frac{1}{J_{\tau} \Psi_{{j_m}}(x)}\bigg[ \psi_{{j_m}+1}\circ \Psi_{j_m}(x)-\frac{f_{{j_m}+1}(x)- f_{j_m}(x)}{\vert N_{j_m} (x) \vert}  \bigg].
		\end{equation}  
		A straightforward computation shows that  $ \| G \|_{C^1(\pa E_0) } \leq C h_m^\gamma$ for some $\gamma \in (0,1)$. Plugging \eqref{estimeoverl4} into \eqref{estimeoverl3} yields \eqref{estimeoverl2}. 
		
		We now return to the claim 5. Up to now, we have established that
		\begin{equation}\label{27052025form1}
			\|f_{j} \|_{H^3(\pa E_0)} \leq C_0, \quad \| f_{j} \|_{L^\infty(\pa E_0)} \leq \sigma_1 ,\quad  \bigg\| \frac{f_{{j}+1}-f_{j}}{\vert N_{j} \vert h_m} \bigg\|_{L^\infty(\pa E_0)} \leq C \, \text{ for all } {j} \colon \,{j} h_m \leq T_0 .
		\end{equation}
		Thus, combining \eqref{27052025form1}, \eqref{estimeoverl2}, and \eqref{estime-1}, we conclude that
		\begin{equation}\label{19052025fonzi}
			\exists \, L^2(\pa E_0)\text{-}\lim_{m \rightarrow + \infty} V_{m,j_m(t)} \circ \Psi_{j_m(t)}(\cdot)= \frac{\pa_t f^\delta(t,\cdot)}{\vert N(t,\cdot)\vert },\,\,  \text{ for a.e. } t \in [0,T_0],
		\end{equation}
		where $\vert N(t,x) \vert= \frac{J_\tau \Psi_t(x)}{1+ f^\delta(t,x) \kappa_{E_0}(x)}$ and $ \Psi_t(x):= x+ f^\delta(t,x)\nu_{E_0}(x)$ for $x \in \pa E_0$. Let $l \in C^0_c(\R^2)$. Multiplying the Euler–Lagrange equation \eqref{eqeulerobvera} by $l$, we obtain
		\begin{equation}\label{dm}
			\begin{split}
				\int_{\pa E_{j_m}} \frac{\psi_{{j_m}+1}(x)}{h_m} & l(x) \, d\mathcal{H}^1_x+ \int_{\pa E_{j_m}} 
				\frac{\psi_{{j_m}+1}^2(x)}{2 h_m}\kappa_{E_{j_m}}(x) l(x)  \, d \mathcal{H}^1_x\\
				&= \int_{\pa E_{j_m}}  \big( g(\nu_{E_{j_m}}(x)) \pa_{\pa E_{j_m}}^2\psi_{{j_m}+1}(x)-\kappa_{E_{{j_m}}}^\varphi(x)-R_0(x)+\lambda_{E_{j}}  \big) l(x)\, d \mathcal{H}^1_x,
			\end{split}
		\end{equation}
        where 
        $$\lambda_{E_{j_m}}= \dashint_{\pa E_{j_m}} \kappa^\varphi_{E_{j_m}} \, d \mathcal{H}^1 .$$
		We observe that
		\begin{equation}\label{dm1}
			\begin{split}
				&\lim_{m \rightarrow +\infty} \int_{\pa E_{j_m}} 
				\frac{\psi_{{j_m}+1}^2(x)}{2 h_m}\kappa_{E_{j_m}}(x) l(x)  \, d \mathcal{H}^1_x \\
				& \qquad\leq \lim_{m \rightarrow +\infty} \| l \|_{L^\infty(\R^2)} \frac{ \| \psi_{{j_m}+1} \|_{L^2(\pa E_{j_m})} }{2h_m} \| \psi_{{j_m}+1} \|_{L^\infty(\pa E_{j_m})} \| \kappa_{E_{j_m}}\|_{L^2(\pa E_{j_m})} =0, 
			\end{split}
		\end{equation}
		where we have used \eqref{zizz2}. From the previous claim, we also have $$ \| \psi_{j_m +1} \|_{H^2(\pa E_{j_m})} \leq C h_m^\gamma.$$ Recalling the very definition of
		$R_0$ from \eqref{interpol1}, we find that
        \begin{equation}\label{dm2}
			\| R_0 \|_{L^2(\pa E_{j_m})} \leq C h_m^\gamma.
		\end{equation}
		Therefore, passing to the limit as $m \rightarrow + \infty$ in the equation \eqref{dm}, and using \eqref{dm1} and \eqref{dm2}, we conclude that $E^\delta_t $ is a distributional solution of \eqref{MAINEQsol} in $[0,T_0]$.\\
        \textit{Claim 6:} $ \{E_t^\delta\}_{t \in [0,T_0]} $ is a classical solution of \eqref{MAINEQsol}. 

By a classical solution of \eqref{MAINEQsol} starting from $E_0$, we mean a function $$f \in L^\infty([0,T_0], H^3(\pa E_0)) \cap \mathrm{Lip}([0,T_0], L^\infty(\pa E_0))$$ that satisfies
\begin{equation} \label{esk}
        \left\{
		\begin{aligned}
			& \partial_t f = g(\nu_{E_0})\pa^2_{ E_0} f +  \langle A(x,f,\nabla_{\pa E_0} f), \nabla_{\pa E_0}^2 f \rangle +J(x,f,\nabla_{\pa E_0} f)+\kappa_{E_0}^\varphi \text{ on } \pa E_0,\\
			& f(0,\cdot)=0, 
		\end{aligned}
		\right.
\end{equation}
where $A$ is a smooth tensor satisfying $A(\cdot,0,0)=0$,  and   $J$ is a smooth function (see \cite{Mantegazza2011}). By applying Gr\"onwall's lemma, one can show that the strong solution to \eqref{esk} with zero initial data is unique. Therefore, we conclude that the family $ \{ E^\delta_t \}_{t \in [0,T_0]}$ parametrized by the diffeomorphisms $$\Phi_t(x)= x+f^\delta(t,x)\nu_{E_0}(x), \quad x \in \pa E_0,$$ constitutes a strong solution of the volume preserving anisotropic mean curvature flow.
\end{proof}
\section{Convergence to the global solution}
	We begin by recalling the definition of the uniform ball condition.
	\begin{definition}
		A set $E\subset \R^2$  is said to satisfy the uniform ball condition (UBC) with radius $r>0$  if, for every point $x \in \pa E$, there exist points $x_{+}$ and $x_{-}$ such that
		\begin{equation}
			B_r(x_{+}) \subset \R^2 \setminus E, \quad B_r(x_{-}) \subset E \text{ and } x \in \pa B_r(x_{+}) \cap \pa B_r(x_{-}).
		\end{equation}
	\end{definition}
    \begin{remark}
		\label{rem:thm2}
		We can formulate a quantitative version of Theorem \ref{thmausilmain} as follows. Let $E_0\Subset \R^2$ be a connected open set of class $C^4$ that satisfies the UBC with radius $2r_0$.
		Let $ \psi_1$, denote the height function from Theorem \ref{thmausilmain}, and set
        $\xi_1= \psi_1 + \frac{\psi_1^2}{2}\kappa_{E_0}$. 
        Assume that
        \begin{equation}\label{01072025form1}
         \|\pa_{\pa E_0}\xi_1 \|_{L^2(\pa E_0)} \leq L_0 h \quad \text{and} \quad \|\pa^2_{\pa E_0}\psi_1 \|_{L^2(E_0)} \leq  L_0\sqrt{h} . 
        \end{equation}
         Then, there exist constants $K_0 = K_0(r_0,L_0)$ and $\hat \delta=\hat \delta(r_0,K_0)$ such that if
		\[
		\|\pa_{\pa E_0}\kappa_{ E_0}^\varphi  \|_{L^2(\pa E_0)} \leq K_0  \quad \text{and} \quad \| \pa^2_{\pa E_0}  \kappa_{\pa E_0}^\varphi \|_{L^2(\pa E_0)} \leq  \frac{K_0}{\sqrt{h}},
		\]
		then  the approximate constrained flat flow $\{E^{h,\delta}_t\}_{t\geq 0}$, with $\delta \leq \hat \delta$, also satisfies the UBC with radius $r_0$, and moreover,
		\[
		\|\pa_{\pa E^{h,\delta}_t}\kappa_{E^{h,\delta}_t }^\varphi \|_{L^2(\pa E^{h,\delta}_t)} \leq  K_0  \quad \text{and} \quad \| \pa^2_{\pa E^{h,\delta}_t} \kappa^\varphi_{E^{h,\delta}_t} \|_{L^2(\pa E^{h,\delta}_t)} \leq \frac{K_0}{\sqrt{h}}
		\]
		for all $t \in [0,T_0]$, where $T_0 = T_0(r_0,K_0)$.  
	\end{remark}	
	\begin{remark}
		\label{rem:uniqueness-strong}
		The arguments in the proof of Theorem \ref{thmausilmain} imply that if an approximate constrained flat flow $\{ E^{h, \delta}_t   \}_{t \geq 0}$ starting from $E_0$  satisfies 
		\[
		\|\pa_{\pa E^{h,\delta}_t}\kappa_{E^{h,\delta}_t }^\varphi \|_{L^2(\pa E^{h,\delta}_t)} \leq  K_0  \quad \text{and} \quad \| \pa^2_{\pa E^{h,\delta}_t} \kappa^\varphi_{E^{h,\delta}_t} \|_{L^2(\pa E^{h,\delta}_t)} \leq  \frac{K_0}{\sqrt{h}} \quad \text{ for all } t \in [0,T],
		\]
		 then the limiting flat flow coincides with the classical solution on the time interval $[0,T]$.  
	\end{remark}
	We recall that the classical solution to \eqref{MAINEQsol} with initial datum $E_0$ exists on the interval $[0,T_e)$, where $T_e$ denotes the maximal existence time. In the next theorem, we show that for every $T< T_e$, there exists $\delta(T)$ such that the approximate constrained flat flow with initial datum $E_0$ converges to the classical solution of \eqref{MAINEQsol} on $[0,T]$ as $h \rightarrow 0^+$. \begin{theorem}\label{03thmfinaleGLOB}
		Let $\{E_t\}_{t \in [0,T_e)}$ be a classical solution of \eqref{MAINEQsol} with initial datum $E_0$. Then, for every $T < T_e$, there exists $\delta(T)$  such that, for all $\delta \in (0,\delta(T)]$, the approximate constrained flat flow  $E^\beta_t$, starting from $E_0$ coincide with $E_t$ on the interval $[0,T]$.    
	\end{theorem}
\begin{proof}
	Let $\{E_t\}_{t \in [0,T_e)}$ be the classical solution of \eqref{MAINEQsol} and fix $ T < T_e$. Since the classical solution is regular on $[0,T]$, there exist constants $ K_2,\sigma_2$ such that
$$
		E_t \in \mathfrak{H}^3_{K_2,\sigma_2}(E_0) \text{ for all } t \in [0,T].
$$
This condition implies the existence of $r_0>0$ such that $E_t $ satisfies the uniform ball condition (UBC) with radius $r_0$ for all  $t \in [0,T]$. Let $\hat \delta, T_0$ be the constants given in Theorem \ref{thmausilmain} and Remark \ref{rem:thm2}, and fix $\delta\leq  \hat \delta$. 
Let $k_0 \in \N$ be such that $T_0 \in [hk_0, h(k_0+1))$, and let $\{E^{h,\delta}_{hk}\}_{k \in \N}$ be an approximate constrained flat flow starting from $E_0$. As established in Theorem \ref{thmausilmain} and Remark \ref{rem:thm2}, we have:
\begin{equation}
	\begin{split}
		&\pa E^{h,\delta}_{h}= \{  x +\psi_1(x)\nu_{E_0}(x)\colon x \in \pa E_0 \},\\
		&\| \pa_{\pa E_0}\psi_1 \|_{L^2(\pa E^0)} \leq \frac{L_0}{2} h, \quad\| \pa_{\pa E_0}\xi_1 \|_{L^2(\pa E^0)} \leq L_0 h, \quad \| \pa^2_{\pa E_0} \psi_1 \|_{L^2(\pa E_0)} \leq L_0 \sqrt{h}, 
		\\
		&  \|\pa_{\pa E^{h,\delta}_t}\kappa_{E^{h,\delta}_t }^\varphi \|_{L^2(\pa E^{h,\delta}_t)} \leq  K_0 , \quad \| \pa^2_{\pa E^{h,\delta}_t} \kappa^\varphi_{E^{h,\delta}_t} \|_{L^2(\pa E^{h,\delta}_t)} \leq \frac{K_0}{\sqrt{h}} \,\, \forall t \in [0,T_0],
	\end{split} 
\end{equation}
where $L_0$ is as in formula \eqref{01072025form1} and such that
$$ \|\pa_{\pa E_t}\kappa_{E_t }^\varphi \|_{L^2(\pa E_t)} \leq  \frac{L_0}{2} \quad \forall t \in [0,T], $$
and $K_0$ is defined as in \eqref{12052025form-1}. We adopt the notation $E^h_t:= E^{h,\delta}_t$, $E_k:= E^{h,\delta}_{hk}$, and recall \eqref{agligderigligb} for the definition of $ E^{h,\delta}_t$.
The conclusion of the theorem follows from the next claim, together with Remarks \ref{rem:thm2} and \ref{rem:uniqueness-strong}.
\\
\textit{Claim:} For every $t \in [0,T]$,
\begin{equation}\label{19052025ancora1}
	\|\pa_{\pa E^h_t }\kappa_{E^h_t}^{\varphi} \|_{L^2(\pa E^h_t )} \leq K_0, \quad  \| \pa_{\pa E^h_t}^2 \kappa_{E^h_t}^\varphi \|_{L^2(\pa E^h_t )} \leq \frac{K_0}{ \sqrt{h}}.
\end{equation}
By Theorem \ref{thmausilmain} and Remark \ref{rem:thm2}, estimate \eqref{19052025ancora1} holds for all $t \in [0,T_0]$. We define
\begin{equation}\label{19052025ttilde}
	\tilde{t}: = \sup \{   s \in [T_0, T]: \text{ formula \eqref{19052025ancora1} is true for all }t \in[0,s]  \} .
\end{equation}
We will show that \eqref{19052025ancora1} remains valid for all $t \leq \tilde{t}+ \frac{T_0}{2}$, which implies the claim. To this end, let $\tilde{k} \in \N$ be such that $ \tilde{t}-\frac{T_0}{2} \in [h \tilde{k}, (\tilde{k}+1)h  )$. We have that $E_{\tilde{k}}$ satisfies \eqref{19052025ancora1}. Hence we apply Theorems \ref{thmausilmain} with $E_0= E_{\tilde{k}}$ to obtain that there exist $k_1 \in \N$ and $c>0$ (we recall that $c=c( K_0)$) such that $0<c \leq   h k_1=T_1 \leq  T_0$, and for all $ k \in \{ \tilde{k}, \dots, \tilde{k}+k_1 \}$
\begin{equation}
	\pa E_k= \{  x+ \psi_k(x)\nu_{E_{k-1}}(x)\colon x \in \pa E_{k-1}\}.
\end{equation}
	Using formula \eqref{zizz2}, we obtain 
\begin{equation}
	\| \pa_{\pa E_{k-1}}\xi_k \|^2_{L^2(\pa E_{k-1})}+ h \sum_{j=\tilde{k}}^{\tilde{k}+k_1} \| \pa^2_{\pa E_{k-1}} \psi_k \|^2_{L^2(\pa E_{k-1})} \leq C h^2,  
\end{equation}
for some constant $C$. Since  $0 <c \leq hk_1=T_1$, there exists $\hat{k} \in \{  \tilde{k}, \dots, \tilde{k}+ \lfloor \frac{k_1}{4} \rfloor \}$ such that
\begin{equation}\label{19052025pom1}
	\| \pa_{\pa E_{\hat{k}-1}}\xi_{\hat{k}} \|_{L^2(\pa E_{\hat{k}-1})}+  \|   \pa^2_{\pa E_{\hat{k}-1} } \psi_{\hat{k}} \|_{L^2(\pa E_{\hat{k}-1})} \leq C h.
\end{equation}
From this and using the very definition of $\tilde{t}$, see \eqref{19052025ttilde}, we have
$$h\hat{k} \leq h (\tilde{k}+ \lfloor \frac{k_1}{4} \rfloor ) \leq \tilde{t}.$$
Note that in the minimizing movements scheme, each set $E_j$  is of class $C^4$, since it solves the Euler–Lagrange equation \eqref{eqeulerobvera}. 
Moreover, $E_{\hat{k}}$ is uniformly $C^{2,\frac{1}{2}}$-regular. Let $t_h= \hat{k}h$ and define
\begin{equation}
	v^h(t_h,x):= \frac{\psi_{\hat{k}}(x) }{h}.
\end{equation}
By \eqref{19052025pom1} and the Sobolev embedding theorem, we obtain
\begin{equation}\label{19052025pom2}
	\| v^h(t_h,\cdot)\|_{C^{1,\frac{1}{2}}(\pa E_{\hat{k}-1})} \leq C.
\end{equation}
Since $t_h = \hat{k}h \in [\tilde{t}- \frac{T_0}{2}, \tilde{t}]$, by passing to a subsequence if necessary, we may assume
\begin{equation}\label{19052025pom3}
	\exists \lim_{h \rightarrow 0^+} t_h= \bar{t} .
\end{equation}
From \eqref{19052025pom2} and \eqref{19052025pom3}, it follows that
\begin{equation}
	v^h(t_h,\cdot) \rightarrow v(\bar t, \cdot) \, \text{in } C^{1,\frac{1}{2}} \text{ as } h \rightarrow 0^+, \, \| v(\bar t, \cdot )\|_{C^{1,\frac{1}{2}}(\pa E^\delta_{\bar t})} \leq C.
\end{equation}
Therefore,
\begin{equation}
	\lim_{h \rightarrow 0^+} \| v^h(t_h,\cdot)\|_{L^2(\pa E_{\hat{k}-1})}= \| v(\bar t, \cdot )\|_{L^2(\pa E^\delta_{\bar t}   )}.
\end{equation}
Since we assumed that \eqref{19052025ancora1} holds for all $t \leq \tilde{t}$,
and since $\bar t \leq \tilde{t}$, Remark \ref{rem:uniqueness-strong} implies that the flat flow agrees with the classical solution up to time $\tilde{t}$.
Using \eqref{19052025fonzi} and \eqref{estimeoverl2} (with $E(\tilde{t}- \frac{T_0}{2})$ in place of $E_0$), we find that $v(\bar t, \cdot)$ coincides with the normal velocity $V_{\bar t}$ of the classical solution $\{ E_t\}_{t\geq 0}$, and
\begin{equation}
	  \| \pa_{\pa E_{\bar t}} v(\bar t,\cdot) \|_{L^2(\pa E^\delta_{\bar t} )} = \| \pa_{\pa E_{\bar t} } V_{\bar t} \|_{L^2(\pa E_{\bar t})} = \|\pa_{\pa E_t}\kappa_{E_t }^\varphi \|_{L^2(\pa E_t)} \leq  \frac{L_0}{2}.   
\end{equation}
Thus,
\begin{equation}
	\| \pa_{E_k}\psi_{k}\|_{L^2(\pa E_{k-1})} \leq \frac{L_0}{2}h,
\end{equation}
which implies
$$  \| \pa_{E_{k-1}}\xi_{k}\|_{L^2(\pa E_{k-1})} \leq L_0h. $$
Using \eqref{19052025pom1}, we also obtain
\begin{equation}
	\| \pa^2_{\pa E_{k-1}}\psi_{k}\|_{L^2(\pa E_{k-1})} \leq  C h \leq K_{0} \sqrt{h}
\end{equation}
for $h$ small enough. Now, using the Euler equation, the above estimates, and Remark \ref{rem:thm2}, we conclude that
$$ \|\pa_{\pa E^h_t }\kappa_{E^h_t}^{\varphi} \|_{L^2(\pa E^h_t )} \leq K_0, \quad  \| \pa_{\pa E^h_t}^2 \kappa_{E^h_t}^\varphi \|_{L^2(\pa E^h_t )} \leq \frac{K_0}{ \sqrt{h}} \, \, t \in [\tilde{t}- \frac{T_0}{2},\tilde{t}+ \frac{T_0}{2}].  $$
Repeating this argument a finite number of times yields the claim.
\end{proof}

\appendix
\section{} 
    In this appendix, we prove the estimates for the quantities
    \begin{equation}
        \int_{\pa E} R_0 \pa^4_\tau \psi, \quad  \| \pa_\tau R_0 \|_{L^2(\pa E)}
    \end{equation}
that were used in the proof of Theorem~\ref{mainthm}.
    \begin{lemma}\label{lemmaapendicead}
   Let the notation and assumptions of Theorem \ref{mainthm} be in force. If we assume \eqref{tesistep2}, then formula \eqref{veform4} holds, that is,
   \begin{equation}
   \label{formutileoggi.0}
     \int_{\pa E} R_0 \pa_{\tau}^4 \psi \leq C(K)+(\frac{c_g}{9}+3\varepsilon) \int_{\pa E} \vert \pa_{\tau}^3 \psi \vert^2 + \frac{C(K)}{\sqrt{h}} \int_{\pa E} \vert \pa_{\tau}^2 \psi \vert^2.
   \end{equation}
   Furthermore, if we assume \eqref{tesistep3}, then the following estimate holds:
   \begin{equation}
    \label{formutileoggi.1}
    \|\pa_{\tau} R_0 \|_{L^2(\pa E)} \leq C(K) . 
\end{equation}
\end{lemma}
 \begin{proof}
We divide the proof into two steps.\\
\textit{Step 1:} In this step, we prove \eqref{formutileoggi.0}.

Recall that $R_0$ is defined in \eqref{eqeulroresto} and consists of several terms. We begin by analyzing the term
$$\int_{\pa E} a(\nu_E,\psi \kappa_E,\pa_\tau\psi) \pa_\tau^2 \psi \pa_{\tau}^4 \psi.$$ 
We integrate by parts and apply Young's inequality to obtain
\begin{equation}\label{stimarest1R_0}
      \begin{split}
          \int_{\pa E} a(\nu_E,\psi \kappa_E,\pa_\tau\psi) \pa_\tau^2 \psi \pa_\tau^4 \psi&= -\int_{\pa E} a(\nu_E,\psi \kappa_E,\pa_\tau\psi) \vert \pa_\tau^3 \psi \vert^2 - \int_{\pa E} \pa_\tau a(\nu_E,\psi \kappa_E, \pa_\tau \psi) \pa_\tau^3 \psi \pa_\tau^2 \psi \\
          & \leq \frac{c_g}{9} \int_{\pa E} \vert \pa_\tau^3 \psi \vert^2+ \varepsilon \int_{\pa E} \vert \pa_\tau^3 \psi \vert^2 + \frac{1}{\varepsilon} \int_{\pa E} \vert \pa_\tau a (\nu_E,\psi \kappa_E,\pa_\tau \psi)  \vert^2 \vert \pa_\tau^2 \psi \vert^2\\
          & \leq (\frac{c_g}{9}+\varepsilon) \int_{\pa E} \vert \pa_\tau^3 \psi \vert^2 + \frac{C(K)}{\varepsilon\sqrt{h}} \int_{\pa E} \vert \pa_\tau^2 \psi \vert^2,
      \end{split}
  \end{equation}
where in the first inequality we used Young's inequality and the estimate \begin{equation}\label{eq:stima.tecn.}
      \|a(\nu_E,\psi \kappa_E,\pa_{\tau} \psi) \|_{L^\infty(\pa E)} \leq \frac{c_g}{9}.
  \end{equation}
For the second inequality in \eqref{stimarest1R_0}, we used the identity
  \begin{equation}\label{difffa}
   \pa_\tau a(\nu_E,\psi \kappa_E, \pa_\tau \psi)= \hat Z(\kappa_E, \psi \kappa_E, \pa_{\tau} \psi)+ \hat A(\nu_E,\psi \kappa_E, \pa_\tau \psi) \pa_\tau(\psi \kappa_E)+ \hat B(\nu_E,\psi \kappa_E, \pa_\tau \psi) \pa_\tau^2 \psi, 
  \end{equation}
   where $ \hat Z, \hat A,  \hat B \in C^\infty$ and we used the estimates \eqref{eq:assumption.curvature} and \eqref{infitolapsi}.\\
To proceed, we analyze the integral $$\int_{\pa E}b(\nu_E,\psi \kappa_E,\pa_\tau\psi)\pa_{\tau} (\psi \kappa_E) \pa_\tau^3 \psi.$$ 
Integrating by parts, we obtain 
\begin{equation}\label{formulaz1}
    \begin{split}
        \int_{\pa E} b(\nu_E,\psi \kappa_E,\pa_\tau \psi) \pa_\tau(\psi \kappa_{E}) \pa_\tau^4 \psi  =&-  
        \int_{\pa E} b(\nu_E,\psi \kappa_E,\pa_\tau \psi) \pa_\tau^2 (\psi \kappa_E ) \pa_\tau^3 \psi\\
        & \quad
        -
        \int_{\pa E}  \pa_\tau b(\nu_E,\psi \kappa_E,\pa_\tau \psi) \pa_\tau(\psi \kappa_E) \pa_\tau^3 \psi.
    \end{split}
\end{equation}
To estimate the first term on the right-hand side, we apply the Leibniz rule and Young’s inequality:
\begin{equation}\label{formulaz2}
    \begin{split}
      \int_{\pa E} b(\nu_E,\psi \kappa_E,\pa_\tau \psi) \pa_\tau^2 (\psi \kappa_E ) \pa_\tau^3 \psi 
      &= \int_{\pa E} b(\nu_E,\psi\kappa_E,\pa_\tau \psi)  \big[ \pa_\tau^2 \psi \kappa_E +2 \pa_\tau \kappa_E\pa_\tau \psi+ \psi \pa_\tau^2\kappa_E\big] \pa_\tau^3 \psi \\
      &\leq \varepsilon \int_{\pa E} \vert \pa_\tau^3 \psi \vert^2 +\frac{C(K)}{\varepsilon} \int_{\pa E} \vert \pa_\tau^2 \psi \vert^2
      \\
      &\;\;+ \frac{C(K)}{\varepsilon} \int_{\pa E} \vert \pa_\tau \kappa_E \vert^2+ \frac{C(K)}{\varepsilon} \| \psi \|_{L^\infty(\pa E)}^2 \int_{\pa E} \vert \pa_\tau^2 \kappa_E \vert^2  \\
      &\leq \varepsilon \int_{\pa E} \vert \pa_{\tau}^3  \psi \vert^2 +\frac{C(K)}{\varepsilon} \int_{\pa E} \vert \pa_{\tau}^2 \psi \vert^2 + \frac{C(K)}{\varepsilon}
      \\
      &\leq \varepsilon \int_{\pa E} \vert \pa_{\tau}^3  \psi \vert^2+ \frac{C(K)}{\varepsilon},
    \end{split}
\end{equation}
where we used \eqref{eq:stima.tecn.} and \eqref{13112024sera1}.
To handle the second term in \eqref{formulaz1}, we first observe that
 \begin{equation}\label{difffb}
     \pa_{\tau} b(\nu_E,\psi \kappa_E,\pa_{\tau} \psi)=\hat{Y}(\kappa_E,\psi \kappa_E, \pa_{\tau} \psi) +\hat C(\nu_E,\psi\kappa_E, \pa_{\tau} \psi)\pa_{\tau}(\psi \kappa_E)+\pa_{\tau}^2  \psi \hat D(\nu_E,\psi \kappa_E,\pa_{\tau} \psi),
\end{equation}
where $ \hat{Y},\,\hat C,\, \hat D \in C^{\infty}$.
Moreover, we have the estimate
\begin{equation} \label{difffd}
    \| \pa_{\tau} (\psi \kappa_E) \|_{L^\infty(\pa E)} = \| \kappa_{E}\pa_{\tau} \psi+ \psi \pa_{\tau} \kappa_E \|_{L^\infty(\pa E)} \leq C(K)+ \sqrt{h} \frac{C(K)}{\sqrt{h}} \leq C(K).
\end{equation}
Using \eqref{difffb}, \eqref{difffd}, and \eqref{tesistep2}, we estimate:
\begin{equation}\label{formulaz3}
    \begin{split}
      | \int_{\pa E} \pa_{\tau} b(\nu_e,\psi \kappa_E,\pa_{\tau} \psi) \pa_{\tau}(\psi \kappa_E) \pa_{\tau}^3 \psi 
         |&\leq C(K)\int_{\pa E}|\pa_\tau b( \nu_E,\psi \kappa_E,\pa_{\tau} \psi  ) |\, |\pa_{\tau}^3 \psi|
        \\
         &\leq\varepsilon \int_{\pa E} \vert \pa_{\tau}^3 \psi \vert^2+ \frac{C(K)}{\varepsilon} \int_{\pa E} \big(  \vert \pa_{\tau}^2 \psi \vert^2 +1\big)\\
        &\leq \varepsilon \int_{\pa E}     \vert \pa_{\tau}^3 \psi \vert^2+ \frac{C(K)}{\varepsilon}.
        \end{split}
\end{equation}
Combining \eqref{formulaz1}, \eqref{formulaz2}, and \eqref{formulaz3}, we conclude that
\begin{equation}\label{stimarest2R_0}
    \int_{\pa E} b(\nu_E,\psi \kappa_E,\pa_{\tau} \psi) \pa_{\tau}(\psi \kappa_{E}) \pa_{\tau}^2 \psi \leq \varepsilon \int_{\pa E} \vert \pa_{\tau}^3   \psi \vert^2+ \frac{C(K)}{\varepsilon}.
\end{equation}
The final term we need to analyze is $$ \int_{\pa E} c(\nu_E,\psi,\pa_{\tau} \psi, \kappa_E) \pa_{\tau}^4 \psi.$$
To this end, we compute:
\begin{equation}\label{stimarest3R_0}
    \begin{split}
        \int_{\pa E} c(\nu_E,\psi,\pa_{\tau} \psi, \kappa_E) \pa_{\tau}^4 \psi 
        &=-\int_{\pa E} \pa_{\tau}^3 \psi  \big[ \hat{X}( \psi,\pa_{\tau} \psi, \kappa_E)+\hat E(\nu_E,\psi,\pa_{\tau}\psi, \kappa_E)\pa_{\tau}\psi \big]\\
        &\;\; +\int_{\pa E}  \pa_{\tau}^3 \psi  \big[ \hat F(\nu_E,\psi,\pa_{\tau}\psi,\kappa_E)\pa_{\tau}^2\psi+ \hat G(\nu_E,\psi,\pa_{\tau} \psi,\kappa_E)\pa_{\tau}\kappa_E \big]\\
         &\leq  \varepsilon \int_{\pa E} \vert \pa_{\tau}^3 \psi \vert^2+ \frac{C(K)}{\varepsilon} \int_{\pa E} 1+\vert \pa_{\tau} \psi \vert^2+ \vert \pa_{\tau} \psi \vert^2 + \vert \pa_{\tau} \kappa_E \vert^2 \\
    &\leq \varepsilon \int_{\pa E} \vert \pa_{\tau}^3 \psi \vert^2+ \frac{C(K)}{\varepsilon},
    \end{split}
\end{equation}
where
\begin{equation}\label{difffc}
\begin{split}
   & \pa_{\tau}c(\nu_E,\psi,\pa_{\tau} \psi, \kappa_E)=\hat{X}( \psi, \pa_\tau \psi, \kappa_E)\\
    &\quad\qquad\qquad+\hat E(\nu_E,\psi,\pa_{\tau}\psi, \kappa_E)\pa_{\tau} \psi + \hat F(\nu_E,\psi,\pa_{\tau} \psi,\kappa_E)\pa_{\tau}^2\psi+ \hat G(\nu_E,\psi,\pa_{\tau} \psi,\kappa_E)\pa_{\tau}\kappa_E 
\end{split}
\end{equation}
with $\hat{X},\hat E,\hat F, \hat G \in C^{\infty}$. In the last inequality of \eqref{stimarest3R_0} we used \eqref{tesistep2}. By the very definition of $R_0$ (see \eqref{interpol1}), and using \eqref{stimarest1R_0}, \eqref{eq:stima.tecn.}, \eqref{formulaz1}, \eqref{formulaz2}, \eqref{formulaz3} and \eqref{stimarest3R_0}, we obtain \eqref{formutileoggi.0}.
\\
\textit{Step 2: }In this step, we prove \eqref{formutileoggi.1}.

Using \eqref{difffa}, \eqref{difffb}, and \eqref{difffc}, we obtain:
\begin{equation*}
\begin{split}
    \pa_{\tau} R_0& = \pa_{\tau}^3 \psi a(\cdot,\psi\kappa_E,\pa_{\tau} \psi) \\
    &\qquad+ \pa_{\tau}^2 \psi \big[ \hat{Z}(\cdot, \psi\kappa_E, \pa_\tau \psi)+\hat A(\cdot,\psi \kappa_E, \pa_{\tau} \psi) \pa_{\tau}(\psi \kappa_E)+\hat B(\cdot,\psi \kappa_E, \pa_{\tau} \psi) \pa_{\tau}^2 \psi   
 \big] \\
 & \qquad +\pa_{\tau}^2(\psi \kappa_E) b(\cdot,\psi \kappa_E,\pa_{\tau} \psi)\\
 & \qquad+ \big[\hat{Y}(\cdot,\psi \kappa_E, \pa_{\tau} \psi) +\hat C(\cdot,\psi\kappa_E, \pa_{\tau} \psi)\pa_{\tau}(\psi \kappa_E)+\pa_{\tau}^2  \psi \hat D(\cdot,\psi \kappa_E,\pa_{\tau} \psi) \big] \pa_{\tau}(\psi \kappa_E)\\
 & \qquad+\hat{X}( \psi, \pa_\tau \psi, \kappa_E)+ \hat E(\nu_E,\psi,\pa_{\tau}\psi, \kappa_E)\pa_{\tau} \psi \\
 & \qquad\qquad\qquad+ \hat F(\nu_E\psi,\pa_{\tau}\psi,\kappa_E)\pa_{\tau}^2\psi+ \hat G(\nu_E,\psi,\pa_{\tau} \psi,\kappa_E)\pa_{\tau}\kappa_E.
\end{split}
\end{equation*}
Therefore, combining this with \eqref{tesistep2}, \eqref{tesistep3}, and using the Sobolev embedding, we deduce that
\begin{equation*} 
    \|\pa_{\tau} R_0 \|_{L^2(\pa E)} \leq C(K) + \| \pa_{\tau}^2 \kappa_{E} \|_{L^2(\pa E)} \| \psi \|_{L^\infty(\pa E)} \leq C(K) +\frac{C(K)}{\sqrt{h}} \sqrt{h}\leq C(K).
\end{equation*}
 \end{proof}

\section*{Acknowledgments} 
A.\ Kubin was supported by the Academy of Finland grant 314227. 
D.\ A.\ La Manna is a member of GNAMPA and has received funding from INdAM under the INdAM--GNAMPA Project 2025 \textit{Local and nonlocal equations with lower order terms} (grant agreement No.\ CUP\_E53\-240\-019\-500\-01).
E.\ Pasqualetto was supported by the Research Council of Finland grant 362898.

\end{document}